\documentclass[12pt]{article}
\usepackage{amsmath}
\usepackage{latexsym}
\usepackage{amssymb}
\usepackage{a4wide}
\newtheorem{thm}{Theorem}[section]
\newtheorem{la}[thm]{Lemma}
\newtheorem{Defn}[thm]{Definition}
\newtheorem{Remark}[thm]{Remark}

\newtheorem{prop}[thm]{Proposition}
\newtheorem{cor}[thm]{Corollary}
\newtheorem{Example}[thm]{Example}
\newtheorem{Number}[thm]{\!\!}
\newenvironment{defn}{\begin{Defn}\rm}{\end{Defn}}
\newenvironment{example}{\begin{Example}\rm}{\end{Example}}
\newenvironment{rem}{\begin{Remark}\rm}{\end{Remark}}
\newenvironment{numba}{\begin{Number}\rm}{\end{Number}}
\newenvironment{proof}{{\noindent\bf Proof.}}%
                  {\nopagebreak\hspace*{\fill}$\Box$\medskip\medskip\par}   
\newcommand{\Punkt}{\nopagebreak\hspace*{\fill}$\Box$}

\newcommand{\ve}{\varepsilon}
\newcommand{\wt}{\widetilde}
\newcommand{\tensor}{\otimes}
\newcommand{\impl}{\Rightarrow}

\newcommand{\mto}{\mapsto}
\newcommand{\isom}{\cong}
\newcommand{\N}{{\mathbb N}}
\newcommand{\R}{{\mathbb R}}
\newcommand{\K}{{\mathbb K}}
\newcommand{\bL}{{\mathbb L}}
\newcommand{\bD}{{\mathbb D}}
\newcommand{\Z}{{\mathbb Z}}
\newcommand{\C}{{\mathbb C}}
\DeclareMathOperator{\ad}{ad}

\newcommand{\cO}{{\mathcal O}}

\newcommand{\cS}{{\mathcal S}}
\newcommand{\cg}{{\mathfrak g}}
\DeclareMathOperator{\Diff}{Diff}
\newcommand{\dl}{{\displaystyle \lim_{\longrightarrow}}}
\newcommand{\pl}{{\displaystyle \lim_{\longleftarrow}}}

\newcommand{\wh}{\widehat}

\newcommand{\sub}{\subseteq}
\DeclareMathOperator{\GL}{GL}

\DeclareMathOperator{\pr}{pr}
\DeclareMathOperator{\id}{id}

\newcommand{\cT}{{\mathcal T}}
\newcommand{\cA}{{\mathcal A}}

\DeclareMathOperator{\Supp}{supp}

\DeclareMathOperator{\ev}{ev}

\DeclareMathOperator{\absconv}{absconv}
\DeclareMathOperator{\op}{op}
\newcommand{\aeq}{\Leftrightarrow}
\begin{document}
$\;$\\[-28mm]
\begin{center}
{\Large \bf Aspects of differential calculus related to infinite-\\[2mm]
dimensional vector bundles and Poisson vector spaces}\\[4mm]
{\bf Helge Gl\"{o}ckner}\vspace{2mm}
\end{center}
\begin{abstract}
\hspace*{-5.9mm}We prove various results in infinite-dimensional
differential calculus
which relate\linebreak
differentiability properties
of functions and associated operator-valued functions
(e.g., differentials).
The results are applied in two areas:
 \begin{itemize}
\item[(1)]
in the theory of infinite-dimensional
vector bundles, to construct new bundles
from given ones,
like dual bundles, topological
tensor products, infinite direct sums,
and completions
(under suitable hypotheses).
\item[(2)]
in the theory of locally convex Poisson vector spaces,
to prove continuity
of the Poisson bracket and continuity of passage
from a function to the associated
Hamiltonian vector field.
\end{itemize}
Topological properties of topological vector spaces
are essential for the studies, which allow
hypocontinuity of bilinear mappings to
be exploited.
Notably, we encounter $k_\R$-spaces and
locally convex spaces~$E$
such that $E\times E$ is a $k_\R$-space.
\end{abstract}
{\footnotesize
{\em Subject Classification.}
26E15 (primary);
%
% 26E15 = real functions: calc of functions on inf-dim spaces
%
17B63, % Poisson algebras
22E65, % inf-dim Lie groups
26E20,
%
% 26E20 = real functions: calc of functions with values in
%         inf-dim spaces
46G20, % inf-dim hol maps
54B10, % product spaces in gen topology
54D50, % k-spaces
55R25, % sphere bundles and vector bundles
58B10\\[1mm]
% inf-dim manifolds: diff'bility questions
% 
{\em Keywords and Phrases.}
Vector bundle, dual bundle,
direct sum, completion, tensor product,
cocycle, smoothness, analyticity, hypocontinuity,
$k$-space, compactly generated space,
infinite-dimensional Lie group, Poisson vector space,
Poisson bracket, Hamiltonian vector field, group action, multilinear map}\vspace{-2mm}
{\footnotesize
\tableofcontents}\pagebreak

\begin{center}
{\bf\large Introduction}
\end{center}
We study questions
of infinite-dimensional
differential calculus
in the setting of
Keller's $C^k_c$-theory~\cite{Kel}
(going back to \cite{Bas}).
Applications to
infinite-dimensional vector bundles are given,
and also
applications in the theory of locally convex Poisson vector spaces.\\[3mm]
{\bf Differentiability properties of operator-valued maps.}
\,Our results
are centered around the following basic problem:
Consider locally convex spaces
$X$, $E$ and $F$, an open set
$U\sub X$ and
a map $f \colon U\to L(E,F)_b$
to the space
of continuous linear maps,
endowed with the topology of uniform
convergence on bounded sets.
How are differentiability\linebreak
properties
of the operator-valued map $f$ related to those
of
\[
f^\wedge\colon U\times E\to F\,,\quad f^\wedge(x,v)\,:=\, f(x)(v)\,?
\]
We show that if~$f^\wedge$ is smooth,
then also~$f$ is smooth (Proposition~\ref{reallyneed1}).
Conversely,\linebreak
exploiting the hypocontinuity of the bilinear
evaluation map
\[
L(E,F)_b\times E\to F\, , \quad (\alpha,v)\mto \alpha(v)\, ,
\]
we find natural hypotheses on~$E$ and~$F$ ensuring
that smoothness of~$f$
entails smoothness of~$f^\wedge$
(Proposition~\ref{reallyneed2}; likewise for compact sets
in place of bounded sets).
Without extra hypotheses on~$E$ and~$F$,
this conclusion becomes false,
e.g.\ if $U=X$ is a non-normable
real locally convex convex space
with dual space $X':=L(X,\R)$.
Then $f:=\id_{X'}\colon X'_b\to X'_b$
is continuous linear and thus smooth, but
$f^\wedge\colon X'_b\times X\to\R$
is the bilinear evaluation map taking $(\lambda,x)$ to
$\lambda(x)$,
which is discontinuous
for non-normable~$X$
(see \cite[p.\,2]{KaM})
and hence not smooth
in the sense of Keller's
$C^\infty_c$-theory.
We also obtain results concerning finite order
differentiability properties,
as well as real and complex analyticity.
Furthermore, $L(E,F)$
can be replaced with the space $L^k(E_1,\ldots, E_k, F)$ of continuous
$k$-linear maps $E_1\times \cdots\times E_k\to F$,
if $E_1,\ldots, E_k$ are locally convex
spaces.\footnote{Related questions also
play a role in the comparative study of
differential calculi~\cite{Kel}.}
As a very special case of our studies,
the differential
\[
f'\colon U\to L(E,F)_b
\]
is $C^{r-2}$,
for each $r\in \N\cup\{\infty\}$ with $r\geq 2$, locally convex
spaces $E$ and $F$, and $C^r$-map
$f\colon U\to F$ on an open set $U\sub E$
(see Proposition~\ref{dfisC}).
This result is used,
e.g., in~\cite{IMP},
to study implicit functions from topological
vector spaces to Banach spaces.\\[3mm]
{\bf Applications to infinite-dimensional vector bundles.}
\,Apparently, mappings of the specific form
just described play a vital role in the theory of vector
bundles:
If~$F$ is a locally convex space, $M$ a (not
necessarily finite-dimensional)
smooth manifold, and $(U_i)_{i\in I}$ an open cover of~$M$,
then the smooth vector bundles $E\to M$, with fibre~$F$,
which are trivial over the sets~$U_i$,
can be described by cocycles
$g_{ij}\colon U_i\cap U_j\to \GL(F)$
such that $G_{ij}:=g_{ij}^\wedge \colon (U_i\cap U_j)\times F\to F$,
$(x,v)\mto g_{ij}(x)(v)$ is smooth
(Proposition~\ref{standfact},
Remark~\ref{stand2}.
Then $g_{ij}$
is smooth as a mapping to the space
$L(F)_b:=L(F,F)_b$
(see Proposition~\ref{reallyneed1}).
In various contexts, for example
when trying to construct dual bundles,
we are in an\linebreak
opposite situation:
we know that each $g_{ij}$ is smooth,
and would like to conclude that also
the mappings $G_{ij}$ are smooth.
Although this is not possible
in general (as examples show),
our results provide
additional
conditions ensuring that the
conclusion is correct in the specific situation
at hand.
Notably,
we obtain conditions
ensuring the existence
of a canonical dual bundle (Proposition~\ref{finalprop}).
Without extra conditions, a canonical dual bundle
need not exist (Example~\ref{nodual}).\\[2.5mm]
Besides dual bundles, we discuss
a variety of construction principles
of new vector bundles from given ones,
including topological tensor products,
completions, and finite or infinite direct sums.
More generally, given
a (finite- or infinite-dimensional)
Lie group acting on the base manifold~$M$,
we discuss the
construction of new
equivariant vector bundles
from given ones.
Most of the constructions
require specific hypotheses on the
base manifold, the fibre of the bundle,
and the Lie group.\\[2.3mm]
As to completions,
complementary topics were
considered in the literature:
Given an infinite-dimensional
smooth manifold~$M$, completions
of the tangent bundle with respect to a weak
Riemannian metric occur in~\cite[p.\,549]{MMM},
in hypotheses for a so-called \emph{robust}
Riemannian manifold.
Each $C^{r+1}$-map between open subsets of locally convex spaces
locally factors
over a $C^r$-map between open subsets
of Banach spaces (see \cite[Appendix A]{GaH}).\\[2.3mm]
We mention that multilinear algebra
and vector bundle constructions can be performed
much more easily in an inequivalent setting of infinite-dimensional
calculus, the convenient differential calculus~\cite{KaM}.
However,
a weak notion of vector bundles
is used there, which need not be
topological vector bundles.
Our discussion of vector bundles
intends to pinpoint additional conditions
ensuring that the natural construction
principles lead to vector bundles
in a stronger sense (which are, in particular,
topological vector bundles).\\[2.3mm]
The work \cite{Wur} was particularly important for our studies.
For an open subset~$U$ of a Fr\'{e}chet space~$E$,
smoothness of $f^\wedge\colon U\times E^k\to \R$
is deduced from smoothness of $f\colon U\to\Lambda^k(E')_b$
in the proof of \cite[Proposition~IV.6]{Wur}.
A typical hypocontinuity argument
already appears in the proof of \cite[Lemma~IV.7]{Wur}.
In contrast to the local calculations in charts,
the global structure on a dual bundle
(and bundles of $k$-forms) asserted in the first remark of \cite[p.\,339]{Wur}
are problematic if Keller's $C^\infty_c$-theory
is used, without further hypotheses.\\[2.3mm]
{\bf Applications related to locally convex Poisson vector spaces.}
\,In the wake of works
by
Odzijewicz and Ratiu on
Banach-Poisson vector spaces and Banach-Poisson manifolds
\cite{OR1,OR2},
certain locally convex Poisson vector spaces
were introduced~\cite{MEM},
which generalize the
Lie-Poisson structure on the dual
space of a finite-dimensional
Lie algebra going back to
Kirillov, Kostant and Souriau.
By now, the latter spaces can be embedded
in a general theory of locally convex Poisson manifolds
(see \cite{NST};
for generalizations of finite-dimensional
Poisson geometry with a different thrust, cf.\ \cite{Bel}).
Recall that many important
examples of bilinear mappings between
locally convex topological vector spaces
are not continuous, but at least
hypocontinuous (cf.\ \cite{Bou}
for this classical concept).
In Sections~\ref{nwsec1} and \ref{nwsec2},
we provide the proofs for two fundamental results
in the theory of locally convex Poisson vector
spaces which are related to hypocontinuity.\footnote{These proofs were
stated in the preprint version of~\cite{MEM},
but not included in the
actual publication.}
We show that the Poisson bracket associated with a continuous
Lie bracket is always continuous
(Theorem~\ref{poiss2})
and that the linear map $C^\infty(E,\R)\to C^\infty(E,E)$
taking a smooth function to the associated
Hamiltonian vector field is continuous (Theorem~\ref{poiss3}).
Ideas from~\cite{MEM} and the current article
were also taken further in \cite[Section 13]{SMO}.\\[2.3mm]
{\footnotesize {\em Acknowledgements.}
A very limited first draft
was written in 2001/02, supported by
the research group FOR~363/1-1
of the German Research Foundation, DFG
(working title:
\emph{Bundles of locally convex spaces,
group actions, and hypocontinuous bilinear mappings}).
The material was expanded in 2007, supported by
DFG grant
GL 357/5-1. Substantial
extensions and a major rewriting
were carried out
in~2022.}
\section{Preliminaries and notation}\label{secprelim}
We describe our setting of differential calculus
and compile
useful facts.
Either references to the literature
are given or a proof;
the proofs can be looked up in Appendix~\ref{appA}.\\[3mm]
{\bf Infinite-dimensional calculus.}
\,We work in the framework
of infinite-dimensional
differential calculus known as Keller's $C^k_c$-theory~\cite{Kel}.
Our main references are \cite{RES} and \cite{GaN}
(see also \cite{Mic}, \cite{Ham} \cite{Mil},
and \cite{NSU}).
If $\K\in \{\R,\C\}$,
we let $\bD:=\{t\in\K\colon |t|\leq 1\}$
and $\bD_\ve:=\{t\in\K\colon |t|\leq \ve\}$
for $\ve>0$.
We write $\N:=\{1,2,\ldots\}$ and $\N_0:=\N\cup\{0\}$.
All topological vector spaces
considered in the article are assumed Hausdorff,
unless the contrary is stated.
For brevity,
Hausdorff locally convex topological
vector spaces will be called locally convex spaces.
As usual, a subset $M$ of a $\K$-vector space
is called \emph{balanced}
if $tx\in M$ for all $x\in M$ and $t\in\bD$.
The subset $M$ is called \emph{absolutely convex}
if it is both convex and balanced.
If $q\colon E\to[0,\infty[$
is a seminorm on a $\K$-vector space~$E$,
we write $B^q_\ve(x):=\{y\in E\colon q(y-x)<\ve\}$
for $x\in E$ and $\ve>0$.
We also write $\|x\|_q$ in place of $q(x)$.
If $E$ is a locally convex $\K$-vector space,
we let $E'$ be the dual space
of continuous $\K$-linear functionals
$\lambda\colon E\to\K$.
We write $M^\circ:=\{\lambda\in E'\colon \lambda(M)\sub\bD\}$
for the polar of a subset $M\sub E$.
If $\alpha\colon E\to F$ is a continuous
$\K$-linear map between locally convex $\K$-vector spaces,
we let $\alpha'\colon F'\to E'$, $\lambda\mto \lambda\circ\alpha$
be the dual linear map.
We say that a mapping $f\colon X\to Y$ between
topological spaces is a \emph{topological embedding}
if it is a homeomorphism onto its image.
\begin{numba}\label{defsmoo}
Let $E$ and~$F$ be locally convex $\K$-vector spaces
over $\K\in\{\R,\C\}$ and $U\sub E$ be an open subset.
A map $f\colon U\to F$ is called $C^0_\K$ if it is
continuous, in which case we set $d^0f:=f$.
Given
$x\in U$ and $y\in E$,
we define
\[
df(x,y):=(D_yf)(x):=\lim_{t\to 0}\frac{f(x+ty)-f(x)}{t}
\]
if the limit exists
(using $t\in \K^\times$ such that $x+ty\in U$).
Let $r\in \N\cup\{\infty\}$.
We say that a continuous map $f\colon U\to F$
is a \emph{$C^r_\K$-map} if the iterated directional
derivative
\[
d^kf(x,y_1,\ldots, y_k):= (D_{y_k}\cdots D_{y_1})(f)(x)
\]
exists for all
$k\in \N$ such that $k\leq r$
and all $(x,y_1,\ldots,y_k)\in U\times E^k$,
and if the maps $d^kf\colon U\times E^k\to F$
so obtained are continuous. Thus $d^1f=df$.
If $\K$ is understood, we write $C^r$ instead
of $C^r_\K$. As usual, $C^\infty$-maps
are also called \emph{smooth}.
\end{numba}
\begin{numba}\label{fanddf}
For $k\in \N$,
it is known that a map
$f\colon U\to F$ as before
is $C^k_\K$ if and only if
$f$ is $C^1_\K$ and $df\colon U\times E\to F$
is $C^{k-1}_\K$
(cf.\ \cite[Proposition 1.3.10]{GaN} or
\cite[Lemma~1.14]{RES}).
\end{numba}
\begin{numba}\label{defcxan}\label{cxlinan}
If $\K=\C$, it is known that a map
$f\colon E\supseteq U\to F$ as before
is $C^\infty_\C$ if and only if it is
\emph{complex analytic} in the sense
of \cite[Definition~5.6]{BaS}:
$f$ is continuous and for each $x\in U$, there exists a
$0$-neighbourhood $Y\sub E$ such that $x+Y\sub U$
and $f(x+y)=\sum_{n=0}^\infty\beta_n(y)$ for
all $y\in Y$ as a pointwise limit,
where $\beta_n\colon E\to F$
is a continuous homogeneous polynomial over~$\C$
of degree~$n$, for each $n\in \N_0$ \cite[Theorem 2.1.12]{GaN}.
Furthermore, $f$ is complex analytic
if and only if $f$ is $C^\infty_\R$
and $df(x,\cdot)\colon E\to F$ is complex linear for
all $x\in U$ (see \cite[Lemma~2.5]{RES}).
Complex analytic maps will also
be called \emph{$\C$-analytic} of $C^\omega_\C$.
\end{numba}
\begin{numba}
If $\K=\R$, then a map
$f\colon U\to F$ as in~\ref{defsmoo}
is called \emph{real analytic}
(or \emph{$\R$-analytic}, or~$C^\omega_\R$)
if it extends
to a complex analytic mapping $\wt{U}\to F_\C$
on some open neighbourhood~$\wt{U}$ of~$U$
in the complexification~$E_\C$ of~$E$.
\end{numba}
In the following, $r\in \N_0\cup\{\infty,\omega\}$
(unless the contrary is stated).
We use the conventions $\infty+k:=\infty-k:=\infty$
and $\omega+k:=\omega-k:=\omega$, for each $k\in \N$.
Furthermore, we extend the order on $\N_0$
to an order on $\N_0\cup \{\infty,\omega\}$
by declaring $n<\infty<\omega$ for each $n\in\N_0$.
\begin{numba}
Compositions of composable
$C^r_\K$-maps are $C^r_\K$-maps
(see \cite[Proposition~1.3.4]{GaN} and\linebreak
\cite[Propositions 2.7 and 2.8]{RES}).
Thus $C^r_\K$-manifolds
modeled on locally convex $\K$-vector spaces can be defined
in the usual way (see \cite[Chapter 3]{GaN} for a detailed exposition).
In this article,
the word ``manifold'' (resp., ``Lie group'')
always refers to a manifold (resp., Lie group)
modeled on a locally convex space.
\end{numba}
The following basic fact
will be used repeatedly.
\begin{la}\label{wippe}
For $k\in \N$, let $X$, $E_1,\ldots, E_k$, and $F$
be locally convex $\K$-vector spaces,
$U\sub X$ be an open subset and
\[
f\colon U\times E_1\times\cdots\times E_k\to F
\]
be a $C^1_\K$-map such that $f^\vee(x):=f(x,\cdot)\colon
E_1\times \cdots\times E_k\to F$ is $k$-linear,
for each $x\in U$. Let $x\in U$ and~$q$ be a continuous
seminorm on~$F$.
Then there exist
a continuous seminorm~$p$ on~$X$
with $B^p_1(x)\sub U$, and
continuous seminorms $p_j$ on~$E_j$
for $j\in\{1,\ldots, k\}$ such that
\begin{equation}\label{wuse1}
\|f(y,v_1,\ldots, v_k)\|_q\leq \|v_1\|_{p_1}\cdots\|v_k\|_{p_k}\quad\mbox{and}
\end{equation}
\begin{equation}\label{wuse2}
\|f(y,v_1,\ldots, v_k)-f(x,v_1,\ldots, v_k)\|_q\leq \|y-x\|_p
\|v_1\|_{p_1}\cdots\|v_k\|_{p_k}
\end{equation}
for all $y\in B^p_1(x)$ and $(v_1,\ldots, v_k)\in E_1\times\cdots \times E_k$.
\end{la}
We shall also use the following fact:
\begin{la}\label{cont-via-gateaux}
Let $E$ and $F$ be locally convex $\K$-vector
spaces, $k\geq 2$ be an integer
and $f\colon U\times E^k\to F$ be
a mapping such that $f(x,\cdot)\colon E^k\to F$
is $k$-linear and symmetric for each
$x\in U$. Let $r\in\N_0\cup\{\infty,\omega\}$.
If
\[
h\colon U\times E\to F,\quad (x,y)\mto f(x,y,\ldots, y)
\]
is $C^r_\K$, then also $f$ is~$C^r_\K$.
Notably, $f$ is continuous
if~$h$ is continuous.
\end{la}
{\bf {\boldmath$k$}-spaces, {\boldmath$k_\R$}-spaces, {\boldmath$k^\infty$}-spaces,
and {\boldmath$k_\omega$}-spaces.}
We recall topological concepts.
\begin{numba}
A topological space $X$ is said to be \emph{completely regular}
if it is Hausdorff
and its topology is initial with
respect to the set $C(X,\R)$
of all continuous real-valued functions
on~$X$.
\end{numba}
Every locally convex space is
completely regular,
like every Hausdorff
topological group (cf.\ \cite[Theorem~8.2]{HaR}).
Compare \cite{Eng} and \cite{Kly}
for the following.
\begin{numba}
A topological space~$X$
is called a \emph{$k$-space}
if it is Hausdorff and a subset $A\sub X$
is closed if and only if $A\cap K$ is closed in~$K$
for each compact subset $K\sub X$.
Every metrizable topological space is
a $k$-space, and every locally compact Hausdorff space.
A Hausdorff space $X$ is a $k$-space
if and only if, for each topological space,
a map $f\colon X\to Y$ is continuous
if and only if $f$ is \emph{$k$-continuous}
in the sense that $f|_K$ is continuous for each
compact subset $K\sub X$.
If $X$ is a $k$-space,
then also every subset $M\sub X$
which is open or closed in~$X$, when the induced
topology is used on~$M$.
\end{numba}
\begin{numba}
A topological space~$X$
is called a \emph{$k_\R$-space}
if it is Hausdorff and a function
$f\colon X\to\R$ is continuous if and only
if $f$ is $k$-continuous.
Then also a map $f\colon X\to Y$
to a completely regular
topological space~$Y$ is continuous
if and only if it is $k$-continuous
(as the latter condition implies continuity of
$g\circ f$ for each $g\in C(Y,\R)$).
For more information, \hspace*{-.2mm}cf.\hspace*{-.4mm}~\cite{Nob}.
\end{numba}
Every $k$-space is a $k_\R$-space.
The converse is not true:
$\R^I$ is known to be a $k_\R$-space
for each set~$I$
(see \cite{Nob}, also \cite{GaM}).
If $I$ has cardinality $\geq 2^{\aleph_0}$,
then $\R^I$ is not a $k$-space.\footnote{If
$\R^I$ was a $k$-space,
then a certain non-discrete
subgroup $G$ of $(\R^\R,+)$
constructed in \cite{FHS}
would be discrete,
contradiction
(see \cite[Remark~A.6.16\,(a)]{GaN}
for more details).
Compare also \cite{Nob}.}\\[2.5mm]
The following facts are well known (cf.\ \cite{Nob}):
\begin{la}\label{basic-k-R}
\begin{itemize}
\item[\rm(a)]
If a $k_\R$-space~$X$ is a direct product
$X_1\times X_2$ of Hausdorff spaces and $X_1\not=\emptyset$,
then $X_2$ is a $k_\R$-space.
\item[\rm(b)]
Every open subset~$U$ of a completely regular
$k_\R$-space~$X$ is a $k_\R$-space
in the induced topology.
\end{itemize}
\end{la}
Notably, $U$ is a $k_\R$-space
for each open subset $U$ of a locally convex space~$E$
which is a $k_\R$-space.
If $E\times E$ is a $k_\R$-space, then also~$E$.
\begin{numba}
Following~\cite{MEM}, a topological space $X$ is called a \emph{$k^\infty$-space}
if the cartesian power $X^n$ is a $k$-space for each $n\in\N$,
using the product topology.
A Hausdorff space $X$ is called \emph{hemicompact}
if $X=\bigcup_{n\in\N}K_n$
for a sequence $K_1\sub K_2\sub\cdots$
of compact subsets $K_n\sub X$ such that
each compact subset of~$X$ is a subset of some~$K_n$.
Hemicompact $k$-spaces are also called \emph{$k_\omega$-spaces}.
If $X$ and $Y$ are $k_\omega$-spaces,
then the product topology makes $X\times Y$ a $k_\omega$-space.
Notably, every $k_\omega$-space
is a $k^\infty$-space. See
\cite{Cha,Fra,GGH} for further information.
Finite products of metrizable spaces
being metrizable, every metrizable topological space is a $k^\infty$-space.
Recall that a locally convex space~$E$
is said to be a \emph{Silva space} or (DFS)-space
if it is the locally convex inductive limit of a sequence
$E_1\sub E_2\sub\cdots$ of Banach spaces
such that each inclusion map $E_n\to E_{n+1}$
is a compact operator. Every Silva space is a $k_\omega$-space
(see \cite[Example~9.4]{COM}).
\end{numba}
{\bf Spaces of multilinear maps.}
\,Given $k\in \N$, locally convex $\K$-vector spaces
$E_1,\ldots, E_k$ and $F$,
and a set $\cS$ of bounded subsets
of $E_1\times\cdots \times E_k$,
we write $L^k(E_1,\ldots, E_k,F)_\cS$ or
$L_\K^k(E_1,\ldots, E_k,F)_\cS$ for the space
of
continuous $k$-linear maps
$E_1\times\cdots\times E_k\to F$,
endowed with the topology $\cO_\cS$ of uniform convergence
on the sets $B\in \cS$.
Recall that finite intersections of
sets of the form
\[
\lfloor B, U\rfloor\; :=\;
\{\beta\in L^k(E_1,\ldots, E_k,F)\colon \beta(B)\sub U\}
\]
yield a basis of $0$-neighbourhoods
for this (not necessarily Hausdorff) locally convex vector topology,
for $U$ ranging through the $0$-neighbourhoods
in~$F$ and $B$ through~$\cS$.
If $\bigcup_{B\in\cS}B=E_1\times\cdots\times E_k$,
then $\cO_\cS$ is Hausdorff.
If $E_1=\cdots=E_k$, we abbreviate $L^k(E,F)_\cS:=L^k(E,\ldots,E,F)_\cS$.
If $k=1$ and $E:=E_1$,
we abbreviate $L(E,F)_\cS:=L^1(E,F)_\cS$,
$L_\K(E,F)_\cS$ $:=L^1_\K(E,F)_\cS$
and $L(E)_\cS:=L(E,E)_\cS$. We write $\GL(E)=L(E)^\times$
for the group of all automorphisms
of the locally convex $\K$-vector space~$E$.
If $\cS$ is the set of all bounded,
compact, and finite
subsets of $E_1\times\cdots\times E_k$, respectively,
we shall usually write ``$b$,'' ``$c$,'' and ``$p$''
in place of $\cS$. For example, we shall write
$L^k(E_1,\ldots, E_k,F)_b$,
$L^k(E_1,\ldots,E_k,F)_c$, and $L^k(E_1,\ldots, E_k,F)_p$.
\begin{numba}\label{closedsub}
Let $E_1,\ldots, E_k$ and $F$ be complex locally convex
spaces and $f\colon \!U\! \to L^k_\C(E_1,\ldots, E_k,F)$ be a map,
defined on an open subset~$U$ of a real
locally convex space. Let $\cS:=b$ or $\cS:=c$.
Since $L_\C^k(E_1,\ldots, E_k,F)_\cS$
is a closed real vector subspace of $L_\R^k(E_1,\ldots, E_k,F)_\cS$,
the map~$f$ is $C^r_\R$
as a map to $L_\C^k(E_1,\ldots, E_k,F)_\cS$
if and only if $f$ is $C^r_\R$
as a map to $L_\R^k(E_1,\ldots, E_k,F)_\cS$
(see \cite[Lemma~1.3.19]{GaN}
and \cite[Proposition~2.11]{RES}).
\end{numba}
\begin{numba}
Given a $C^r_\K$-map $f\colon E\supseteq U\to F$ as
in~\ref{defsmoo}, we define $f^{(0)}:=f$ and
\[
f^{(j)}\colon U\to L^j_\K(E,F),\quad f^{(j)}(x):=(d^jf)^\vee(x)=d^jf(x,\cdot)
\] 
for $j\in \N$ such that $j\leq r$.
\end{numba}
{\bf Hypocontinuous multilinear maps.}
Beyond normed spaces, typical multilinear
maps are not continuous, but merely
hypocontinuous.
Hypocontinuous \emph{bi}linear maps
are discussed in many textbooks.
An analogous notion
of hypocontinuity for multilinear maps
(to be described presently)
is useful us.
It can be discussed like
the bilinear case.
\begin{la}\label{prehypo}
For an integer $k\geq 2$,
let $\beta\colon E_1\times \cdots\times E_k\to F$
be a separately continuous $k$-linear mapping and
$j\in \{2,\ldots,k\}$ such that,
for each $x\in E_1\times\cdots\times E_{j-1}$, the map 
\[
\beta^\vee(x):=\beta(x,\cdot)\colon E_j\times \cdots\times E_k\to F
\]
is continuous.
Let $\cS$ be a set of bounded subsets
of $E_j\times\cdots \times E_k$.
Consider the conditions:
\begin{itemize}
\item[\rm(a)]
For each $M \in \cS$
and each $0$-neighbourhood $W \sub F$,
there exists a $0$-neighbourhood
$V\sub E_1\times\cdots\times E_{j-1}$
such that $\beta(V\times M)\sub W$.
\item[\rm(b)]
The $(j-1)$-linear map $\beta^\vee\colon
E_1\times \cdots\times E_{j-1}\to L^{k-j+1}(E_j,\ldots, E_k,F)_\cS$
is continuous.
\item[\rm(c)]
$\beta|_{E_1\times\cdots\times E_{j-1}\times M}\colon E_1\times \cdots\times E_{j-1}\times
M\to F$
is continuous, for each $M\in \cS$.
\end{itemize}
The {\rm(a)} and {\rm(b)}
are equivalent, and {\rm(b)}
implies {\rm(c)}.
If
\begin{equation}\label{simplf}
(\forall M\in \cS)\;(\exists N\in \cS)\;\;\;
\bD M\sub N,
\end{equation}
then {\rm(a)}, {\rm(b)}, and {\rm(c)}
are equivalent.
\end{la}
\begin{defn}\label{defhypo}
A $k$-linear map $\beta$
which satisfies the hypotheses and Condition~(a)
of Lemma~\ref{prehypo}
is called \emph{$\cS$-hypocontinuous}
in its arguments $(j,\ldots, k)$.
If $j=k$, we also say that $\beta$ is \emph{$\cS$-hypocontinuous in
the $k$-th argument}.
Analogously,
we define $\cS$-hypocontinuity of~$\beta$ in the $j$-th argument,
if $j\in \{1,\ldots, k\}$
and a set $\cS$ of bounded subsets of $E_j$
are given.
\end{defn}
We are mainly interested in
$b$-, $c$-, and $p$-hypocontinuity,
viz., in $\cS$-hypocontinuity
with respect
to the set~$\cS$ of all bounded subsets of $E_j\times\cdots\times E_k$,
the set~$\cS$ of all compact subsets,
and the set $\cS$ of all finite subsets,
respectively.
If $\cS$ and $\cT$ are sets
of bounded subsets of $E_j\times\cdots\times E_k$
such that $\cS\sub\cT$
and $\beta$ is $\cT$-hypocontinuous in its variables
$(j,\ldots, k)$,
then $\beta$ is also $\cS$-hypocontinuous
in the latter.
The following is obvious from
Lemma~\ref{prehypo}\,(c)
(as the elements of a convergent sequence,
together with its limit, form a compact set):
\begin{la}\label{hypo-seq-cts}
If $\beta\colon E_1\times\cdots \times E_k\to F$
is $c$-hypocontinuous in some argument,
or in its arguments $(j,\ldots,k)$
for some $j\in \{2,\ldots,k\}$,
then $\beta$ is sequentially continuous.\,\Punkt
\end{la}
In many cases,
separately continuous
bilinear maps are automatically
hypocontinuous.\linebreak
Recall that
a subset~$B$ of a locally convex space~$E$
is a \emph{barrel}
if it is closed,
absolutely convex, and absorbing.
The space $E$ is called
\emph{barrelled}
if every barrel is a $0$-neighbourhood.
See Proposition~6 in \cite[Chapter~III, \S5, no.\,3]{Bou}
for the following fact.
\begin{la}\label{auothypo}
If $\beta \colon E_1\times E_2 \to F$
is a separately continuous
bilinear map and $E_1$ is\linebreak
barrelled, then~$\beta$
is $\cS$-hypocontinuous in its second argument,
with respect to
any set~$\cS$ of bounded subsets of~$E_2$.\,\Punkt
\end{la}
Evaluation maps are paradigmatic examples of
hypocontinuous multilinear maps.
\begin{la}\label{resulteval}
Let $E_1,\ldots, E_k$ and~$F$ be locally convex $\K$-vector
spaces
and~$\cS$ be a set of bounded
subsets of $E:=E_1\times\cdots \times E_k$
with $\bigcup_{M\in \cS}M=E$.
Then the
$(k+1)$-linear map
\[
\ve\colon L^k(E_1,\ldots, E_k,F)_\cS \times E_1\times\cdots\times E_k \to F\, ,\quad
(\beta, x)\mto \beta(x)
\]
is $\cS$-hypocontinuous
in its arguments $(2,\ldots, k+1)$.
If $k=1$ and $E=E_1$ is barrelled, then
$\ve\colon L(E,F)\times E\to F$ is also hypocontinuous
in the first argument,
with respect to any locally
convex topology~$\cO$
on $L(E,F)$ which is finer
than the topology of pointwise
convergence, and any set~$\cT$ of bounded subsets
of $(L(E,F),\cO)$.
\end{la}
\begin{la}\label{compo-is-cts}
Let $k\geq 2$ be an integer,
$E_1,\ldots, E_k$ and $F$
be locally convex spaces,
and $\beta\colon E_1\times\cdots\times E_k\to F$
be a $k$-linear map.
\begin{itemize}
\item[\rm(a)]
If $\beta$ is sequentially continuous,
then $\beta\circ f$
is continuous for each continuous function
$f\colon X\to E_1\times\cdots\times E_k$
on a topological space~$X$
which is metrizable or satisfies
the first axiom of countability.
\item[\rm(b)]
If $\beta$ is $c$-hypocontinuous
in its arguments $(j,\ldots,k)$
for some $j\in\{2,\ldots,k\}$
and $X$ is a $k_\R$-space,
then $\beta\circ f$
is continuous for each continuous
function $f\colon X\to E_1\times\cdots\times E_k$.
\end{itemize}
\end{la}
{\bf Lipschitz differentiable maps.}
\,In Section~\ref{sec-complete},
it will be useful to work with
certain Lipschitz differentiable maps,
instead of $C^r$-maps.
We briefly recall concepts and facts.
\begin{defn}\label{lip-diffble}
Let $E$ and $F$ be locally convex $\K$-vector spaces,
$U\sub E$ be open and $f\colon U\to F$ be a map.
We say that $f$ is \emph{locally Lipschitz continuous}
or $LC^0_\K$
if it has the following property:
For each $x\in U$ and continuous
seminorm $q$ on~$F$, there exists
a continuous seminorm $p$ on~$E$
such that $B^p_1(x)\sub U$
and
\[
q(f(z)-f(y))\leq p(z-y)\quad 
\mbox{for all $y,z\in B^p_1(x)$.}
\]
Given $r\in \N_0\cup\{\infty\}$,
we say that $f$ is $LC^r_\K$
if $f$ is $C^r_\K$ and $d^kf\colon U\times E^k\to F$
is $LC^0_\K$ for each $k\in \N_0$ such
that $k\leq r$.
\end{defn}
Every $C^1$-map is $LC^0_\K$
(see, for example, \cite[Lemma~1.59]{MRG}).
As a consequence, for each $r\in \N\cup\{\infty\}$
every $C^r_\K$-map
is $LC^{r-1}_\K$. Notably,
every smooth map is $LC^\infty_\K$.
Moreover, a $C^r_\K$-map
with finite~$r$ is $LC^r_\K$
if and only if $d^rf$ is $LC^0_\K$.
The following facts are known, or part of the folklore.
\begin{la}\label{lem-lipdiff}
For locally convex spaces over
$\K\in\{\R,\C\}$ and $r\in \N_0\cup\{\infty\}$,
we have:
\begin{itemize}
\item[\rm(a)]
A map $f\colon E\supseteq U\to \prod_{j\in J}F_j$
to a direct product
of locally convex spaces is $LC^r_\K$ if and only
each component is $LC^r_\K$;
\item[\rm(b)]
Compositions
of composable $LC^r_\K$-maps are $LC^r_\K$;
\item[\rm(c)]
Let $F$ be a locally convex space and $F_0\sub F$ be
a vector subspace which is closed in~$F$,
or sequentially closed.
Then a map $f\colon E\supseteq U\to F_0$
is $FC^r_\K$ if and only if it is $FC^r_\K$
as a map to~$F$.
\item[\rm(d)]
A map $E\supseteq U\to P$
to a projective limit $P=\pl\,F_j$\vspace{-.4mm}
of locally convex spaces is $LC^r_\K$
if and only if $p_j\circ f\colon U\to F_j$
is $LC^r_\K$ for all $j\in J$,
where $p_j\colon P\to F_j$ is the limit~map.
\end{itemize}
\end{la}
Our concept of local Lipschitz continuity is weaker
than the one in~\cite[Definition 1.5.4]{GaN}.\\[2.5mm]
{\bf The compact-open {\boldmath$C^r$}-topology}
\begin{numba}\label{def-co-Cr}
If $E$ and~$F$ are locally convex
$\K$-vector spaces,
$U\sub E$ is an open set
and $r\in \N_0\cup\{\infty\}$,
then the vector space $C^r_\K(U,F)$
of all $C^r_\K$-maps $U\to F$
carries a natural topology
(the ``compact-open $C^r$-topology''),
namely the initial
topology with respect to the mappings
\[
C^r_\K(U,F)\to C(U\times E^j,F)_{c.o.}\,\quad
f \mto d^jf
\]
for $j\in \N_0$ such that $j\leq r$,
where the right hand side is
endowed with the compact-open topology.
Then $C^r_\K(U,F)$ is a locally convex
$\K$-vector space. If~$F$ is a complex locally convex space, then
also $C^r_\K(U,F)$.
See, e.g., \cite[\S{}1.7]{GaN}
for further information, or~\cite{ZOO}.
\end{numba}
The following observation was useful in an earlier
version of the manuscript, to enable
exponential laws (as in \cite{AaS})
to be applied.
We retain it as it may be useful
elsewhere.
\begin{la}\label{new-Ck-top-on-multi}
Let $\bL\in\{\R,\C\}$,
$\K\in\{\R,\bL\}$,
$k\in\N$, $r\in \N_0\cup\{\infty\}$,
and $E_1,\ldots, E_k$
as well as $F$ be locally convex
$\bL$-vector spaces.
Then $L_\bL^k(E_1,\ldots, E_k,F)$
is a closed vector subspace
of $C^r_\K(E_1\times \cdots\times E_k,F)$
with respect to the compact-open
$C^r$-topology.
The latter induces on $L^k_\bL(E_1,\ldots, E_k,F)$
the compact-open topology.
\end{la}
\section{Differentiability properties of operator-valued maps}\label{sec2}
Let $\bL\in\{\R,\C\}$, $\K\in\{\R,\bL\}$,
and $r\in \N_0\cup\{\infty,\omega\}$.
In this section,
we establish the following proposition.
\begin{prop}\label{reallyneed1}
Let $k\in \N$, $r\in\N_0\cup\{\infty,\omega\}$,
$E_1,\ldots, E_k$ and~$F$ be locally convex
$\bL$-vector spaces,
$X$ be a locally convex $\K$-vector space,
and $U\sub X$ be an open subset.
Let $f \colon U\to L_\bL^k(E_1,\ldots, E_k,F)$
be a map such that
\[
f^\wedge\colon U\times E_1\times \cdots\times E_k\to F\,, \quad
f^\wedge(x,v):=f(x)(v)\quad\mbox{for $x\in U$, $v\in E_1\times\cdots\times E_k$}
\]
is $C^r_\K$. Then the following holds:
\begin{itemize}
\item[\rm(a)]
$f$ is $C^r_\K$ as a map to $L_\bL^k(E_1,\ldots, E_k,F)_c$.
\item[\rm(b)]
If $r\geq 1$, then
$f$ is $C^{r-1}_\K$ as a map to $L_\bL^k(E_1,\ldots, E_k,F)_b$.
\end{itemize}
Furthermore,
\begin{equation}\label{ugly}
d^jf(x,y_1,\ldots, y_j)(v)=d^j(f^\wedge)((x,v),(y_1,0),\ldots, (y_j,0))
\end{equation}
for all $j\in \N$ with $j\leq r$
$($resp., $j\leq r-1$, in {\rm (b))},
all $x\in U$, $v\in E_1\times \cdots\times E_k$, and $y_1,\ldots, y_j\in X$.
\end{prop}
\begin{cor}\label{dfisC}
Let $E$ and $F$ be locally
convex $\K$-vector spaces
and $f\colon U\to F$ be a $C^r_\K$-map on an
open subset $U\sub E$, where $r\in \N\cup\{\infty,\omega\}$.
Then the following holds:
\begin{itemize}
\item[\rm(a)]
The map
$f^{(k)}\colon U\to L^k_\K(E,F)_c$,
$x\mto f^{(k)}(x)=d^kf(x,\cdot)$
is $C^{r-k}_\K$,
for each $k\in \N$ such that $k\leq r$.
\item[\rm(b)]
The map $f^{(k)}\colon U\to L^k_\K(E,F)_b$
is $C^{r-k-1}_\K$, for each
$k\in \N$ such that $k\leq r-1$.
\end{itemize}
Furthermore,
$d^j(f^{(k)})(x,y_1,\ldots, y_j)=d^{j+k}f(x,\cdot,y_1,\ldots, y_j)$,
for all $j\in \N$ with $j+k\leq r$
$($resp., $j+k\leq r-1)$,
all $x\in U$, and $y_1,\ldots, y_j\in E$.
\end{cor}
\begin{proof}
For each $k\in\N$ such that $k\leq r$,
the map $d^kf\colon U\times E^k\to F$
is $C^{r-k}_\K$
(see \cite[Remark~1.3.13 and Exercise 2.2.7]{GaN}),\footnote{Alternatively,
with a view towards Remark~\ref{rem-non-lcx},
we might use that $d^kf$ is a partial
map of a certain $C^{r-k}$-map $f^{[k]}$
if $k\not=\omega$
(cf.\ \cite[Proposition~7.4]{BGN}),
which also settles the case $(r,\K)=(\omega,\C)$
(see \cite[Proposition~7.7]{BGN}).
The real analytic case follows via complex analytic extension.}
and $f^{(k)}(x)=d^kf(x,\cdot)$
is $k$-linear for each $x\in U$, by \cite[Proposition 1.3.17]{GaN}
(or \cite[Lemma~4.8]{BGN}).
Moreover, $(f^{(k)})^\wedge=d^kf$.
Thus Proposition~\ref{reallyneed1} applies
with $f^{(k)}$ in place of~$f$
and $r-k$ in place of~$r$.
\end{proof}
Given a topological space $X$ and
locally convex space~$F$,
we endow the space $C(X,F)$
of continuous $F$-valued functions
on~$X$ with the compact-open
topology.\footnote{It is known that
this topology coincides
with the topology of uniform convergence on compact sets.}
The next lemma will be useful when we discuss
mappings to $L^k(E,F)_c$.
\begin{la}\label{passfvee}
Let $X$, $E$, and $F$ be locally convex
$\K$-vector spaces, $U\sub X$ and $W \sub E$
be open subsets, and $f\colon U\times W\to F$
be a $C^r_\K$-map, with $r\in \N_0\cup\{\infty\}$.
Then also the map
\[
f^\vee \colon U\to C(W,F)\,,\quad
x\mto f(x,\cdot)
\]
is $C^r_\K$.
If $\K=\R$
and $f$ admits
a complex analytic extension $g\colon \wt{U}\times
\wt{W}\to F_\C$
for suitable open neighbourhoods
$\wt{U}$ of~$U$ in $X_\C$
and $\wt{W}$ of~$W$ in~$E_\C$,
then $f^\vee$ is real analytic.
\end{la}
\begin{proof}
We first assume that $r\in \N_0$,
and proceed by induction.
For $r=0$, the assertion is well known (see, e.g.,
\cite[Proposition A.6.17]{GaN}).
Now assume that $r\in \N$.
Given $x\in U$ and $y\in X$,
there exists
$\ve>0$ such that $x+ \bD_\ve^0 y\sub U$,
where $\bD_r^0:=\{t\in \K\colon |t|<\ve\}$.
Consider
\[
g\colon \bD^0_\ve \times W\to F\,,\quad
(t,w)\mto
\left\{
\begin{array}{cl}
\frac{f(x+ty,w)-f(x,w)}{t} &\,\mbox{\,if $\,t\not=0$;}\\
df((x,w),(y,0)) &\,\mbox{\,if $\,t=0$.}
\end{array}
\right.
\]
Then $g(t,w)=\int_0^1 df((x+sty,w),(y,0))\, ds$,
by the Mean Value Theorem.
The integrand being continuous,
also $g$ is continuous
(by the Theorem on Parameter-Dependent Integrals,
\cite[Lemma 1.1.11]{GaN}).
Hence $g^\vee\colon V\to C(W,F)$ is continuous,
by induction, and hence
\[
\frac{f^\vee(x+ty)-f^\vee(x)}{t}\;=\;
g^\vee(t)\;\to\; g^\vee(0)
\]
as $t\to 0$, where $g^\vee(0)=df((x,\cdot),(y,0))=h^\vee(x,y)$
with
\[
h\colon (U\times E)\times W\to F\,,\quad
(x,y,w)\mto df((x,w),(y,0))\,.
\]
Since $h$ is $C^{r-1}_\K$,
the map $d(f^\vee)=h^\vee$
is $C^{r-1}_\K$, by the inductive hypothesis.
In particular, $df$ is continuous and hence
$f$ is $C^1_\K$. Now $f$ being $C^1_\K$
with $df$ a $C^{r-1}_\K$-map,
$f$ is $C^r_\K$.\vspace{2mm}

\emph{The case $r=\infty$}. If $f$ is $C^\infty_\K$,
then $f$ is $C^k_\K$ for each $k\in \N_0$.
Hence $f^\vee$ is $C^k_\K$ for each $k\in \N_0$
(by the case already treated), and thus $f^\vee$ is $C^\infty_\K$.\vspace{2mm}

\emph{Final assertion}. By the $C^\infty_\C$-case
already treated, the map
\[
g^\vee\colon \wt{U}\to C(\wt{W},F_\C)
\]
is $C^\infty_\C$. The restriction map
\[
\rho\colon C(\wt{W},F_\C)\to C(W,F_\C)\,,\quad \gamma\mto
\gamma|_W
\]
being continuous $\C$-linear and thus $C^\infty_\C$,
it follows that the composition
\[
h:=\rho\circ g^\vee\colon \wt{U}\to C(W,F_\C) =C(W,F)_\C
\]
is $C^\infty_\C$ and hence complex analytic.
Since $h$ extends $f^\vee$, we see that $f^\vee$
is real analytic.
\end{proof}
{\bf Proof of Proposition~\ref{reallyneed1}.}
(a) Abbreviate $E:=E_1\times\cdots\times E_k$.
Because $L_\bL^k(E_1,\ldots,E_k,F)_c$ is a closed $\K$-vector
subspace of $C(E,F)$ and carries
the induced topology, $f$ will be $C^r_\K$
as a map to $L_\bL^k(E_1,\ldots, E_k,F)_c$
if we can show that $f$ is $C^r_\K$
as a map to $C(E,F)$
(see \cite[Lemma~1.3.19]{GaN}
and \cite[Proposition~2.11]{RES}).
Since $f^\wedge$ is $C^r_\K$
and $f=(f^\wedge)^\vee$, the latter
follows from Lemma~\ref{passfvee}.
This is obvious unless
$\K=\R$ and $r=\omega$.
In this case, the map $f^\wedge$
admits a $\C$-analytic extension
$p\colon Q\to F_\C$
to an open neighbourhood $Q$
of $U\times E$ in $X_\C\times E_\C$.
For each $x\in U$,
there exists an open, connected neighbourhood
$U_x$ of~$x$ in $X_\C$ and
a balanced,
open $0$-neighbourhood $W_x\sub E_\C$
such that $U_x\times W_x\sub Q$
and $U_x\cap X\sub U$.
Let $D:=\{z\in \C\colon |z|<1\}$. Then
\[
q\colon U_x\times W_x\times D\to F_\C\,,\quad
(y,w,z)\mto p(y,zw)-z^kp(y,w)
\]
is a $\C$-analytic map which vanishes on
$(U_x\times W_x\times D)\cap (X\times E\times \R)$.
Hence $q=0$, by the Identity Theorem
(see \cite[Theorem~2.1.16\,(c)]{GaN}).
Then $p(y,zw)=z^kp(y,w)$ for all $z\in \C$ such that
$|z|\leq 1$, by continuity.
This implies that the map
\[
g \colon U_x\times E_\C\to F_\C\,,
\quad (y,w)\mto z^kp(y,z^{-1}w)\quad\mbox{for some $z\in \C^\times$
with $z^{-1}w\in W_x$}
\]
is well defined. Since $g$ is $\C$-analytic,
the final statement
of Lemma~\ref{passfvee} applies.\vspace{1mm}

(b) We prove the assertion for
$r\in \N$ first; then also the case $r=\infty$ follows.
If $r=1$, let $x\in U$.
Given an open
$0$-neighbourhood
$W\sub F$
and bounded subset $B\sub E:=E_1\times\cdots \times E_k$,
let~$q$ be a continuous seminorm on~$F$
such that $B^q_1(0)\sub W$.
By Lemma~\ref{wippe},
there exist continuous seminorms
$p$ on~$X$ and $p_j$ on~$E_j$ for $j\in \{1,\ldots, k\}$
such that $B^p_1(x)\sub U$
and
\[
\|f^\wedge(y,v)-f^\wedge(x,v)\|_q\leq \|y-x\|_p\|v_1\|_{p_1}\cdots\|v_k\|_{p_k}
\]
for all $y\in B^p_1(x)$
and all $v=(v_1,\ldots, v_k)\in E_1\times \cdots\times E_k$.
Since $B$ is bounded, we have
\[
C:=\sup\{\|v_1\|_{p_1}\cdots\|v_k\|_{p_k}\colon
v=(v_1,\ldots, v_k)\in B\} \; <\; \infty\,.
\]
Choose $\delta\in\, ]0,1]$ such that $\delta C\leq 1$.
For each $y\in B^p_\delta(x)$,
we get
$\|f^\wedge(y,v)-f^\wedge(x,v)\|_q<\delta C\leq 1$
for each $v\in B$
and thus $f^\wedge(y,v)-f^\wedge(x,v)\in B^q_1(0)\sub W$.
Hence
\[
f(y)-f(x)\in \lfloor B, W\rfloor \quad \mbox{for each $y\in B^p_1(x)$}\,.
\]
entailing that~$f$ is continuous.\\[2.5mm]
Induction step:
Now assume that $r\geq 2$.
Given $x\in U$ and $y\in X$,
there exists
$\ve>0$ such that $x+ \bD_\ve^0 y\sub U$,
where $\bD_r^0:=\{t\in \K\colon |t|<\ve\}$.
Consider
\[
g\colon \bD^0_\ve \times E^k \to F\,,\quad
(t,v)\mto
\left\{
\begin{array}{cl}
\frac{f^\wedge(x+ty,v)-f^\wedge(x,v)}{t} &\,\mbox{\,if $\,t\not=0$;}\\
d(f^\wedge)((x,v),(y,0)) &\,\mbox{\,if $\,t=0$.}
\end{array}
\right.
\]
Then $g$ is $C^{r-1}_\K$ and hence $C^1_\K$,
as a consequence
of \cite[Propositions~7.4 and 7.7]{BGN}.
Since $g(t,v)$ is $k$-linear in~$v$,
it follows that $g^\vee\colon U\to L^k(E,F)_b$ is continuous,
by induction. As a consequence,
\[
\frac{f(x+ty)-f(x)}{t}\;=\;
g^\vee(t)\;\to\; g^\vee(0)
\]
as $t\to 0$, where $g^\vee(0)=d(f^\wedge)((x,\cdot),(y,0))=h^\vee(x,y)$
with
\[
h\colon (U\times E^k)\times W\to F\,,\quad
h((x,y),v):=d(f^\wedge)((x,v),(y,0))\,.
\]
Since $h$ is $C^{r-1}_\K$ and $h((x,y),v)$
is $k$-linear in~$v$,
the map $df=h^\vee$
is $C^{r-2}_\K$, by induction.
Hence $df$ is continuous and thus
$f$ is $C^1_\K$. Now $f$ being $C^1_\K$
with $df$ a $C^{r-2}_\K$-map,
$f$ is $C^{r-1}_\K$.\vspace{1.3mm}

The case $\K=\R$, $r=\omega$.
By \ref{closedsub},
we may assume that $\bL=\R$
(the case $\bL=\C$ then follows).
Given $x\in U$,
let $g\colon U_x\times (E_\C)^k \to F_\C$
be as in the proof of~(a).
Then $g$ is complex $k$-linear in the second
variable and hence
$g^\vee\colon U_x\to L_\C^k(E_\C,F_\C)_b$
is $\C$-analytic,
by the $C^\infty_\C$-case already discussed.
Because the
map $\rho\colon L_\C^k(E_\C,F_\C)_b\to L_\R^k(E,F_\C)_b=(L^k_\R(E,F)_b)_\C$,
$\alpha\mto \alpha|_E$
is continuous $\C$-linear,
the composition $\rho\circ g^\vee$
is $\C$-analytic. But this mapping
extends $f|_{U_x\cap E}$.
Hence $f|_{U_x\cap E}$
is real analytic and hence so is~$f$,
using that the open sets $U_x\cap U$
form an open cover of~$U$.\vspace{1.3mm}

\emph{Formula for the differentials}:
Let $j\in \N$ with $j\leq r$,
$x\in U$, $v\in E_1\times\cdots\times E_k$ and $y_1,\ldots, y_j\in X$.
Exploiting that
$\ev_v\colon L^k_\bL(E_1,\ldots, E_k,F)_c\to F$,
$\beta\mto \beta(v)$
is continuous and linear,
we deduce that
\begin{eqnarray*}
\ev_v(d^jf(x,y_1,\ldots, y_j))
&=&d^j(\ev_v\circ f)(x,y_1,\ldots, y_j)
\, =\, d^j(f^\wedge(\cdot,v))(x,y_1,\ldots,y_j)\\
&=& d^j(f^\wedge)((x,v),(y_1,0),\ldots, (y_j,0))
\end{eqnarray*}
for $f$ as a map to $L^k_\bL(E_1,\ldots, E_k,F)_c$.
If $j\leq r-1$,
the same calculation applies to $f$ as a mapping to
$L^k_\bL(E_1,\ldots, E_k,F)_b$.\,\vspace{2.3mm}\Punkt

\noindent
For the special case of~(a)
when $r=0$ and $X$ as well as $E_1=\cdots=E_k$
are metrizable, see already \cite[Lemma~0.1.2]{Kel}.
\begin{rem}\label{rem-non-lcx}
We mention that the local convexity of $X$
and $E_1,\ldots, E_k$ (resp., $\!E$) in
Lemma~\ref{wippe},
Proposition~\ref{reallyneed1} and Corollary~\ref{dfisC}
and their proofs
is unnecessary
(if one\linebreak replaces continuous seminorms
by continuous gauges in
Lemma~\ref{wippe} and the proof
of\linebreak
Proposition~\ref{reallyneed1}).
Only the local convexity of~$F$ is
essential.
In \cite{IMP}, we used Corollary~\ref{dfisC}
from the current article also
in the case of non-locally convex domains.
\end{rem}
\section{Compositions with hypocontinuous {\boldmath$k$}-linear maps}
We study differentiability properties
of compositions of the form $\beta\circ f$,
where $\beta$ is $k$-linear map
which need not be continuous.
\begin{la}\label{diff-compo-with-hypo}
Let $k\geq 2$ be an integer, $E_1,\ldots, E_k$, $X$, and $F$
be locally convex $\K$-vector spaces,
$\beta\colon E_1\times\cdots\times E_k\to F$
be a $k$-linear map,
$r\in \N_0\cup\{\infty,\omega\}$
and $f\colon U\to E_1\times \cdots\times E_k=:E$
be a $C^r_\K$-map on an open subset $U\sub X$.
Assume that
\begin{itemize}
\item[\rm(a)]
$\beta$ is sequentially continuous
and $X$ is metrizable; or
\item[\rm(b)]
For some $j\in\{2,\ldots, k\}$, the $k$-linear map
$\beta$ is $c$-hypocontinous
in its variables $(j,\ldots, k)$.
Moreover, $X\times X$ is a $k_\R$-space,
or $r=0$ and~$X$ is a $k_\R$-space,
or $(r,\K)=(\infty,\C)$ and~$X$ is a $k_\R$-space.
\end{itemize}
Then $\beta\circ f\colon U\to F$ is a $C^r_\K$-map.
\end{la}
\begin{proof}
The case $r=0$ was treated in Lemma~\ref{compo-is-cts}.
We first assume that
$r\in \N$.\vspace{1mm}

(a) Assuming~(a), let $x\in U$, $y\in X$,
and $(t_n)_{n\in\N}$
be a sequence in $\K\setminus\{0\}$ such that $t_n\to 0$
as $n\to\infty$
and $x+t_n y\in U$ for all $n\in \N$.
Using the components of $f=(f_1,\ldots, f_k)$,
we can write the difference quotient
$\frac{1}{t_n}(\beta(f(x+t_ny))-\beta(f(x)))$
as the telescopic sum
\[
\sum_{\nu=1}^k\beta\Big(f_1(x+t_ny),\ldots, f_{\nu-1}(x+t_ny),
\frac{f_\nu(x+t_ny)-f_\nu(x)}{t_n},f_{\nu+1}(x),\ldots,f_k(x)\Big),
\]
which converges to
\begin{equation}\label{formula-der}
\sum_{\nu=1}^k\beta(f_1(x),\ldots,f_{\nu-1}(x),
df_\nu(x,y),f_{\nu+1}(x),\ldots, f_k(x))=d(\beta\circ f)(x,y)
\end{equation}
as $n\to\infty$, using the sequential continuity of~$\beta$.
By Lemma~\ref{compo-is-cts}, $d(\beta\circ f)$ is continuous,
whence $\beta\circ f$ is $C^1_\K$.
If $r\geq 2$, then
\[
g_\nu\colon U\times X\to E,\quad (x,y)\mto
(f_1(x),\ldots,f_{\nu-1}(x),
df_\nu(x,y),f_{\nu+1}(x),\ldots, f_k(x))
\]
is a $C^{r-1}_\K$-map and
$d(\beta\circ f)=\sum_{\nu=1}^k \beta\circ g_\nu$
is $C^{r-1}_\K$ by induction;
thus $\beta\circ f$ is $C^r_\K$.
If $r=\infty$, the preceding shows that $\beta\circ f$ is $C^s_\K$
for each $s\in\N_0$, whence $\beta\circ f$ is~$C^r_\K$.\vspace{1mm}

(b) If $X\times X$ is a $k_\R$-space,
then $U\times X$ and $U$ are $k_\R$-spaces.
By Lemma~\ref{hypo-seq-cts}, $\beta$ is sequentially
continuous.
The argument from~(a)
shows that $d(\beta\circ f)(x,y)$
exists for all $(x,y)\in U\times X$
and is given by~(\ref{formula-der}).
Thus $d(\beta\circ f)$ is continuous,
by Lemma~\ref{compo-is-cts}, and thus $\beta\circ f$
is $C^1_\K$.
Let $f$ be $C^{r+1}_\K$ now
and assume $\beta\circ f$ is $C^r_\K$
with $r$th differential of the form
\begin{equation}\label{partition-diff}
d^r(\beta\circ f)(x,y_1,\ldots,y_r)
=\sum_{(I_1,\ldots, I_r)}\beta(d^{|I_1|}f_1(x,y_{I_1}),\ldots,
d^{|I_k|}f_k(x,y_{I_k}))
\end{equation}
for $x\in U$ and $y_1,\ldots, y_r\in X$,
where $(I_1,\ldots,I_k)$
ranges through $k$-tuples of (possibly empty) disjoint sets
$I_1,\ldots,I_k$ with $I_1\cup\cdots\cup I_k=\{1,\ldots, r\}$,
and the following notation is used:
For $\nu\in\{1,\ldots,k\}$,
we let $|I_\nu|\in\N_0$
be the cardinality of~$I_\nu$
and define $y_{I_\nu}:=(y_{i_1},\ldots,y_{i_m})\in X^m$
if $i_1<i_2<\cdots< i_m$
are the elements of~$I_\nu$,
abbreviating $m:=|I_j|$
(if $I_\nu$ is empty, the symbol $y_\emptyset$
is to be ignored).
Holding $y_1,\ldots,y_r$ fixed,
we can apply the case $r=1$
to the function $d^rf(\cdot,y_1,\ldots,y_r)$
and find that, for each $x\in U$ and $y_{r+1}\in X$,
the directional derivative at~$x$ in the direction
$y_{r+1}$ exists and is given by
\begin{eqnarray*}
d^{r+1}(\beta\circ f)(x,y_1,\ldots,y_{r+1})\!\! &=& \!\!
\sum_{(I_1,\ldots, I_r)}\sum_{\nu=1}^k
\beta\big(d^{|I_1|}f_1(x,y_{I_1}),\ldots,
d^{|I_{\nu-1}|}f_{\nu-1}(x,y_{I_{\nu-1}}),\\[.3mm]
\!\! & & \!\! d^{|I_\nu|+1}f_\nu(x,y_{I_\nu},y_{r+1}),
d^{|I_{\nu+1}|}f_{\nu+1}(x,y_{I_{\nu+1}}),\ldots,
d^{|I_k|}f_k(x,y_{I_k})\big).
\end{eqnarray*}
Thus also $d^{r+1}(\beta\circ f)$
is of the form~(\ref{partition-diff}),
with~$r+1$ in place of~$r$.
Using Lemma~\ref{compo-is-cts},
we deduce from the preceding formula that the map
\[
U\times E\to F,\quad (x,y)\mto d^{r+1}(\beta\circ f)(x,y,\ldots,y)
\]
is continuous. Thus $d^{r+1}(\beta\circ f)$
is continuous, by Lemma~\ref{cont-via-gateaux},
and thus $\beta\circ f$ is $C^{r+1}_\K$.\vspace{1mm}

If $(r,\K)=(\infty,\R)$,
then $\beta\circ f$ is~$C^s_\R$ for each
$s\in \N_0$ and hence $C^\infty_\R$
(still assuming~(b)).

If $(r,\K)=(\infty,\C)$
and~$X$ is only assumed~$k_\R$,
then $\beta\circ f$ is continuous
by the case $r=0$.
Moreover, the restriction $\beta\circ f|_{U\cap Y}$
is $C^\infty_\C$ for each finite-dimensional
vector subspace $Y\sub X$, by case~(a).
Hence~$f$ is $C^\omega_\C$ (and thus $C^\infty_\C$)
as a mapping to a completion of~$F$
(see \cite[Theorem~6.2]{BaS}).
Then $f$ is also $C^\infty_\C$
as a map to~$F$,
as all of its iterated directional derivatives are in~$F$.

Both in~(a) and~(b),
it remains to consider the case
$(r,\K)=(\omega,\R)$.
Then $f$ admits a $\C$-analytic
extension $\wt{f}\colon \wt{U}\to (E_1)_\C\times\cdots\times (E_k)_\C$,
defined on an open neighbourhood $\wt{U}$ of~$U$
in $X_\C$.
The complex $k$-linear extension
$\beta_\C\colon (E_1)_\C\times\cdots\times (E_k)_\C\to F_\C$
of~$\beta$ is given by
\[
z\mto \sum_{a_1,\ldots,a_k=0}^1
i^{a_1+\cdots+a_k}\beta(x_{1,a_1},\ldots, x_{k,a_k})
\]
for $z=(x_{1,0}+ix_{1,1},\ldots,x_{k,0}+ix_{k,1})$
with $x_{\nu,0}\in E_\nu$ and $x_{\nu,1}\in E_\nu$ for $\nu\in\{1,\ldots,k\}$.
By the latter formula,
$\beta_\C$ is sequentially continuous
in the situation of~(a),
and $c$-hypocontinuous in its arguments
$(j,\ldots,k)$ in the situation of~(b).
The case $(\infty,\C)$ shows that $\beta_\C\circ \wt{f}$
is complex analytic. As this mapping extends
$\beta\circ f$, the latter map is real analytic.
In case~(b), we used here that $X_\C\cong X\times X$
is a $k_\R$-space.
\end{proof}
Also the following variant will be useful.
\begin{la}\label{usehypo}
Let $X_1$, $X_2$, $E_1$, $E_2$ and $F$ be locally convex $\K$-vector
spaces, and\linebreak
$U_1\sub X_1$, $U_2\sub X_2$ be open subsets.
Let $r\in \N_0\cup\{\infty,\omega\}$
and
$\beta\colon E_1\times E_2\to F$ be a $\K$-bilinear
map.
Assume that
$X_1$ is finite-dimensional and~$\beta$ is $c$-hypocontinuous
in its first variable.
Then, for all $C^r_\K$-maps $f_1\colon U_1\to E_1$ and
$f_2\colon U_1\times U_2\to E_2$, also
the following map is $C^r_\K$:
\[
g\colon U_1\times U_2  \to F\, , \quad
(x_1,x_2)\mto \beta(f_1(x_1),f_2(x_1,x_2))\,.
\]
\end{la}
\begin{proof}
We first prove the assertion for
$r\in \N_0$ (from which the case
$r=\infty$ follows).
If $r=0$, we have to show
that~$g$ is continuous.
If
$(x_1,x_2)\in U_1\times U_2$, then $x_1$
has a compact
neighbourhood $W=W_{x_1}$ in~$U_1$.
Then $f_1(W)$ is compact,
and thus $\beta|_{f_1(W)\times E_2}$ is continuous,
by $c$-hypocontinuity.
Hence
$g|_{W\times U_2}=\beta|_{f_1(W)\times E_2}\circ (f_1\circ \pi_W,f_2)$
is continuous, where $\pi_W\colon W\times U_2\to W$
is the projection onto the first factor.
Since $(W_{x_1}\times U_2)_{x_1\in U_1}$ is an
open cover of $U_1\times U_2$,
the map $g$ is continuous.\\[2.5mm]
Since $\beta$ is sequentially continuous
by Lemma~\ref{hypo-seq-cts},
we see as in the preceding proof
that the directional derivative $dg(x,y)$
exists for all $x=(x_1,x_2)\in U_1\times U_2$
and $y=(y_1,y_2)\in X_1\times X_2$,
and is given by
\begin{equation}\label{enabind2}
dg(x,y)=\beta(df_1(x_1,y_1),\, f_2(x))+\beta(f_1(x_1),\, df_2(x,y))\, .
\end{equation}
Note that $(x_1,y_1)\mto f_1(x_1)$ and $df_1$ are $C^{r-1}_\K$-mappings
$U_1\times X_1\to E_1$. Moreover,\linebreak
$((x_1,y_1),(x_2,y_2))\mto f_2(x_1,x_2)$ and
and $((x_1,y_1),(x_2,y_2))\mto df_2((x_1,x_2),(y_1,y_2))$
are $C^{r-1}_\K$-maps $(U_1\times X_1)\times (U_2\times X_2)\to E_2$
(cf.\hspace{-.2mm}~\ref{fanddf}).
By induction, the right hand side
of (\ref{enabind2}) is a $C^{r-1}_\K$-map.
Hence~$g$ is~$C^r_\K$.\\[2.5mm]
The case $(r,\K)=(\omega,\R)$
follows from the case $(\infty,\C)$
as in the preceding proof.
\end{proof}
\begin{rem}
In a setting of differential calculus in which
continuity on products is replaced with
$k$-continuity (as championed by E.\,G.\,F. Thomas),
every bilinear map $\beta$ which
is $c$-hypocontinuous in the second
factor is smooth (see \cite[Theorem 4.1]{Tho});
smoothness of $\beta\circ f$ for a smooth map~$f$
then follows from the Chain Rule
(cf.\ also \cite{Sei}).
Likewise, $\beta$ is smooth in the sense
of convenient differential calculus.
\end{rem}
\section{Differentiability properties of {\boldmath$f^\wedge$}}
For $k=1$, the following result
is essential for our counstructions
of vector bundles.
\begin{prop}\label{reallyneed2}
Let $\bL\in\{\R,\C\}$, $r\in\N_0\cup\{\infty,\omega\}$,
$\K\in\{\R,\bL\}$,
$k\in \N$,
$E_1,\ldots, E_k$ and $F$
be locally convex $\bL$-vector
spaces,
$X$ be a locally convex $\K$-vector space,
and $U\sub X$ be an open subset.
Then the following holds.
\begin{itemize}
\item[\rm(a)]
If
$(X\times E_1\times\cdots \times E_k)\times (X\times E_1\times\cdots\times E_k)$
is a $k_\R$-space,
or $r=0$ and $X\times E_1\times\cdots \times E_k$
is a $k_\R$-space,
or $(r,\K)=(\infty,\C)$
and
$X\times E_1\times\cdots \times E_k$ is a $k_\R$-space,
or all of the vector spaces
$E_1,\ldots, E_k$
are finite dimensional, then
\[
f^\wedge \colon U\times E_1\times\cdots\times E_k\to F\, ,\quad
(x,y_1,\ldots,y_k)\mto f(x)(y_1,\ldots, y_k)
\]
is $C^r_\K$ for each $C^r_\K$-map
$f \colon U\to L_\bL^k(E_1,\ldots, E_k,F)_c$.
\item[\rm(b)]
If
$E:=E_1=E_2=\cdots=E_k$
holds and, moreover,
$(X\times E)\times (X\times E)$ is a $k_\R$-space
or $r=0$ and $X\times E$ is a $k_\R$-space,
or $(r,\K)=(\infty,\C)$ and $X\times E$ is a $k_\R$-space,
then
\[
f^\wedge \colon U\times E^k\to F\, ,\quad
(x,y_1,\ldots,y_k)\mto f(x)(y_1,\ldots, y_k)
\]
is $C^r_\K$ for each $C^r_\K$-map
$f \colon U\to L_\bL^k(E,F)_c$
such that $f(x)$ is a symmetric $k$-linear
map for each $x\in U$.
\item[\rm(c)]
If $X$ is finite dimensional, $k=1$, and $E:=E_1$ is barrelled,
then $f^\wedge\colon U\times E\to F$, $(x,y)\mto f(x)(y)$
is $C^r_\K$ for each $C^r_\K$-map
$f \colon U\to L_\bL(E,F)_c$.
\item[\rm(d)]
If all of the spaces $E_1,\ldots, E_k$ are normable,
then $f^\wedge\colon U\times E_1\times \cdots\times E_k\to F$
is $C^r_\K$
for each
each $C^r_\K$-map
$f \colon U\to L_\bL^k(E_1,\ldots,E_k,F)_b$.
\end{itemize}
\end{prop}
\begin{proof}
Let $\ev\colon L^k_\bL(E_1,\ldots,E_k,F)_c\times E_1\times\cdots\times E_k\to F$
be the evaluation map, which is $c$-hypocontinuous
in its arguments $(2,\ldots,k+1)$ by Lemma~\ref{resulteval}.\vspace{1mm}

(a)
Assuming the respective $k_\R$-property,
the map $f^\wedge =\ev\circ (f\times \id_{E_1\times\cdots\times E_k})$
is $C^r_\K$, by Lemma~\ref{diff-compo-with-hypo}\,(b).
If $E_1,\ldots, E_k$ are finite dimensional,
then
$L^k_\bL(E_1,\ldots, E_k,F)_c$ equals\linebreak
$L^k_\bL(E_1,\ldots, E_k,F)_b$,
whence the conclusion of~(a)
is a special case of~(d).\vspace{1mm}

(b) By Lemma~\ref{diff-compo-with-hypo}\,(b), the map
\[
g\colon U\times E\to F,\quad (x,y)\mto f^\wedge(x,y,\ldots,y)
\]
is $C^r_\K$, as $g=\ev\circ (f\times\delta)$
with $\delta\colon E\to E^k$, $y\mto (y,\ldots,y)$,
which is continuous $\K$-linear.
Then also $f^\wedge$ is $C^r_\K$, by Lemma~\ref{cont-via-gateaux}.\vspace{1mm}

(c) The bilinear map $\ev\colon L_\K(E,F)_c\times E\to F$
is $c$-hypocontinuous
in its first argument, by
Lemma~\ref{resulteval}.
Hence $f^\wedge=\ev\circ (f\times \id_E)$
is $C^r_\K$, by Lemma~\ref{usehypo}.\vspace{1mm}

(d) If $E_1,\ldots, E_k$ are normable, then the
evaluation map
\[
\ve\colon L^k_\bL(E_1,\ldots, E_k,F)_b
\times E_1\times\cdots\times E_k\to F
\]
is continuous $(k+1)$-linear and hence $C^r_\K$,
whence also $f^\wedge=\ve\circ (f\times \id_{E_1\times\cdots\times E_k})$
is $C^r_\K$.
\end{proof}
\begin{rem}\label{metr-or-k-omega}
If $X$ and all of $E_1,\ldots, E_k$
are metrizable,
then the topological space\linebreak
$(X\times E_1\times \cdots\times E_k)\times (X\times E_1\times\cdots\times E_k)$
is metrizable and hence a $k$-space.
If $X$ and all of $E_1,\ldots, E_k$
are $k_\omega$-spaces,
then also
$(X\times E_1\times \cdots\times E_k)\times (X\times E_1\times\cdots\times E_k)$
is a $k_\omega$-space and hence a $k$-space.
In either case, we are in the situation of~(a).
\end{rem}
\section{Infinite-dimensional vector bundles}
In this section,
we provide foundational material
concerning vector bundles
modeled on locally convex spaces
(cf.\ also~\cite{DMS} and \cite[Chapter~3]{GaN}).
Notably, we discuss the description of vector
bundles via cocycles,
and define equivariant vector bundles.\\[2.5mm]
Let $\bL\in \{\R,\C\}$,
$\K\in\{\R,\bL\}$,
and $r\in \N_0\cup\{\infty,\omega\}$.
The word ``manifold'' always refers
to a manifold modeled on a locally convex
space. Likewise,
the Lie groups we consider
need not have finite dimension.
\begin{defn}\label{defnbun}
Let~$M$ be a~$C^r_\K$-manifold
and~$F$ be a locally convex
$\bL$-vector space.
An \emph{$\bL$-vector bundle of class $C^r_\K$
over~$M$, with typical fibre~$F$},
is a $C^r_\K$-manifold~$E$, together
with a surjective $C^r_\K$-map
$\pi\colon E\to M$
and endowed with an $\bL$-vector space structure
on each fibre $E_x:=\pi^{-1}(\{x\})$,
such that, for
each $x\in M$, there exists an open neighbourhood
$U\sub M$ of $x$
and a $C^r_\K$-diffeomorphism
\[
\psi\colon \pi^{-1}(U)\to U\times F
\]
(called a ``local trivialization'')
such that $\psi(E_y)=\{y\}\times F$
for each $y\in U$
and the map
$\pr_F\circ \psi|_{E_y}\colon E_y\to F$
is $\bL$-linear,\footnote{And hence an isomorphism of topological
vector spaces, if we give $E_y$
the topology induced by~$E$.}
where $\pr_F\colon U\times F\to F$
is the projection.
\end{defn}
\begin{numba}\label{defncocyc}
In the situation of Definition~\ref{defnbun},
let $(\psi_i)_{i\in I}$
be an atlas of local trivializations for~$E$,
{\em i.e.},
a family of local trivializations
\[
\psi_i\!: \pi^{-1}(U_i)\to U_i\times F
\]
of~$E$
whose domains $U_i$ cover~$M$.
Then, given $i,j\in I$, we have
\[
\psi_i(\psi^{-1}_j(x,v))=(x,g_{ij}(x)(v))
\]
for $x\in U_i\cap U_j$, $v\in F$,
for some function
\[
g_{ij}\!: U_i\cap U_j\to \GL(F)\sub L(F).
\]
Here
\[
G_{ij}\!:(U_i\cap U_j)\times F \to F,\;\;\;\;
(x,v)\mto g_{ij}(x)(v)\]
is $C^r_\K$,
as $\psi_i(\psi^{-1}_j(x,v))=(x,G_{ij}(x,v))$ is
$C^r_\K$ in $(x,v)\in (U_i\cap U_j)\times F$.
By Proposition~\ref{reallyneed1}, $g_{ij}\colon U_i\cap U_j\to
L(F)_c$ is a $C^r_\K$-map,
and as a map to $L(F)_b$, it is at least~$C^{r-1}_\K$
(if $r\geq 1$).
Note that the ``transition maps''
$g_{ij}$ satisfy the ``cocycle conditions''
\[
\left\{
\begin{array}{ll}
(\forall i\in I)\,(\forall x\in U_i) &
g_{ii}(x)=\id_F\quad \mbox{and}\\
(\forall i,j,k \in I)\;(\forall x\in U_i\cap U_j\cap U_k) &
g_{ij}(x)\circ g_{jk}(x)=g_{ik}(x)\, .
\end{array}
\right.
\]
\end{numba}
\begin{prop}\label{standfact}
Let $\bL\in\{\R,\C\}$, $\K\in \{\R,\bL\}$. Assume that
\begin{itemize}
\item[\rm(a)]
$M$ is a $C^r_\K$-manifold modeled on a locally convex $\K$-vector
space~$Z$;
\item[\rm(b)]
$E$ is a set and $\pi\colon E\to M$ a surjective map;
\item[\rm(c)]
$F$ is a locally convex $\bL$-vector space;
\item[\rm(d)]
$(U_i)_{i\in I}$ is an open cover of~$M$;
\item[\rm(e)]
$(\psi_i)_{i\in I}$
is a family of bijections $\pi^{-1}(U_i)\to U_i\times F$
such that $\psi_i(\pi^{-1}(\{x\}))=\{x\}\times F$
for all $x\in U_i$;
\item[\rm(f)]
$g_{ij}(x)(v):=\pr_F(\psi_i(\psi_j^{-1}(x,v)))$
depends $\bL$-linearly on $v\in F$, for all $i,j\in I$,
$x\in U_i\cap U_j$;
\item[\rm(g)]
$G_{ij}\!: (U_i\cap U_j)\times F\to F$,
$G_{ij}(x,v):=g_{ij}(x)(v)$ is a $C^r_\K$-map.
\end{itemize}
Then there is a unique
$\bL$-vector bundle structure of class $C^r_\K$
on~$E$ making $\psi_i$
a local trivialization for each $i\in I$.
\end{prop}
\begin{proof}
For $i,j\in I$, let $\pr_{ij}\colon (U_i\cap U_j)\times F\to U_i\cap U_j$
be the projection onto the first component.
As the maps
\[
\psi_i\circ
\psi_j^{-1}|_{(U_i\cap U_j)\times F}
= (\pr_{ij}, G_{ij})
\]
are $C^r_\K$,
there is a uniquely determined $C^r_\K$-manifold
structure on~$E$ making $\psi_i$
a $C^r_\K$-diffeomorphism
for each $i\in I$.
Given $x\in M$, we pick $i\in I$ with
$x\in U_i$; we give $E_x:=\pi^{-1}(\{x\})$
the unique $\bL$-vector space structure
making the bijection $\pr_F\circ \psi_i|_{E_x}\colon E_x\to F$
an isomorphism of vector spaces.
It is easy to see that the vector space structure on~$E_x$
is independent of the choice of~$\psi_i$,
and it is easily verified that we have turned
$E$ into an $\bL$-vector
bundle of class~$C^r_\K$ with the asserted properties.
\end{proof}
\begin{rem}\label{stand2}
Let $M$ be a $C^r_\K$-manifold,
$F$ be a locally convex $\bL$-vector space,
$(U_i)_{i\in I}$ be an open cover
of~$M$, and $(g_{ij})_{i,j\in I}$
be a family of maps $g_{ij}\colon U_i\cap U_j\to \GL(F)$
satisfying the cocycle conditions
and such that
\[
G_{ij}\colon (U_i\cap U_j)\times F\to F\, ,\quad
(x,v)\mto g_{ij}(x)(v)
\]
is $C^r_\K$,
for all $i,j\in I$.
Using Proposition~\ref{standfact},
the usual construction familiar from
the finite-dimensional case provides
an $\bL$-vector bundle $\pi\colon E\to M$
of class $C^r_\K$,
with typical fibre~$F$,
and a family $(\psi_i)_{i\in I}$ of local trivializations
$\pi^{-1}(U_i)\to U_i\times F$,
whose associated transition maps are the given $g_{ij}$'s.
The bundle~$E$ is unique up to canonical isomorphism.
\end{rem}
Combining Proposition~\ref{standfact}
and Proposition~\ref{reallyneed2}, we obtain:
\begin{cor}\label{cornew}
Retaining the hypotheses {\rm (a)--(f)}
from Proposition~{\rm \ref{standfact}}
but omitting {\rm (g)},
consider the following conditions:
\begin{itemize}
\item[\rm $\mbox{(g)}'$]
$g_{ij}(x)\in L(F)$
for all $i,j\in I$, $x\in U_i\cap U_j$,
and $g_{ij}\colon U_i\cap U_j\to L(F)_c$
is $C^r_\K$;
\item[\rm $\mbox{\,(g)}''$]
$g_{ij}(x)\in L(F)$
for all $i,j\in I$, $x\in U_i\cap U_j$,
and $g_{ij}\colon U_i\cap U_j\to L(F)_b$
is $C^r_\K$;
\item[\rm(i)]
$(Z\times F)\times (Z\times F)$
is a $k_\R$-space,
or $r=0$ and $Z\times F$ is a $k_\R$-space,
or $(r,\K)=(\infty,\C)$ and $Z\times F$
is a $k_\R$-space;
\item[\rm(ii)]
$\dim(M)<\infty$ and $F$ is barrelled;
\item[\rm(iii)]
$F$ is normable.
\end{itemize}
If {\rm$\mbox{(g)}'$}
holds as well as {\rm (i)} or {\rm(ii)},
then the conclusions of
Proposition~{\rm \ref{standfact}}
remain valid.
They also remain valid if
{\rm$\mbox{(g)}''$} and {\rm (iii)} hold.\,\Punkt
\end{cor}
Example~\ref{nodual} below
shows that Conditions
(a)--(f) and $\mbox{(g)}'$ alone are not sufficient
for the conclusion of Proposition~\ref{standfact},
without extra conditions on~$Z$ and~$F$.
Note that (i) is satisfied if both $Z$ and $F$
are metrizable, or both $Z$ and $F$
are $k_\omega$-spaces.\vspace{.5mm}
\begin{center}
{\bf Equivariant vector bundles}
\end{center}
Beyond vector bundles,
we shall discuss
\emph{equivariant} vector bundles
in the following, i.e.,
vector bundles together
with an action
of a (finite- or infinite-dimensional)
Lie group~$G$.
Choosing $G=\{e\}$ as a trivial group,
we obtain results about ordinary vector
bundles (without a group action),
as a special case.\\[2.5mm]
For the remainder of this section, and also in Section~\ref{sec-complete},
let
$\bL\in\{\R,\C\}$, $\K\in\{\R,\bL\}$,
$s\in \{\infty,\omega\}$, and $r\in \N_0\cup\{\infty,\omega\}$
with $r\leq s$.
Let $G$ be
a $C^s_\K$-Lie group
(modeled on a locally convex $\K$-vector space~$Y$)
and $M$ be a $C^r_\K$-manifold.
We assume that a $C^r_\K$-action
\[
\alpha\colon G\times M\to M
\]
is given.
Then $(M,\alpha)$
is called a \emph{$G$-manifold of class $C^r_\K$.}
\begin{defn}\label{def-equi}
An
\emph{equivariant $\bL$-vector bundle of class $C^r_\K$}
over a $G$-manifold $(M,\alpha)$ of class $C^r_\K$
is an $\bL$-vector bundle $\pi\colon E\to M$ of class $C^r_\K$,
together with a
$C^r_\K$-action
\[
\beta\colon G\times E\to E
\]
such that $\beta(g,E_x)\sub E_{\alpha(g,x)}$
for all $(g,x)\in G\times M$, and
$\beta(g,\cdot)|_{E_x}
\colon
E_x\to E_{\alpha(g,x)}$ is $\bL$-linear.
\end{defn}
In other words, $\beta(g,\cdot)$
takes fibres linearly to fibres and coincides
with $\alpha(g,\cdot)$ on the zero
section. The mapping~$\pi$ is then equivariant
in the sense that $\alpha\circ (\id_G\times \pi)=\pi\circ\beta$.
\begin{example}\label{acttan}
If $M$ is a $G$-manifold of class $C^r_\K$,
with $r\geq 1$, then the tangent bundle $TM$ is an
equivariant
$\bL$-vector bundle of class $C^{r-1}_\K$ in a natural way,
with $\bL:=\K$.
In fact, the action $\alpha\colon G\times M\to M$
has a tangent map
$T\alpha\colon T(G\times M)\to TM$ which is $C^{r-1}_\K$.
Let $0_G\colon G\to TG$ be the $0$-section.
Identifying $T(G\times M)$ with $TG\times TM$
in the usual way, we obtain a $C^{r-1}_\K$-map
$\beta\colon G\times TM\to TM$ via
\[
\beta:=(T\alpha)\circ (0_G\times \id_{TM})\, .
\]
It is easy to see that
$\beta(g,v)=T_x(\alpha(g,\cdot))(v)\in
T_{\alpha(g,x)}M$ for $g\in G$
and $v\in T_xM$,
whence $\beta(g,T_xM)\sub T_{\alpha(g,x)}M$
and $\beta(g,\cdot)|_{T_xM}
= T_x(\alpha(g,\cdot))$.
Clearly $\beta$ is an action
making $TM$ an
equivariant $\K$-vector bundle of class $C^{r-1}_\K$ over
the $G$-manifold~$M$.
\end{example}
\begin{center}
{\bf Induced action on an invariant subbundle}
\end{center}
Given an
$\bL$-vector bundle $\pi\!: E\to M$ of class~$C^r_\K$,
with typical fibre~$F$,
we call a subset $E_0\sub E$
a {\em subbundle\/}
if there exists a sequentially closed
$\bL$-vector subspace $F_0\sub F$
such that for each $x\in M$ there exists
a local trivialization $\psi\!: \pi^{-1}(U)\to U\times F$
of~$E$
such that $\psi(E_0\cap \pi^{-1}(U))=U\times F_0$.
It readily follows from
\cite[Lemma~1.3.19]{GaN}
and \cite[Proposition~2.11]{RES}
that there is a unique
$\bL$-vector bundle structure of class $C^r_\K$ on
$\pi|_{E_0}\!: E_0\to M$
making $\psi|_{\pi^{-1}(U)\cap E_0}\!:
\pi^{-1}(U)\cap E_0\to U\times F_0$
a local trivialization of~$E_0$,
for each local trivialization~$\psi$
as before.
Then the inclusion map $E_0\to E$ is $C^r_\K$, and a
mapping $N\to E$ from a $C^r_\K$-manifold~$N$ to~$E$
with image in~$E_0$ is~$C^r_\K$
as a mapping to~$E$
if and only if its co-restriction to~$E_0$ is~$C^r_\K$,
by the facts just cited.
In the preceding situation,
suppose that a $C^s_\K$-Lie group~$G$ acts~$C^s_\K$ on~$M$
and $E$ is an
equivariant vector bundle of class~$C^r_\K$
with respect to the action $\beta\!: G\times E\to E$.
If $E_0$ is invariant under the $G$-action, {\em i.e.}
if
$\beta(G\times E_0)\sub E_0$,
as a special case of the preceding observations
we deduce from the $C^r_\K$-property
of~$\beta$ that $\beta|_{G\times E_0}$ and
thus also $\beta|_{G\times E_0}\colon G\times E_0\to E_0$
is~$C^r_\K$. Summing up:
\begin{prop}\label{propsubbun}
If $E$ is an
equivariant $\bL$-vector bundle of class $C^r_\K$
over a $G$-manifold~$M$,
then the action induced
on any $G$-invariant subbundle~$E_0$
is $C^r_\K$
and thus makes the latter an
equivariant $\bL$-vector bundle of class $C^r_\K$.\, \Punkt
\end{prop}
\section{Completions of vector bundles}\label{sec-complete}
Let $\pi\colon  E\to M$ be an
equivariant $\bL$-vector bundle of class $C^r_\K$,
as in Definition~\ref{def-equi},
with typical fibre~$F$ and $G$-actions
$\alpha\colon G\times M\to M$ and
$\beta\colon G\times E\to E$.
Assume that $r\geq 1$.
Our goal is to complete the fibre of the bundle,
i.e., to find a $G$-equivariant
vector bundle~$\wt{E}$ whose typical
fibre is a completion of the locally convex space~$F$,
and which contains~$E$ as a dense subset.
\begin{numba}\label{the-setting-compl}
Let $\wt{F}$ be a completion of~$F$
such that $F\sub \wt{F}$
and, for each $x\in M$, let $\wt{E}_x$ be
a completion of $E_x$
such that $E_x\sub \wt{E}_x$.
We may assume that the sets $\wt{E}_x$
are pairwise disjoint for $x\in M$.
Consider the (disjoint) union
\[
\wt{E}\; :=\; \bigcup_{x\in M}\wt{E}_x\,.
\]
We shall turn $\wt{E}$
into an equivariant vector bundle.
Consider the map
$\wt{\beta}\colon G\times \wt{E}\to \wt{E}$,
defined using the continuous
extension $(\beta(g,\cdot)|_{E_x})\wt{\;}\colon \wt{E}_x\to \wt{E}_{\alpha(g,x)}$
of the linear mapping\linebreak
$\beta(g,\cdot)|_{E_x}\colon E_x\to E_{\alpha(g,x)}$
via
\[
\wt{\beta}(g,v)
:=
(\beta(g,\cdot)|_{E_x})\wt{\;}(v)
\]
for $g\in G$, $x\in M$, and $v \in \wt{E}_x$.
It is clear that $\wt{\beta}$
makes $\wt{E}$ a $G$-set.
Let
$\wt{\pi}\colon \wt{E}\to M$ be the map
taking elements
from $\wt{E}_x$ to~$x$.
Then $\wt{\pi}$ is $G$-equivariant.
If
$\psi\colon \pi^{-1}(U)\to U\times F$ is a local trivialization
of~$E$ and $\pr_F\colon U\times F\to F$, $(x,y)\mto y$, we define
\begin{equation}\label{wtpsi2}
\wt{\psi}\colon
\wt{\pi}^{-1}(U)\to U\times \wt{F}\,,\quad
\wt{E}_x \ni v \mto (x, (\pr_F\circ \psi|_{E_x})\wt{\;}(v))\,.
\end{equation}
Then the following holds:
\end{numba}
\begin{prop}\label{complbdl}
$(\wt{E},\wt{\beta})$ can be made
an equivariant $\bL$-vector bundle of class
$C^{r-1}_\K$ over the $G$-manifold~$M$,
such that $\wt{\psi}$
is a local trivialization of~$\wt{E}$
for each local trivialization $\psi$ of~$E$.
\end{prop}
\begin{rem}\label{strength-compl}
Omitting the hypothesis that $r\geq 1$,
assume instead that~$E$ is a equivariant $\bL$-vector bundle
\emph{of class $LC^r_\K$}.
That is, both $E$ and $M$
are $LC^r_\K$-manifolds (each admitting an atlas
with transition maps of class $LC^r_\K$),
a family of local trivializations can be chosen
with $LC^r_\K$-transition maps,
and the $G$-actions on~$E$ and~$M$
are $LC^r_\K$.
Then also
$\wt{E}$ is an equivariant vector bundle
of class $LC^r_\K$ $($and hence of class $C^r_\K)$.
\end{rem}
{\bf Extension of differentiable maps to subsets of the completions.}
\,To enable the proof of Proposition~\ref{complbdl},
we need to discuss conditions ensuring
that a $C^r$-map\linebreak
$f\colon E\supseteq U\to F$ (with
locally convex spaces $E$ and~$F$)
can be extended to a $C^r$-map $\wt{U}\to \wt{F}$
on an open subset of the completion $\wt{E}$ of~$E$,
or at least to a $C^{r-1}$-map.
Although this is not possible
in general, it is possible if~$F$
is normed and~$r$ is finite.
This will be sufficient  for
our ends. The natural framework for the discussion
of the problem are not $C^r$-maps, but Lipschitz
differentiable maps,
as in Definition~\ref{lip-diffble}.
\begin{prop}\label{rwafct}
Let $E$ be a locally convex
$\K$-vector space,
$(F,\|\cdot\|)$ be a Banach space over~$\K$,
$U\sub E$ be open and
$f\colon U\to F$
be an $LC_\K^r$-map,
where $r\in \N_0$.
Let $\wt{E}$
be a completion of~$E$
such that $E\sub \wt{E}$.
Then $f$ extends to an
$LC^r_\K$-map
$\wt{f}\colon \wt{U}\to F$
on an open subset $\wt{U}\sub \wt{E}$
which contains~$U$ as a dense subset.
\end{prop}
The following lemma enables an inductive proof
of Proposition~\ref{rwafct}.
\begin{la}\label{enblinduc}
Let $k\in \N$,
$X$ be a locally convex $\K$-vector space, and
$E_1,\ldots, E_k,F$
be locally convex $\bL$-vector spaces,
with completions $\wt{X}$, $\wt{E}_1,\ldots, \wt{E}_k$
and $\wt{F}$.
Let $U\sub X$ be open and $f\colon U\times E_1\times \cdots\times E_k\to F$
be a map such that $f^\vee(x):=f(x,\cdot)\colon
E_1\times \cdots\times E_k\to F$
is $k$-linear over~$\bL$ for each $x\in U$.
Assume that there exists an $LC^r_\K$-map
$h \colon W\to \wt{F}$ which extends~$f$,
defined on
an open set
$W\sub \wt{X}\times \wt{E}_1\times \cdots\times \wt{E}_k$
in which $U\times E_1\times\cdots\times E_k$
is dense.
Then there exists an $LC^r_\K$-map
\begin{equation}\label{desidat}
\wt{f}\colon \wt{U}\times \wt{E}_1\times\cdots\times \wt{E}_k\to
\wt{F}
\end{equation}
which extends~$f$, for some open subset
$\wt{U}\sub \wt{E}$ in which~$U$ is dense.
The maps $(\wt{f})^\vee(x):=\wt{f}(x,\cdot)\colon
\wt{E}_1\times\cdots\times \wt{E}_k\to \wt{F}$
are $k$-linear over~$\bL$, for each $x\in \wt{U}$.
\end{la}
\begin{proof}
For each $x\in U$, there exist
an open neighbourhood $V_x$ of~$x$ in~$\wt{X}$
and a balanced, open $0$-neighbourhood
$Q_x\sub \wt{E}_1\times\cdots\times \wt{E}_k$
such that $V_x\times Q_x\sub W$.
After shrinking~$V_x$, we may assume that $X\cap V_x=U$,
whence $U\cap V_x=X\cap V_x$ is dense in~$V_x$.
Given $z\in \bL$ such that $|z|\leq 1$,
consider the map
\[
V_x\times Q_x\to \wt{F}\,,\quad (y,v)\mto
h(y,zv)-z^kh(y,v)\,.
\]
This map vanishes, because it is continuous
and vanishes on the dense subset\linebreak
$(V_x\cap X)\times (Q_x\cap (E_1\times\cdots\times E_k))$.
As a consequence, we obtain a well-defined
map
\[
f_x\colon V_x\times
\wt{E}_1\times\cdots\times \wt{E}_k\to \wt{F}\,,
\quad
(y,v)\mto
z^{-k} h(y,zv)
\]
for $y\in V_x$, $v\in
\wt{E}_1\times\cdots\times \wt{E}_k$
and $z\in \bL\setminus\{0\}$ with
$zv\in Q_x$.
As $f_x(y,v)=z^{-k} h(y,zv)$
is $LC^r_\K$ in $(y,v)\in V_x\times z^{-1}Q_x$ and these
sets form an open cover of $V_x\times
\wt{E}_1\times\cdots\times \wt{E}_k$,
we see that $f_x$ is~$LC^r_\K$.
Given $x,y\in U$,
the set $U\cap V_x\cap V_y=X\cap V_x\cap V_y$
is dense in the open set $V_x\cap V_y\sub \wt{X}$.
Since $f_x$, $f_y$, and $f$ coincide
on the set $(U\cap V_x\cap V_y)\times E_1\times \cdots\times E_k$,
it follows that the continuous maps
$f_x$ and $f_y$ coincide on the set
$(V_x\cap V_y)\times \wt{E}_1\times\cdots\times \wt{E}_k$
in which the former set is dense.
Hence, setting $\wt{U}:=\bigcup_{x\in U}V_x$,
a well-defined map
$\wt{f}$ as in~(\ref{desidat}) is obtained
if we set
\[
\wt{f}(y,v):=f_x(y,v)\quad \mbox{if $\,x\in U$, $y\in V_x$ and $v\in
\wt{E}_1\times\cdots\times \wt{E}_k$.}
\]
The final assertion follows by continuity
from the linearity of the maps $f^\vee(x)$
for $x\in U$.
\end{proof}
{\bf Proof of Proposition~\ref{rwafct}.}
We proceed by induction on $r\in \N_0$.\\[2.5mm]
\emph{The case $r=0$.}
Given $x\in U$, there exists a continuous seminorm
$q$ on~$E$ such that $B^q_1(x)\sub U$ and
\begin{equation}\label{ensfactor}
\|f(z)-f(y)\|\leq q(z-y)\quad\mbox{for all $\,y,z\in B^q_1(x)$.}
\end{equation}
Then $N_q:=\{y\in E\colon q(y)=0\}$
is a closed vector subspace of~$E$
and $\|y+N_q\|_q:=q(y)$ for $y\in E$
defines a norm on $E_q:=E/N_q$
making the map $\alpha_q\colon E\to E_q$, $y\mto y+N_q$
continuous linear.
By (\ref{ensfactor}),
we have $\|f(z)-f(y)\|=0$
for all $y,z\in B^q_1(x)$ such that $y-z\in N_q$.
Hence
\[
h\colon \alpha_q(B^q_1(x))\to F\,,\quad
y+N_q\mto f(y)
\]
is a well-defined map.
Note that $\alpha_q(B^q_1(x))$ is the
open ball $B:=\{y\in E_q\colon \|y-\alpha_q(x)\|_q<1\}$
in~$E_q$.
Let $\wt{E}_q$ be the completion of
the normed space $E_q$;
the extended norm will again be denoted by $\|.\|_q$.
Applying (\ref{ensfactor})
to representatives,
we see that
\[
\|h(z)-h(y)\|\leq \|z-y\|_q\quad \mbox{for all $\, y,z\in B$.}
\]
Hence $h$ satisfies a global Lipschitz
condition (with Lipschitz constant~$1$),
and hence $h$ is uniformly
continuous, entailing that
$h$ extends uniquely
to a uniformly continuous map
\[
\wt{h}\colon \wt{B}\to F
\]
on the corresponding open ball $\wt{B}$ in $\wt{E}_q$.
Then $\|\wt{h}(z)-\wt{h}(y)\|\leq \|z-y\|_q$
for all $y,z\in \wt{B}$, by continuity.
Let $\wt{\alpha}_q\colon \wt{E}\to
\wt{E}_q$ be the continuous extension of
the continuous linear map~$\alpha_q$.
Then
$V_x:=(\wt{\alpha}_q)^{-1}(\wt{B})$
is an open neighbourhood of~$x$
in~$\wt{E}$
such that $V_x\cap E=B^q_1(x)\sub U$.
Moreover,
$f_x:=\wt{h}\circ \wt{\alpha}_q|_{V_x}$
is a continuous map extending $f|_{V_x\cap E}$,
which furthermore
satisfies
\begin{equation}\label{fxlipcts}
\|f_x(z)-f_x(y)\|\,\leq \, \wt{q}(z-y)\quad\mbox{for all $\, y,z\in V_x$,}
\end{equation}
where we use the continuous seminorm
$\wt{q}:=\|.\|_q\circ \wt{\alpha}_q\colon \wt{E}\to [0,\infty[$
extending~$q$.
Then
\[
\wt{U}:=\bigcup_{x\in U}V_x
\]
is an open subset of~$\wt{E}$
and $E\cap \wt{U}=U$
is dense in~$\wt{U}$.
Given $x,y\in U$,
the set $U\cap V_x\cap V_y=E\cap V_x\cap V_y$
is dense in the open set $V_x\cap V_y\sub \wt{E}$.
Since
\[
f_x|_{U\cap V_x\cap V_y}\;=\;
f|_{U\cap V_x\cap V_y}
\; =\; f_y|_{U\cap V_x\cap V_y}\,,
\]
it follows that
$f_x|_{V_x\cap V_y}
=f_y|_{V_x\cap V_y}$.
Hence
\[
\wt{f} \colon \wt{U} \to F\,,\quad
z \mto f_x(z)\quad \mbox{for $\, x\in U$
such that $z\in V_x$}
\]
is a well-defined map.
Since $\wt{f}|_{V_x}=f_x$ is~$LC^0_\K$
for each $x\in U$
(by (\ref{fxlipcts})),
the map $\wt{f}$ is~$LC^0_\K$.
Furthermore, $\wt{f}$ extends~$f$ by construction.\\[2.5mm]
Induction step. If $f$ is $LC^{r+1}_\K$,
then $f$ extends to an $LC^0_\K$-map
$\wt{f}\colon \wt{U}\to F$ on an open
subset $\wt{U}\sub \wt{E}$
such that $\wt{U}\cap E=U$,
and $df\colon U\times E\to F$
extends to an $LC^r_\K$-map
$h\colon W\to F$
on an open subset $W$ of $\wt{E}\times \wt{E}$,
by induction.
Using Lemma~\ref{enblinduc},
we find an open neighbourhood~$V$ of~$U$ in~$\wt{E}$
and an $LC^r_\K$-map $g\colon V\times \wt{E}\to F$
which extends $df$.
After replacing $\wt{U}$ and~$V$
with their intersection, we may assume
that $\wt{U}=V$.
If $x_0\in \wt{U}$ and $y_0\in \wt{E}$,
there exist
open neighbourhoods $Q$ of~$x_0$
and $P$ of~$y_0$ in~$\wt{E}$,
and $\ve>0$
such that $Q+\bD_\ve P \sub \wt{U}$.
Then the map
\[
\ell\colon Q\times P\times \bD_\ve \to F,
\quad
(x,y,t)\mto \int_0^1 g(x+sty,y)\,ds
\]
is continuous, being given
by a parameter dependent
weak integral with continuous integrand.
For $(x,y,t)$
in the dense subset
$(Q\cap E)\times (P\cap E)\times (\bD_\ve\setminus\{0\})$
of
$Q\times P \times (\bD_\ve\setminus\{0\})$,
the Mean Value Theorem implies that
\[
\ell(x,y,t)\;=\;\frac{f(x+ty)-f(x)}{t}\;=\;
\frac{\wt{f}(x+ty)-\wt{f}(x)}{t}.
\]
Then
$\ell(x,y,t)
=\frac{\wt{f}(x+ty)-\wt{f}(x)}{t}$
for all
$(x,y,t)\in Q\times P \times (\bD_\ve\setminus\{0\})$,
by continuity. Thus
\[
\frac{f(x_0+ty_0)-f(x_0)}{t}\;=\;
\ell(x_0,y_0,t)\;\to\; \ell(x_0,y_0,0)\;=\; g(x_0,y_0)
\]
as $t\to 0$.
Hence $d\wt{f}(x_0,y_0)=g(x_0,y_0)$.
Since $g$ is $LC^r_\K$, it follows that $\wt{f}$
is $LC^{r+1}_\K$.\,\vspace{2.4mm}\Punkt

\noindent
The conclusion of Proposition~\ref{rwafct} becomes
false in general if the Banach space~$F$ is replaced by a
complete locally convex space.
In fact, there exists a smooth map
$E \to(\ell^1)^\Omega$ from
a proper, dense vector subspace~$E$ of~$\ell^1$
to a suitable power of~$\ell^1$
which has no continuous extension
to $E\cup\{x\}$ for any $x\in \ell^1\setminus E$
(see Appendix~\ref{appB}).
Nonetheless,
we have the following result.
\begin{prop}\label{extend-second-multi}
Let $k\in \N$,
$X$ be a locally convex $\K$-vector space,
and $E_1,\ldots, E_k,F$
be locally convex $\bL$-vector spaces,
with completions $\wt{X}$, $\wt{E}_1,\ldots, \wt{E}_k$
and $\wt{F}$, respectively.
Let $U\sub X$ be open
and $f\colon U\times E_1\times \cdots\times E_k\to F$
be a mapping such that $f^\vee(x):=f(x,\cdot)\colon$
$E_1\times \cdots\times E_k\to F$
is $k$-linear over~$\bL$ for each $x\in U$.
If $f$ is $LC^r_\K$
for some $r\in \N_0\cup\{\infty\}$
$($resp., $C^r_\K$ for some
$r\in \N\cup\{\infty,\omega\})$,
then there exists a unique map
\begin{equation}\label{desidat2}
\wt{f}\colon U \times
\wt{E}_1\times\cdots\times \wt{E}_k\to
\wt{F}
\end{equation}
which is $LC^r_\K$ $($resp., $C^{r-1}_\K)$
and extends~$f$.
The maps $\wt{f}^\vee(x):=\wt{f}(x,\cdot)\colon
\wt{E}_1\times\cdots\times \wt{E}_k\to \wt{F}$
are $k$-linear over~$\bL$, for each $x\in U$.
\end{prop}
\begin{proof}
Abbreviate $E:=E_1\times\cdots\times E_k$
and $\wt{E}:=\wt{E}_1\times\cdots\times \wt{E}_k$.
Assume first that $r\not=\omega$.
Since $LC^r_\K$-maps are continuous
and $U\times E$
is dense in $U\times\wt{E}$,
there is at most one map $\wt{f}$
with the asserted properties. We may therefore
assume that $r\in \N_0$.
We may also assume that $F$ is complete.
Then $F=\pl\, F_j$ for some
projective system $((F_j)_{j\in J}, (p_{ij})_{i\leq j})$
of Banach spaces $F_j$ and continuous linear maps
$p_{ij}\colon F_j\to F_i$,
with limit maps $p_j\colon F\to F_j$.
We claim that $p_j\circ f\colon U\times
E \to F_j$
has an $LC^r_\K$-extension $g_j:=(p_j\circ f)\wt{\;}\colon
U\times\wt{E}\to F_j$, for each $j\in J$.
If this is true, then
$p_{ij}\circ g_j=g_i$ for $i\leq j$,
by uniqueness
of continuous extensions.
Hence, by the universal property
of the projective limit,
there exists a unique map $\wt{f}\colon U\times\wt{E}\to F$ such that
$p_j\circ \wt{f}=g_j$.
Then $p_j\circ \wt{f}|_{U\times E}=g_j|_{U\times E}=p_j\circ f$
and hence $\wt{f}|_{U\times E}=f$. Furthermore, $\wt{f}$ is $LC^r_\K$,
by Lemma~\ref{lem-lipdiff}\,(d).
To prove the claim,
note that Proposition~\ref{rwafct} yields
an $LC^r_\K$-extension $h_j\colon W_j\to F_j$
of $p_j\circ f$ to an open subset $W_j\sub \wt{X}\times \wt{E}$
which contains $U\times E$ as a dense subset.
Now Lemma~\ref{enblinduc}
yields an open subset $U_j\sub \wt{X}$
in which $U$ is dense,
and an $LC^r_\K$-extension $e_j\colon U_j\times \wt{E}\to F_j$
of~$p_j\circ f$. Then $g_j:=e_j|_{U\times\wt{E}}$ is as desired.\vspace{1mm}

We now consider the case $(r,\K)=(\omega,\R)$. If $\bL=\C$,
by density of $U\times E$ in $U\times\wt{E}$,
for any real analytic
extension $\wt{f}\colon U\times \wt{E}\to\wt{F}$
and $x\in U$, the map $\wt{f}(x,\cdot)$
will be $k$-linear over~$\bL$.
We may therefore assume that $\bL=\R$.
Let $h\colon W\to F_\C$
be a $\C$-analytic extension
of~$f$, defined on an open subset
$W\sub X_\C\times E_\C$ such that $U\times E\sub W$.
For each $x\in U$,
there exist an open $x$-neighbourhood $U_x\sub U$
and balanced open $0$-neighbourhoods $V_x\sub X$ and $W_x\sub E_\C$
such that $(U_x+iV_x)\times W_x\sub W$.
We claim that there exists a $\C$-analytic
map $g_x \colon (U_x+iV_x)\times E_\C\to F_\C$
such that $g_x|_{U_x\times E}=f|_{U_x\times E}$.
For $x,y\in U$,
the intersection $((U_x+iV_x)\times E_\C)\cap ((U_y+iV_y)\times E_\C)
=((U_x\cap U_y)+i(V_x\cap V_y))\times E_\C$
is connected and meets $U\times E$ whenever
it is non-empty.
Hence, by the Identity Theorem,
$g_x$ and $g_y$ coincide on the intersection
of their domains. We therefore obtain a well-defined
$\C$-analytic map $g\colon Q\times E_\C\to F_\C$
such that $g|_{(U_x+iV_x)\times E_\C}=g_x$
for each $x\in U$, using the open subset
$Q:=\bigcup_{x\in U}(U_j+iV_j)$ of $X_\C$.
For each $x\in U$, the map $g(x,\cdot)|_E=g_x(x,\cdot)|_E=f(x,\cdot)$
is $k$-linear over~$\R$.
Using the Identity Theorem,
we see that $g(x,\cdot)$ is $k$-linear over~$\C$
for each $x\in U$, and hence for each $x\in Q$
by the Identity Theorem.
By the case $(\infty,\C)$,
$g$ has a $\C$-analytic extension $\wt{g}\colon Q\times \wt{E}_\C\to \wt{F}_\C$.
Since $g(U\times E)=f(U\times E)\sub F\sub\wt{F}$
and $U\times E$ is dense in $U\times \wt{E}$,
we deduce that $\wt{g}(U\times \wt{E})\sub \wt{F}$;
we therefore obtain a map
\[
\wt{f}\colon U\times\wt{E}\to\wt{F},\quad (x,y)\mto \wt{g}(x,y)
\]
for $x\in U$, $y\in \wt{E}$.
Since $\wt{g}$ is a $\C$-analytic
extension for~$\wt{f}$,
the function $\wt{f}$ is $\R$-analytic.
To prove the claim, consider for $x\in U$
and $n\in\N$ the $\C$-analytic map
\[
g_{x,n}\colon (U_x+iV_x)\times n W_x\to F_\C,\quad
(z,y)\mto n^kh(z,(1/n)y).
\]
If $n\leq m$ and $y\in nW_x\cap E$,
we have for all $z\in U_x$ 
\[
g_{x,m}(z,y)=m^kh(z,(1/m)y)=m^kf(z,(1/m)y)
=f(z,y)=n^kf(t,(1/n)y)=g_{x,n}(z,y),
\]
whence $g_{x,m}(z,y)=g_{x,n}(z,y)$
for all $z\in U_x+iV_x$ and $y\in nW_x$,
by the Identity Theorem. Thus $g_x\colon$ $(U_x+iV_x)\times E_\C\to F_\C$,
$(z,y)\mto g_{x,n}(z,y)$ if $y\in nW_x$
is a well-defined $\C$-analytic
extension of $f|_{U_x\times E}$.
\end{proof}
{\bf Proof of Proposition~\ref{complbdl}.}
It suffices to prove the strengthening described in Remark~\ref{strength-compl}.
Let $(\psi_i)_{i\in I}$
be a family of local trivializations
$\psi_i\colon \pi^{-1}(U_i)\to U_i\times F$ of an $LC^r_\K$-vector
bundle~$E$
such that each local trivialization is some~$\psi_i$.
Let $(g_{ij})_{i,j\in I}$ be the corresponding cocycle
and $G_{ij}$
be the $LC^r_\K$-map $g_{ij}^\wedge\colon (U_i\cap U_j)\times F\to F$
which is $\bL$-linear in the second argument.
By Proposition~\ref{extend-second-multi},
there is a unique $LC^r_\K$-map
$\wt{G}_{ij}\colon U\times\wt{F}\to\wt{F}$
which extends~$G_{ij}$,
and $\wt{G}_{ij}$ is $\bL$-linear in the second argument.
Thus, we obtain a map
\[
\wt{g}_{ij}\colon U_i\cap U_j\to L_{\bL}(\wt{F}),\quad
x\mto \wt{G}_{ij}(x,\cdot).
\]
By continuity and density, for all $i\in I$
we have $\wt{G}_{ii}(x,y)=y$ for all $(x,y)\in U_i\times \wt{F}$.
Thus $\wt{g}_{ii}(x)=\id_{\wt{F}}$
for all $x\in U$.
For all $i,j,k\in I$, we have
\[
\wt{G}_{ij}(x,\wt{G}_{jk}(x,y))=\wt{G}_{ik}(x,y)\quad\mbox{for all
$(x,y)\in (U_i\cap U_j\cap U_k)\times \wt{F}$,}
\]
as both sides are continuous in $(x,y)$
and equality holds for~$y$ in the dense
subset~$F$ of~$\wt{F}$; thus
$\wt{g}_{ij}(x)\circ \wt{g}_{jk}(x)=\wt{g}_{ik}(x)$.
Notably, $\wt{g}_{ij}(x)\circ\wt{g}_{ji}(x)=\wt{g}_{ii}(x)=\id_{\wt{F}}$
for all $x\in U$ and $i,j\in I$,
entailing that $\wt{g}_{ij}(x)\in \GL(\wt{F})$.
By the preceding, the $\wt{g}_{ij}$
satisfy the cocycle conditions.
Let $\wt{E}$ and $\wt{\pi}$
be as in \ref{the-setting-compl};
define $\wt{\psi}_i\colon\wt{\pi}^{-1}(U_i)\to U_i\times\wt{F}$
as in (\ref{wtpsi2}), replacing~$\psi$ with $\psi_i$.
For all $i,j\in I$ and $x\in U$,
we then have that
\[
\wt{\psi}_i(\wt{\psi}_j^{-1}(x,y))=(x,\wt{G}_{ij}(x,y))
\]
holds for all $y\in\wt{F}$, as equality holds for all
$y\in F$.
As an analogue of Proposition~\ref{standfact}
holds with $LC^r_\K$-maps in place of $C^r_\K$-maps,
we get a unique $\bL$-vector bundle structure of class $LC^r_\K$
on $\wt{E}$ making $\wt{\psi}_i$ a local trivialization for each $i\in I$.\vspace{1mm}

It is apparent that $\wt{\beta}\colon G\times \wt{E}\to \wt{E}$ is an action,
and $\wt{E}_x$
is taken
$\bL$-linearly to $\wt{E}_{\alpha(g,x)}$
by $\wt{\beta}(g,\cdot)$,
for each $g \in G$ and $x\in M$.
It only remains to show that~$\wt{\beta}$
is $LC^r_\K$.
To this end, let $g_0\in G$ and $x_0\in M$;
we show that $\wt{\beta}$
is $LC^r_\K$ on $U\times \wt{\pi}^{-1}(V)$
for some open neighbourhood~$U$
of~$g_0$ in $G$ and an open neighbourhood
$V$ of $x_0$ in~$M$.
Indeed, there exists a local trivialization
$\psi\colon \pi^{-1}(W)\to W\times F$ of~$E$
over an open neighbourhood~$W$ of $\alpha(g_0,x_0)$
in~$M$. The action $\alpha$ being continuous,
we find an open neighbourhood $U$ of $g_0$
in~$G$ and an open neighbourhood~$V$ of $x_0$ in~$M$
over which~$E$ is trivial, such that $\alpha(U\times V)\sub W$.
Let $\phi\colon \pi^{-1}(V)\to V\times F$
be a local trivialization of~$E$ over~$V$.
Then
\[
\phi(\beta(g^{-1},\psi^{-1}(\alpha(g,x),v)))=(x,A(g,x,v))\quad
\mbox{for all $\,g\in U$, $x\in V$, and $v\in F$,}
\]
for an $LC^r_\K$-map
$A\colon U\times V\times F\to F$
which is $\bL$-linear in the third argument.
By Proposition~\ref{extend-second-multi},
there is a unique extension of~$A$
to an $LC^r_\K$-map
\[
\wt{A}\colon U\times V\times\wt{F}\to\wt{F},
\]
and the latter is $\bL$-linear in its third
argument.
For all $g\in U$ and $x\in V$, we then have
\[
\wt{\phi}(\wt{\beta}(g^{-1},\wt{\psi}^{-1}(\alpha(g,x),v)))=(x,\wt{A}(g,x,v))
\]
for all $v\in \wt{F}$,
as equality holds for all $v\in F$.
Thus $\wt{\beta}$ is $LC^r_\K$.\, \Punkt
\section{Tensor products of vector bundles}
Throughout this section, let $\bL\in\{\R,\C\}$, $\K\in\{\R,\bL\}$,
$s\in \{\infty,\omega\}$, and $r\in \N_0\cup\{\infty,\omega\}$
such that $r\leq s$.
Let $G$ be a $C^s_\K$-Lie group modeled on a locally convex $\K$-vector space~$Y$,
$M$ be a $C^r_\K$-manifold modeled on a locally convex $\K$-vector space~$Z$,
and $\alpha\colon G\times M\to M$ be a $C^r_\K$-action.
For $k\in \{1,2\}$,
let $\pi_k\!: E_k\to M$
be an equivariant $\bL$-vector bundle of class~$C^r_\K$
over~$M$,
whose typical fibre is a locally convex $\bL$-vector space~$F_k$.
Let $\beta_k\colon G\times E_k\to E_k$ be the $G$-action of class~$C^r_\K$.
Consider the set~$\cA$
of all pairs
of local trivializations
of~$E_1$ and~$E_2$
trivializing these over the same
open subset of~$M$. Using an index set~$I$,
we have $\cA=\{(\psi_i^1,\psi_i^2)\!: i\in I\}$,
where
$\psi_i^k\!: \pi_k^{-1}(U_i)\to U_i\times F_k$
is a local trivialization of~$E_k$ for $k\in\{1,2\}$,
for each $i\in I$. Apparently,
$(U_i)_{i\in I}$ is an open cover of~$M$.
\begin{numba}\label{set-1-tensor}
For our first result concerning tensor products,
Proposition~\ref{proptensor},
we assume that~$F_1$
is finite dimensional. Then, fixing a basis
$e_1,\ldots, e_n$ for $F_1$,
the map
$\theta\colon (F_2)^n\to F_1\tensor_\bL F_2$, $(y_1,\ldots, y_n)\mto
\sum_{\tau=1}^n e_\tau\tensor y_\tau$
is an isomorphism of $\bL$-vector spaces.
We give $F_1\tensor_\bL F_2$
the topology $\cT$ making~$\theta$
a homeomorphism. This topology makes $F_1\otimes_\bL F_2$
a locally convex $\bL$-vector space
and~$\theta$ an isomorphism of topological
$\bL$-vector spaces.
It is easy to check (and well known) that
the topology~$\cT$
is independent of the chosen basis.
Let $e_1^*,\ldots, e_n^*\in F_1'$ be the basis dual
to $e_1,\ldots, e_n$.
Our goal is to make
the union
\[
E_1\tensor E_2:=\bigcup_{x\in M} (E_1)_x\tensor_\bL
(E_2)_x
\]
an equivariant $\bL$-vector bundle
of class~$C^r_\K$ over $M$, with typical fibre $F_1\tensor_\bL F_2$;
the tensor products $(E_1)_x\otimes_\bL (E_2)_x$
are chosen pairwise disjoint here for $x\in M$.
Let $\pi\!: E_1\tensor E_2\to M$ be the mapping which takes
$v\in (E_1)_x\tensor_\bL (E_2)_x$ to~$x$.
\end{numba}
\begin{numba}
We define $\psi_i\!: \pi^{-1}(U_i)\to U_i\times (F_1\tensor_\bL F_2)$
via
\[
\psi_i(v):=(x, ((\pr_{F_1}\circ \psi^1_i|_{(E_1)_x})\tensor
(\pr_{F_2}\circ \psi^2_i|_{(E_2)_x}))(v))
\]
for
$x\in U_i$ and
$v\in (E_1)_x\tensor_\bL
(E_2)_x$, where $\pr_{F_k}\colon M \times F_k\to F_k$ is the projection.
\end{numba}
\begin{numba}
Given $i,j\in I$ and $x\in U_i\cap U_j$,
we have $\psi_i^k((\psi_j^k)^{-1}(x,v))=(x, G_{ij}^k(x,v))$
for all $k\in\{1,2\}$ and $v\in F_k$,
where $G_{ij}^k\!:(U_i\cap U_j)\times F_k\to F_k$
is $C^r_\K$ and
$g^k_{ij}(x):=G^k_{ij}(x,\cdot)$
an $\bL$-linear mapping.
Then
$c_{\sigma,\tau}\!:
U_i\cap U_j\to \K$, $x\mto e_\sigma^*(G^1_{ij}(x,e_\tau))$
is $C^r_\K$,
and
$\psi_i((\psi_j)^{-1}(x,v))=(x,G_{ij}(x,v))$
for $x\in U_i\cap U_j$ and $v=\sum_{\tau=1}^n e_\tau\tensor v_\tau\in F_1\tensor_\bL F_2$,
where
\begin{eqnarray*}
G_{ij}(x,v) & = &(g_{ij}^1(x)\tensor g_{ij}^2(x))(v)
=\sum_{\tau=1}^n (g_{ij}^1(x)e_\tau)\tensor (g^2_{ij}(x)v_\tau)\\
& = & \sum_{\sigma,\tau=1}^n
e_\sigma\tensor (c_{\sigma,\tau}(x) g^2_{ij}(x)v_\tau)=
\theta\left(\left( \sum_{\tau=1}^n c_{\sigma,\tau}(x)
G^2_{ij}(x,v_\tau)\right)_{\sigma=1}^n\right)\,.
\end{eqnarray*}
As $F_1\tensor_\bL F_2\to F_2$,
$v\mto v_\tau=\pr_\tau(\theta^{-1}(v))$
is a continuous linear map
(where $\pr_\tau\!: (F_2)^n\to F_2$
is the projection onto the $\tau$-component),
in view of
the preceding formula~$G_{ij}$ is~$C^r_\K$.
Thus, by Proposition~\ref{standfact},
there is a unique
$\bL$-vector bundle structure of class~$C^r_\K$
on
$E_1\tensor E_2$ making each $\psi_i$ a local trivialization.
\end{numba}
\begin{numba}
Note that
$\beta\!: G\times (E_1\tensor E_2)\to E_1\tensor E_2$,
$(g,v)\mto (\beta_1(g,\cdot)|_{(E_1)_x}^{E_{\alpha(g,x)}}
\tensor \beta_2(g,\cdot)|_{(E_2)_x}^{(E_2)_{\alpha(g,x)}})(v)$
for $g\in G$, $x\in M$, $v\in (E_1\tensor E_2)_x$
defines an action of~$G$ on $E_1\tensor
E_2$ by $\bL$-linear mappings, which makes $\pi\!:
E_1\tensor E_2\to M$ an equivariant mapping
and such that $\beta(g,\cdot)$
is $\bL$-linear on $(E_1)_x\otimes_\bL(E_2)_x$
for all $g\in G$ and $x\in M$.
\end{numba}
\begin{numba}\label{acttensor}
To show that $\beta$ is $C^r_\K$,
let $g_0\in G$ and $x_0\in M$.
We pick $i\in I$ such that $\alpha(g_0,x_0)\in U_i$.
The mapping~$\alpha$ being continuous,
we find open neighbourhoods~$U$ of~$g_0$ in~$G$ and~$V$
of~$x_0$ in~$M$ such that $\alpha(U\times V)\sub U_i$.
There is $j\in I$ such that
$x_0\in U_j\sub V$.
For $k\in\{1,2\}$, $g\in U$, $x\in U_j$ and $v\in F_k$,
we have
\[
\psi_i^k(\beta_k(g,(\psi_j^k)^{-1}(x,v)))=(\alpha(g,x),a_k(g,x,v))
\]
for some $C^r_\K$-map
$a_k\!: U\times U_j\times F_k\to F_k$ which is $\bL$-linear in the
final argument.
Define
$b_{\sigma,\tau}\!: U\times U_j \to \bL$,
$(g,x)\mto e_\sigma^*(a_1(g,x,e_\tau))$;
then $b_{\sigma,\tau}$ is~$C^r_\K$.
If $g\in U$, $x\in U_j$ and $v=\sum_{\tau=1}^n
e_\tau\tensor v_\tau\in F_1\tensor_\bL F_2$,then
$\psi_i(\beta(g,\psi_j^{-1}(x,v)))$ equals
\[
\left(\alpha(g,x),
\sum_{\tau=1}^n
a_1(g,x,e_\tau)\tensor a_2(g,x,v_\tau)\right)
=
\left(\alpha(g,x),\theta\left(\left(\sum_{\tau=1}^n b_{\sigma,\tau}(g,x)a_2(g,x,v_\tau)
\right)_{\sigma=1}^n\right)\right),
\]
which is a $C^r_\K$-function
of $(g,x,v)$. As a consequence, $\beta|_{U\times \pi^{-1}(U_j)}$
is~$C^r_\K$
and thus $\beta$, being~$C^r_\K$ locally,
is~$C^r_\K$. Summing up:
\end{numba}
\begin{prop}\label{proptensor}
Let~$G$ be a $C^s_\K$-Lie group and $M$ be a
$G$-manifold of class~$C^r_\K$.
Let $E_1$ and $E_2$ be equivariant $\bL$-vector bundles of class~$C^r_\K$
over~$M$.
If the typical fibre of~$E_1$
is finite dimensional,
then $E_1\tensor E_2$, as defined above,
is an
equivariant $\bL$-vector bundle of class $C^r_\K$ over~$M$.\, \Punkt
\end{prop}
\begin{numba}\label{set-2-tensor}
Instead of $\dim(F_1)<\infty$ (as in~\ref{set-1-tensor}),
assume that $F_1$ and $F_2$ are Fr\'{e}chet
spaces and the modeling spaces of~$G$
and $M$ are metrizable.
The completed projective tensor product
\[
F:=F_1 \wh{\otimes}_\pi F_2
\]
over~$\bL$ then is a Fr\'{e}chet space (cf.\ \cite[p.\,438,
lines after Definitions 43.4]{Tre}).
We define
\[
E:=E_1\wh{\otimes}_\pi E_2:=\vspace{-.7mm}\bigcup_{x\in M} (E_1)_x\wh{\otimes}_\pi (E_2)_x,
\]
where the $(E_1)_x\wh{\otimes}_\pi (E_2)_x$
for $x\in M$ are chosen pairwise disjoint.
Let $\pi\colon E\to M$ be the map taking
$v\in E_x:=(E_1)_x\wh{\otimes}_\pi (E_2)_x$
to~$x$.
Define $\psi_i\!: \pi^{-1}(U_i)\to U_i\times (F_1\wh{\tensor}_\pi F_2)$
via
\[
\psi_i(v):=\big(x, \,((\pr_{F_1}\circ \psi^1_i|_{(E_1)_x})\wh{\tensor}_\pi
(\pr_{F_2}\circ \psi^2_i|_{(E_2)_x}))(v)\big)
\]
for
$x\in U_i$ and
$v\in (E_1)_x\wh{\tensor}_\pi
(E_2)_x$, where $\pr_{F_k}\colon M \times F_k\to F_k$ is the projection.
Note that
$\beta\!: G\times E\to E$,
$(g,v)\mto (\beta_1(g,\cdot)|_{(E_1)_x}
\wh{\tensor}_\pi \beta_2(g,\cdot)|_{(E_2)_x})(v)$
for $g\in G$, $x\in M$, $v\in E_x$
defines an action of~$G$ on~$E$
which makes $\pi\!:
E \to M$ an equivariant mapping.
We show:
\end{numba}
\begin{prop}\label{tensor-projective}
$\pi\colon E_1\wh{\otimes}_\pi E_2\to M$
admits a unique structure of equivariant $\bL$-vector bundle
of class~$C^r_\K$ over~$M$
such that $\psi_i$ is a local trivialization for each $i\in I$.
\end{prop}
\begin{proof}
The uniqueness for prescribed local trivializations is clear.
Let us show existence of the structure.
Given $i,j\in I$ and $x\in U_i\cap U_j$,
we have $\psi_i^k((\psi_j^k)^{-1}(x,v))=(x, G_{ij}^k(x,v))$
for all $k\in\{1,2\}$ and $v\in F_k$,
where $G_{ij}^k\!:(U_i\cap U_j)\times F_k\to F_k$
is $C^r_\K$ and
$g^k_{ij}(x):=G^k_{ij}(x,\cdot)$
an $\bL$-linear mapping.
By Proposition~\ref{reallyneed1}\,(a),
the map $g_{ij}^k\colon U_i\cap U_j\to L(F_k)_c$
is~$C^r_\K$. Now
\[
L_\bL(F_1)_c\times L_\bL(F_2)\to L_\bL(F_1\wh{\otimes}_\pi F_2)_c,\;\;
(S,T)\mto S\wh{\otimes}_\pi T
\]
being continuous $\bL$-bilinear (as recalled in Lemma~\ref{via-cp-in-proj}),
we deduce that
\[
g_{ij}\colon U_i\cap U_j\to L_\bL(F_1\wh{\otimes}_\pi F_2)_c,\quad
x\mto g_{ij}^1(x)\wh{\otimes}_\pi g_{ij}^2(x)
\]
is $C^r_\K$. Hence $G_{ij}:=g_{ij}^\wedge\colon (U_i\cap U_j)\times
(F_1\wh{\otimes}_\pi F_2)\to F_1 \wh{\otimes}_\pi F_2$,
$(x,v)\mto g_{ij}(x)(v)$ is $C^r_\K$, by Proposition~\ref{reallyneed2}\,(a).
We easily check that
$\psi_i((\psi_j)^{-1}(x,v))=(x,G_{ij}(x,v))$
holds for~$G_{ij}$ as just defined,
for all $x\in U_i\cap U_j$ and $v\in F_1\wh{\otimes}_\pi F_2$.
Hence $E_1\wh{\otimes}_\pi E_2$
can be made an $\bL$-vector bundle
of class $C^r_\K$
in such a way that each $\psi_i$
is a local trivialization, by Proposition~\ref{standfact}.
Note that
$\beta(g,\cdot)$ is $\bL$-linear on $E_x$ for
all $g\in G$ and $x\in M$.
To show that $\beta$ is $C^r_\K$,
let $g_0$, $x_0$,
$i$, $U$, $V$, $j$ and the $C^r_\K$-map $a_k$
be as in the proof of Proposition~\ref{proptensor}.
By Proposition~\ref{reallyneed1}\,(a), $a_k^\vee\colon U\times U_j\to L(F_k)_c$,
$(g,x)\mto a_k(g,x,\cdot)$
is $C^r_\K$. Hence
\[
a\colon U\times U_j\to L(F_1\wh{\otimes}_\pi F_2)_c,\;\;
(g,x)\mto a_1^\vee(g,x)\wh{\otimes}_\pi a_2^\vee(g,x)
\]
is $C^r_\K$, by the Chain Rule and Lemma~\ref{via-cp-in-proj}.
Using Proposition~\ref{reallyneed2}\,(a), we find that the map
$a^\wedge\colon U\times U_j\times (F_1\wh{\otimes}_\pi F)\to F_1\wh{\otimes}_\pi F_2$,
$(g,x,v)\mto a(g,x)(v)$ is~$C^r_\K$.
We easily verify
that
$\psi_i(\beta(g,(\psi_j)^{-1}(x,v)))=(\alpha(g,x),a^\wedge(g,x,v))$
for all $(g,x,v)\in U\times U_j\times (F_1\wh{\otimes}_\pi F_2)$.
Thus $\psi_i(\beta(g,(\psi_j)^{-1}(x,v)))$ is $C^r_\K$ in
$(g,x,v)$, which completes the proof.
\end{proof}
We used the following fact:
\begin{la}\label{via-cp-in-proj}
Let $E_1$, $E_2$, $F_1$, and $F_2$ be Fr\'{e}chet
spaces over $\bL\in\{\R,\C\}$.
Then the following bilinear map is continuous:
\[
\Xi \colon L_\bL(E_1,F_1)_c\times L_\bL(E_2,F_2)_c\to L_\bL((E_1\wh{\otimes}_\pi E_2),
(F_1\wh{\otimes}_\pi F_2))_c,\;\;
(S_1,S_2)\mto S_1\wh{\otimes}_\pi S_2.
\]
\end{la}
\begin{proof}
Let $K\sub E_1\wh{\otimes}E_2$ be compact, $q$
be a continuous seminorm on $F_1\wh{\otimes}_\pi F_2$,
and $\ve>0$.
After increasing~$q$, we may assume that
$q=q_1\otimes q_2$ for continuous seminorms
$q_k$ on~$F_k$ for $k\in\{1,2\}$.
By \cite[p.\,465, Corollary 2 to Theorem~45.2]{Tre},
$K$ is contained in the closed absolutely convex hull of $K_1\otimes K_2$
for certain compact subsets $K_k\sub E_k$ for $k\in\{1,2\}$.
For all $S_k\in L(E_k,F_k)$ such that
$\sup q_k(S_k(K_k))\leq\sqrt{\ve}$,
we have
\[
\sup q((S_1\wh{\otimes}_\pi S_2)(K))\leq
\sup q((S_1\wh{\otimes}_\pi S_2)(K_1\otimes K_2))
=\sup q_1(S_1(K_1))q_2(S_2(K_2))\leq \sqrt{\ve}^2=\ve,
\]
using \cite[Proposition~43.1]{Tre}.
The assertion follows.
\end{proof}
\begin{numba}\label{cont-if-Hilbert}
If $E_1$ and $E_2$ are Hilbert spaces over~$\bL$
with Hilbert space tensor product $E_1\wh{\otimes} E_2$,
and also $F_1$ and $F_2$ are Hilbert space over~$\bL$,
then the bilinear map
\[
\Xi\colon L(E_1,F_1)_b\times L(E_2,F_2)_b\to L((E_1\wh{\otimes} E_2),(F_1\wh{\otimes} F_2))_b
\]
is continuous, as $\|S_1\wh{\otimes} S_2\|_{\op}
\leq \|S_1\|_{\op}\|S_2\|_{\op}$.
\end{numba}
\begin{numba}
Replace the hypotheses in \ref{set-1-tensor}
and \ref{set-2-tensor}
with the requirements
that $G$ and $M$ are modeled on metrizable
locally convex spaces, $r\geq 1$
and $F_1$, $F_2$ are Hilbert
spaces. We now use \ref{cont-if-Hilbert}
instead of Lemma~\ref{via-cp-in-proj},
replace
$F_1\wh{\otimes}_\pi F_2$
with the Hilbert space $F_1\wh{\otimes}F_2$,
Proposition~\ref{reallyneed1}\,(a)
with Proposition~\ref{reallyneed1}\,(b)
(so that operator-valued maps are only $C^{r-1}_\K$)
and use Proposition~\ref{reallyneed2}\,(b)
with $r-1$ in place of~$r$.
Repeating the proof of Proposition~\ref{tensor-projective},
we get:
\end{numba}
\begin{prop}\label{bdl-hilbert}
On $E_1\wh{\otimes}E_2=\bigcup_{x\in M}(E_1)_x\wh{\otimes}(E_2)_x$,
there is a unique equivariant
$\bL$-vector bundle structure of class $C^{r-1}_\K$
over~$M$ whose typical
fibre is the Hilbert space $F_1\wh{\otimes}F_2$,
such that $\psi_i\colon \pi^{-1}(U_i)\to U_i\times
(F_1\wh{\otimes} F_2)$
is a local trivialization for each~$i\in I$. \Punkt
\end{prop}
\begin{rem}
If $r\geq 1$, $G$ and $M$ are modeled on metrizable spaces
and both $F_1$ and $F_2$ are pre-Hilbert spaces
with Hilbert space completions $\wt{F}_1$ and $\wt{F}_2$,\vspace{.4mm}
we can use the non-completed tensor product
$F_1\otimes_\bL F_2\sub \wt{F}_1\wh{\otimes} \wt{F}_2$
with the induced topology as the fibre
and get an equivariant $\bL$-vector bundle structure
over~$M$ of class $C^{r-1}_\K$ over~$M$
on $E_1\otimes E_2=\bigcup_{x\in M}(E_1)_x\otimes_\bL (E_2)_x$,
exploiting that the $\bL$-bilinear
map $L_\bL(F_1)_b\times L_\bL(F_2)_b\to L_\bL(F_1\otimes_\bL F_2)_b$,
$(S_1,S_2)\mto S_1\otimes S_2$ is continuous.
\end{rem}
\section{Locally convex direct sums of vector bundles}
Let $\bL\in\{\R,\C\}$,
$\K\in\{\R,\bL\}$,
$s\in \{\infty,\omega\}$,
$r\in \N_0\cup\{\infty,\omega\}$
such that $r\leq s$,
$G$ be a $C^s_\K$-Lie group
modeled on a locally convex space~$Y$,
and $M$ be a $C^r_\K$-manifold
modeled on a locally convex $\K$-vector space~$Z$,
together with a $C^r_\K$-action
$\alpha\colon G\times M\to M$.
\begin{numba}\label{setting Whitney}
Let $n\in\N$ and $\pi_k\!: E_k\to M$ be an
equivariant $\bL$-vector bundle of class $C^r_\K$
over~$M$
for $k\in\{1,\ldots, n\}$,
with typical fibre a locally convex $\bL$-vector space~$F_k$;
let $\beta_k\colon G\times E_k\to E_k$
be the $G$-action
and $\pr_{F_k}\colon M\times F_k\to F_k$
be the projection onto the 2nd component.
We easily check: There is a unique
$\bL$-vector bundle structure of class~$C^r_\K$
on the ``Whitney~sum''
\[
E:=E_1\oplus\cdots\oplus E_n:=
\bigcup_{x\in M} (E_1)_x\times\cdots \times (E_n)_x,\vspace{-.2mm}
\]
with the apparent map $\pi\!:
E\to M$,
such that $\psi\!: \pi^{-1}(U)
\to U\times F_1\times \cdots\times F_n$,
$v=(v_1,\ldots,v_n)\mto (\pi(v),\pr_{F_1}(\psi_1(v_1)),\ldots, \pr_{F_n}(\psi_n(v_n)))$
is a local trivialization of~$E$,
for all families $(\psi_k)_{k=1}^n$ of
local trivializations
$\psi_k\!: (\pi_k)^{-1}(U)\to U\times F_k$
which trivialize the $E_k$'s over a joint
open subset~$U$ of~$M$.
Then
$\beta(g,v):=(\beta_1(g,v_1),\ldots,
\beta_n(g,v_n))$
for $g\in G$, $v=(v_1,\ldots, v_n)\in
E$ yields an action of~$G$ on~$E$.
It is straightforward that~$\beta$ is~$C^r_\K$.
Thus:
\end{numba}
\begin{prop}\label{propfinsum}
If $E_1,\ldots,E_n$ are equivariant
$\bL$-vector bundles of class~$C^r_\K$
over a $G$-manifold~$M$ of class~$C^r_\K$,
then also $E_1\oplus \cdots\oplus E_n$ is an equivariant
$\bL$-vector bundle of class~$C^r_\K$ over~$M$.\, \Punkt
\end{prop}
The following lemma allows
infinite direct sums to be tackled.
\begin{la}\label{mapsoplus}
Let $(E_i)_{i\in I}$ and $(F_i)_{i\in I}$ be families
of locally convex spaces over $\K\in\{\R,\C\}$, with locally convex direct sums
$E:=\bigoplus_{i\in I}E_i$
and $F:=\bigoplus_{i\in I} F_i$, respectively.
Let $V$ be an open subset of a locally convex $\K$-vector space~$Z$.
Let $r\in \N_0\cup\{\infty\}$,
and assume that $f_i\!: V\times E_i\to F_i$
is a mapping
which is linear in the second argument,
for each $i\in I$.
Moreover, assume that {\rm(a)}
or {\rm(b)} holds:
\begin{itemize}
\item[\rm(a)]
$Z$ is finite dimensional; or
\item[\rm(b)]
$Z$ and each $E_i$ is a $k_\omega$-space
and $I$ is countable.
\end{itemize}
If~$f_i$
is of class~$C^r_\K$ for each $i\in I$,
then also the following map is $C^r_\K$:
\[
f\!: V\times E\to F\, ,\;\;\;\;
(x,(v_i)_{i\in I})\mto (f_i(x,v_i))_{i\in I}.
\]
\end{la}
\begin{proof}
If (b) holds,
we may assume that $I$ is countably infinite,
excluding a trivial case. Thus, assume that $I=\N$.
For each $n\in \N$, identify $E_1\times\cdots\times E_n$
with a vector subspace of $E$,
identifying $x\in E_1\times\cdots \times E_n$ with $(x,0)$.
For each $n\in\N$, we then have
\[
Z\times E=\bigcup_{n\in\N}(Z\times E_1\times\cdots\times E_n)\quad\mbox{and}\quad
V\times E=\bigcup_{n\in \N}(V\times E_1\times\cdots\times E_n),\vspace{-.5mm}
\]
where $Z\times E_1\times\cdots\times E_n$
is a $k_\omega$-space in the product topology.
The inclusion map
\[
\lambda_n\colon F_1\times\cdots\times F_n\to\bigoplus_{i\in \N}F_i,\;\;
v\mto(v,0)\vspace{-1mm}
\]
is continuous and $\K$-linear.
Moreover,
\[
g_n\colon V\times E_1\times\cdots\times E_n\to
F_1\times\cdots\times F_n,\;\;
(x,v_1,\ldots,v_n)\mto(f_1(x,v_1),\ldots,f_n(x,v_n))
\]
is a $C^r_\K$-map and so is
$f|_{V\times E_1\times \cdots\times E_n}=\lambda_n\circ g_n$,
for each $n\in\N$.
Hence~$f$ is $C^r_\K$
on the open subset $V\times E$ of $Z\times E$,
considered as the locally convex direct limit
$\dl(Z\times E_1\times\cdots \times E_n)$,\vspace{-.5mm}
by \cite[Proposition 4.5\,(a)]{DSU}.
But this locally convex space equals
$Z\times \dl(E_1\times\cdots\times E_n)=Z\times E$\vspace{-1mm}
with the product topology (see \cite[Theorem~3.4]{Hir}).\vspace{1.3mm}

If (a) holds,
it suffices to prove the assertion for $r\in \N_0$. We proceed by induction.
{\em The case $r=0$}.
Let $(x,v)=(x,(v_i)_{i\in I})\in V\times E$;
we show that $f$ is continuous
at $(x,v)$.
To this end, let
$Q$ be
an absolutely convex,
open $0$-neighbourhood in~$F$.
There is a finite subset $J\sub I$
such that $v_i=0$ for all $i\in I\setminus  J$.
Let $N:=|J|+1$.
For each $i\in I$,
the intersection
$Q_i:=(\frac{1}{N}Q)\cap F_i$
is an absolutely convex, open
$0$-neighbourhood in~$F_i$.
For the absolutely convex hull, we get
$\absconv (\bigcup_{i\in I}Q_i)\sub \frac{1}{N}Q$.
Since $f_i$ is continuous
for each $i\in J$ and~$J$ is finite,
we find a compact neighbourhood~$K$
of~$x$ in~$V$ such that
$f_i(y,v_i)-f_i(x,v_i)\in Q_i$
for all $y\in K$ and $i\in J$.
Since $f_i(K\times\{0\})=\{0\}$, where~$K$ is compact
and $f_i$ is continuous,
for each $i\in I$
there is an absolutely convex, open $0$-neighbourhood~$P_i$
in~$E_i$
such that $f_i(K\times P_i)\sub Q_i$.
Then $W:=v+\absconv(\bigcup_{i\in I} P_i)$
is an open neighbourhood of~$v$ in~$E$.
Let
$y\in K$ and $w\in W$ be given, say $w=(w_i)_{i\in I}=v+(t_ip_i)_{i\in I}$
where $p_i\in P_i$
and $(t_i)_{i\in I}\in \bigoplus_{i\in I}\R$
such that $t_i\in [0,1]$
and $\sum_{i\in I} t_i=1$.
Then, for each $i\in I\setminus  J$,
since $v_i=0$ we obtain
\[
f_i(y,w_i)-f(x,v_i)=f_i(y,t_ip_i)=t_if_i(y,p_i)\in t_i Q_i\,.
\]
For $i\in J$ on the other hand, we have
\begin{eqnarray*}
f_i(y,w_i)-f(x,v_i) & = & f_i(y,w_i-v_i)+(f_i(y,v_i)-f_i(x,v_i))\\
& = & t_if_i(y,p_i)+(f_i(y,v_i)-f_i(x,v_i))\in t_iQ_i+Q_i.
\end{eqnarray*}
As a consequence,
$f(y,w)-f(x,v)\in (\prod_{i\in I}t_iQ_i)+\sum_{i\in J} Q_i
\sub \frac{1}{N}Q+\sum_{i\in J}\frac{1}{N}Q=Q$,
using the convexity of~$Q$.
We have shown that~$f$ is continuous at $(x,v)$.

{\em Induction step.} Let $r\geq 1$ and assume
the assertion is true for all numbers $<r$.
Given $u,v\in E$,
$x\in V$, and $z\in Z$,
we have $u,v\in \bigoplus_{i\in J}E_i=\prod_{i\in J}E_i$
for some finite subset $J\sub I$.
The map $f_J \!:
V\times \prod_{i\in J}E_i\to \prod_{i\in J}F_i$,
$(x,(v_i)_{i\in J})\mto (f_i(x,v_i))_{i\in J}$
is~$C^1_\K$,
whence
\begin{eqnarray*}
df_J((x,u),(z,v)) & = & \lim_{t\to 0}t^{-1}(f_J((x,u)+t(z,v))-f_J(x,u))\\
& = & \lim_{t\to 0}t^{-1}(f((x,u)+t(z,v))-f(x,u))
=df((x,u),(z,v))
\end{eqnarray*}
exists in $\prod_{i\in J} F_i$
and thus in~$F$; its $i$th component
is
\[
df_i((x,u_i),(z,v_i))=d_1f_i(x,u_i,z)+d_2f_i(x,u_i,v_i)
\]
in terms of partial differentials.
Note that the mappings $g_i\!:
(V\times Z)\times (E_i\times E_i)\to F_i$,\linebreak
$(x,z,u_i,v_i)\mto d_1f_i(x,u_i,z)$
and $h_i\!: (V\times Z)\times (E_i\times E_i)\to F_i$,
$(x,z,u_i,v_i)\mto 
d_2f_i(x,u_i,v_i)=f_i(x,v_i)$
are $C^{r-1}_\K\!$ and linear in $(u_i,v_i)$.
By induction, the mappings
\[g\!: (V\times Z)\times (E\times E)\to F,
\;\;\;(x,z,(u_i)_{i\in I},
(v_i)_{i\in I})\mto (g_i(x,z,u_i,v_i))_{i\in I}\;\;\;\mbox{and}
\]
\[h\!: (V\times Z)\times (E\times E)\to F,\;\;\;
(x,z,(u_i)_{i\in I},(v_i)_{i\in I})\mto (h_i(x,z,u_i,v_i))_{i\in I}
\]
are~$C^{r-1}_\K$,
using that $E\times E\isom \bigoplus_{i\in I} (E_i\times E_i)$.
Hence also
$df\!: (V\times E)\times (Z\times E)\to F$
is~$C^{r-1}_\K$,
as $df((x,u),(z,v))=g(x,z,u,v)+h(x,z,u,v)$.
Since $df$ exists and is~$C^{r-1}_\K$,
the continuous map~$f$ is~$C^r_\K$.
\end{proof}
\begin{rem}
The conclusion of
Lemma~\ref{mapsoplus}
does not hold for
$(r,\K)=(\omega,\R)$ in
the example $I=\N$, $V=Z=\R$, $E_k=\R$,
$f_k(r,t):= \frac{t}{1+kr^2}$,
using that the Taylor series
of $f_k(\cdot,t)$ around~$0$ has radius of convergence
$\frac{1}{\sqrt{k}}$ for all $t\in \R\setminus \{0\}$
(cf.\
\cite{GCX}, Remark~3.14).
\end{rem}
\begin{numba}\label{sum-situa-3}
Assuming now $r\not=\omega$,
let $(E_i)_{i\in I}$
be a family of equivariant $\bL$-vector bundles\linebreak
$\pi_i\!:
E_i \to M$ of class $C^r_\K$
with typical~fibre~$F_i$
and $G$-action $\beta_i\colon G\times E_i\to E_i$.
We assume that (a) or (b) is satisfied:
\begin{itemize}
\item[(a)]
$G$ and $M$ are finite dimensional; or
\item[(b)]
$I$ is countable
and each $F_i$ as well as the modeling spaces
of~$G$ and~$M$ are $k_\omega$-spaces.
\end{itemize}
Moreover, we assume:
\begin{itemize}
\item[(c)]
For each $x\in M$,
there exists an open neighbourhood
$U$ of~$x$ in~$M$, such that, for each $i\in I$,
the vector bundle~$E_i$ admits
a local trivialization $\psi_i\!: (\pi_i)^{-1}(U)\to U\times F_i$.
\end{itemize}
Thus, the $C^r_\K$-vector bundle
$E_i|_U$
is trivializable for each $i\in I$.
Define $E:=\bigcup_{x\in M}\bigoplus_{i\in I}(E_i)_x$
with pairwise disjoint direct sums
and $\pi\!: E\to M$, $\bigoplus_{i\in I}(E_i)_x\ni v \mto x$.
Then
\[
\beta\colon G\times E\to E,\quad (g,(v_i)_{i\in I})\mto(\beta_i(g,v_i))_{i\in I}
\]
is a $G$-action such that $\beta(g,\cdot)|_{E_x}\colon E_x\to E_{\alpha(g,x)}$
is $\bL$-linear for all $(g,x)\in G\times M$,
where $E_x:=\pi^{-1}(\{x\})$.
We readily deduce from
Proposition~\ref{standfact} and
Proposition~\ref{mapsoplus}
that there is a unique
$\bL$-vector bundle
structure of class~$C^r_\K$ on~$E$
such that
\[
\pi^{-1}(U)\to U\times \bigoplus_{i\in I}F_i,\quad
E_x\ni (v_i)_{i\in I}\mto (x,(\pr_{F_i}(\psi_i(v_i)))_{i\in I})
\]
is a local trivialization for~$E$,
for each family $(\psi_i)_{i\in I}$ of local trivializations
as above. The latter makes $E$
an equivariant $\bL$-vector bundle of class~$C^r_\K$.
In fact, the $C^r_\K$-property of~$\beta$
can be checked using pairs of local trivializations,
as in the proofs of Propositions~\ref{complbdl}, \ref{proptensor},
and \ref{tensor-projective}.
Then apply Proposition~\ref{mapsoplus},
with~$F_i$ in place of~$E_i$ and $Y\times Z$
in place of~$Z$.
Thus:
\end{numba}
\begin{prop}\label{propinfsum}
In the situation of {\rm\ref{sum-situa-3}},
$\bigoplus_{i\in I}E_i$
is an equivariant $\bL$-vector bundle
of class $C^r_\K$ over~$M$.\, \Punkt
\end{prop}
\begin{rem}
If $M$ is a $C^r_\R$-manifold,
then every $x\in M$
has an open neighbourhood~$U$
which is $C^r_\R$-diffeomorphic to
a convex open subset~$W$
in the modeling space~$Z$ of~$M$.
If $W$ can be chosen
$C^r_\R$-paracompact,
then every $C^r_\R$-vector bundle
over~$U$ is trivializable
(see \cite[Corollary~15.10]{SMO}).
The latter condition is satisfied,
for example, if~$Z$
is finite dimensional,
a Hilbert space, or a countable
direct limit of finite-dimensional
vector spaces
(cf.\ \cite[Corollary 16.16]{KaM}
and \cite[Proposition~3.6]{FUN}).
If $(r,\K)=(\infty,\C)$
and~$Z$ has finite dimension,
then each finite-dimensional
holomorphic vector bundle
over a, say, polycylinder
in~$Z$ is $C^\infty_\C$-trivializable
(cf.\ \cite{Gra}).
Under suitable hypotheses,
holomorphic Banach vector bundles
over contractible bases
are $C^\infty_\C$-trivializable
as well~\cite{Pat}.
\end{rem}
\section{Dual bundles and cotangent bundles}
In this section, we discuss conditions ensuring
that a vector bundle has a canonical dual
bundle. Let $\bL\in\{\R,\C\}$,
$\K\in\{\R,\bL\}$, $r\in\N_0\cup\{\infty,\omega\}$,
and $M$ be a $C^r_\K$-manifold modeled
on a locally convex space~$Z$.
\begin{defn}\label{defdual}
Let $\pi\colon E\to M$ be an $\bL$-vector bundle
of class $C^r_\K$, with typical fibre~$F$.
Consider the disjoint union
\[
E':=\bigcup_{x\in M}(E_x)'\, ;
\]
let $p\colon E'\to M$ be the map taking $\lambda\in (E_x)'$ to~$x$,
for each $x\in M$.
Given $t\in \N_0\cup\{\infty,\omega\}$
such that $t\leq r$,
we say that \emph{$E$ has a canonical dual bundle}
of class $C^t_\K$ with respect to $\cS\in \{b,c\}$
if $E'$ can be made an $\bL$-vector bundle
of class~$C^t_\K$ over~$M$, with typical fibre $F'_\cS$
and bundle projection~$p$,
such that
\begin{equation}\label{wtpsi}
\wt{\psi}\colon
p^{-1}(U)\to U\times F'_\cS\,,\quad
(E')_x=(E_x)'\ni \lambda\mto (x, ((\pr_F\circ \psi|_{E_x})^{-1})'(\lambda))
\end{equation}
is a local trivialization of~$E'$,
for each local trivialization $\psi\colon \pi^{-1}(U)\to U\times F$
of~$E$.
\end{defn}
To pinpoint situations
where the dual bundle exists, we recall
a fact concerning the formation of
dual linear maps (see \cite[Proposition~16.30]{MEM}):
\begin{la}\label{lemmem1}
Let $E$ and~$F$ be locally convex spaces,
and $\cS\in \{b,c\}$.
If the evaluation homomorphism
$\eta_{F,\cS}\colon F\to (F'_\cS)'_\cS$, $\eta_{F,\cS}(x)(\lambda):=\lambda(x)$
is continuous, then
\[
\Theta\colon L(E,F)_\cS\to L(F'_\cS,E'_\cS)_\cS\, ,\quad
\alpha\mto \alpha'
\]
is a continuous linear map.\,\Punkt
\end{la}
\begin{rem}\label{lemmem2}
Let~$F$ be a locally convex $\K$-vector space over $\K\in\{\R,\C\}$.
It is known that $\eta_{F,b}$ is continuous if and only if
$F$ is quasi-barrelled,
i.e., every bornivorous barrel in~$F$
is a $0$-neighbourhood \cite[Proposition 2 in Section 11.2]{Jar}.
In particular,
$\eta_{F,b}$ is continuous if~$F$ is
bornological or barrelled.
It is also known that
$\eta_{F,c}$ is continuous
(and actually a topological embedding) if~$F$
is a $k_\R$-space.
If $\K=\R$, this follows from \cite[Theorem~2.3]{Nb2}
and \cite[Lemma~14.3]{Ban}
(cf.\ also \cite[Propositions~2.3 and 2.4]{Ban}).
If $\K=\C$ and~$F$ is a $k_\R$-space,
then $\eta_{F_\R,c}$ is a topological
embedding for the real topological vector space~$F_\R$
underlying~$F$.
Now $(F'_c)_\R\cong (F_\R)'$ as a real topological
vector spaces, using that a continuous $\C$-linear
functional $\lambda\colon F\to\C$
is determined by its real part.
Transporting the complex vector space structure
from $F'_c$ to $(F_\R)'$,
the latter becomes a complex locally convex space.
Thus $((F'_c)'_c)_\R$
can be identified with $((F_\R)'_c)'_c$,
and it is easy to verify that $\eta_{F,c}$ corresponds
to $\eta_{F_\R,c}$
if we make the latter identification.
\end{rem}
\begin{prop}\label{finalprop}
Let $\pi\colon E\to M$ be an
$\bL$-vector bundle of class~$C^r_\K$, with typical fibre~$F$.
Let $\cS\in \{b,c\}$.
If $\cS=c$, let $r_- := r$;
if $\cS=b$, assume $r\geq 1$ and set $r_-:=r-1$.
Consider the following
conditions:
\begin{itemize}
\item[\rm($\alpha$)]
The modeling space~$Z$ of~$M$
is finite dimensional, $\eta_{F,\cS}$ is continuous,
and $F'_\cS$ is barrelled.\footnote{Example for $\cS=b$:
if $F$ is a reflexive locally convex space,
then~$\eta_{F,b}$ is continuous and $F'_b$ is barrelled, being reflexive.}
\item[\rm($\beta$)]
$\eta_{F,\cS}$ is continuous
and, moreover,
$(Z\times F'_\cS)\times (Z\times F'_\cS)$
is a $k_\R$-space,
or $r_-=0$ and $Z\times F'_\cS$
is a $k_\R$-space, or $(r,\K)=(\infty,\C)$
and $Z\times F'_\cS$ is a $k_\R$-space.
\item[\rm($\gamma$)]
$F$ is normable.
\end{itemize}
If $(\alpha)$ or $(\beta)$ is satisfied
with $\cS=c$,
then $E$ has a canonical dual bundle
of class $C^r_\K$ with respect to $\cS=c$.
If $(\alpha)$, $(\beta)$, or $(\gamma)$ is satisfied
with $\cS=b$,
then $E$ has a canonical dual bundle
of class $C^{r-1}_\K$ with respect to $\cS=b$.
\end{prop}
\begin{proof}
Let $E'$ be the disjoint union
$\bigcup_{x\in M}(E_x)'$,
and $p\colon E'\to M$ be as in Definition~\ref{defdual}.
Let $(\psi_i)_{i\in I}$ be a
family such that
the $\psi_i\colon \pi^{-1}(U_i)\to
U_i\times F$
form the set of all local trivializations of~$E$.
Let $(g_{ij})_{i,j\in I}$ be the associated
cocycle (see \ref{defncocyc}).
Then $G_{ij}:=g_{ij}^\wedge$ is
$C^r_\K$ and hence
$g_{ij}=(G_{ij})^\vee$
is $C^{r_-}_\K$,
by Proposition~\ref{reallyneed1}.
Given $i\in I$,
we define $\wt{\psi_i}\colon p^{-1}(U_i)\to U_i\times F'_\cS$
as in\,(\ref{wtpsi}),
using $\psi_i$ instead of~$\psi$.
Then
\begin{eqnarray*}
\wt{\psi}_i(\wt{\psi}_j^{-1}(x,\lambda))
& = & (x,
((\pr_F\circ \psi_i|_{E_x})^{-1})'\circ (\pr_F\circ \psi_j|_{E_x})'(\lambda))\\
& = & (x,(\pr_F\circ \psi_j|_{E_x}\circ (\pr_F\circ \psi_i|_{E_x})^{-1})'
(\lambda))\; = \; (x, g_{ji}(x)'(\lambda))
\end{eqnarray*}
for all $x\in U_i\cap U_j$ and $\lambda\in F'$
shows that
\[
(\wt{\psi}_i\circ \wt{\psi}_j^{-1})(x,\lambda)=
(x,h_{ij}(x)(\lambda))\, ,
\]
where $h_{ij}(x):=g_{ji}(x)'\in \GL(F'_\cS)$.
If ($\alpha$) or ($\beta$)
holds, then
$\eta_{F,\cS}\colon  F\to (F'_\cS)'_\cS$ is
continuous by hypothesis.
If $\cS=b$ and ($\gamma$) holds,
then $\eta_{F,b}$ is an isometric
embedding (as is well known) and hence continuous.
Thus $\Theta\colon L(F)_\cS\to L(F'_\cS)_\cS$,
$\alpha\mto\alpha'$ is a continuous $\bL$-linear map
(Lemma~\ref{lemmem1}).
Since
$g_{ji}\colon U_i\cap U_j\to L(F)_\cS$
is $C^{r_-}_\K$,
we deduce that
$h_{ij}=\Theta\circ g_{ji}\colon U_i\cap U_j\to L(F'_\cS)_\cS$
is $C^{r_-}_\K$.
Thus Condition\,$\mbox{(g)}'$ of Corollary~\ref{cornew}
is satisfied,
with $r_-$ in place of~$r$.
Conditions\,(a)--(f) being apparent,
the cited corollary provides an $\bL$-vector bundle
structure of class $C^{r_-}_\K$
on~$E'$.
\end{proof}
Without specific hypotheses, a canonical dual bundle
need not exist.
\begin{example}\label{nodual}
Let $A$ be a unital, associative, locally convex topological
$\K$-algebra whose group of units $A^\times$ is open
in~$A$, and such that the inversion map $\iota\colon
A^\times\to A^\times$
is continuous. Then $\iota$ is smooth (and
indeed $\K$-analytic), see \cite{ALG}.
We assume that the locally convex space underlying~$A$
is a non-normable Fr\'{e}chet-Schwartz space
and hence Montel, ensuring that $L(A)_b=L(A)_c$.
for example, we might take $A:=C^\infty(K,\K)$,
where~$K$ is a connected,
compact smooth manifold of positive dimension (cf.\ \cite{ALG}).
Let $r,t\in \N_0\cup \{\infty,\omega\}$
with $t\leq r$
and $\cS\in \{b,c\}$.
We
consider the trivial vector bundle
\[
\pr_1\colon E:=A^\times \times A\to A^\times\,.
\]
(Thus $E\isom T A^\times$, the tangent bundle).
Then $E$ is a $\K$-vector bundle of class $C^r_\K$
over the base $A^\times$,
with typical fibre~$A$.
Both
$\psi_1:=\id\colon A^\times\times A\to A^\times\times A$
and $\psi_2\colon A^\times\times A\to A^\times\times A$,
$(a,v)\mto (a,av)$ are global trivializations
of~$E$.
Identifying $E':=\bigcup_{a\in A^\times}(E_a)'$
with the set $A^\times\times A'$,
we consider the associated bijections
$\wt{\psi}_i\colon E'=A^\times\times A'\to A^\times \times A'$
for $i\in\{1,2\}$ (cf.\ (\ref{wtpsi})).
Thus $\wt{\psi}_1=\id$,
and $\wt{\psi}_2(a,\lambda)=(a,\lambda(a^{-1}\cdot))$
for $a\in A^\times$, $\lambda\in A'$.
The map $G_{ij}\!: A^\times \times A\to A$,
$(a,v)\mto \pr_2(\psi_i(\psi_j^{-1}(a,v)))$
is $C^r_\K$ for
$i,j\in \{1,2\}$, where $\pr_2\colon
A^\times\times A\to A$ is the projection
onto the second factor.
Then also $g_{ij}\colon A^\times\to L(A)_c=L(A)_b$,
$a\mto G_{ij}(a,\cdot)$ is $C^r_\K$, by
Proposition~\ref{reallyneed1}\,(a).
Now, $A$ being Fr\'{e}chet and thus
barrelled, the evaluation homomorphism
$\eta_{A,b}$ is continuous;
since $A$ is metrizable
and hence a $k$-space, also
$\eta_{A,c}$ is continuous
(see Remark~\ref{lemmem2}).
Since $g_{ij}$ is $C^r_\K$,
we deduce
with Lemma~\ref{lemmem1}
that also
$h_{ij}\colon A^\times \to L(A'_\cS)_\cS$,
$a\mto (g_{ji}(a))'$ is $C^r_\K$.
Define
\[
H_{ij} \colon A^\times \times A'_\cS\to A'_\cS\, \quad
(a,\lambda)\mto h_{ij}(a)(\lambda)
\]
for $i,j\in \{1,2\}$.
Then $H_{12}$ is discontinuous.
To see this, we compose $H_{12}$
with the map $\ev_1\colon
A'_b\to\K$, $\lambda\mto \lambda(1)$
which evaluates functionals at the identity element $1\in A$,
and recall that $\ev_1$ is continuous.
Then
$\ev_1(H_{12}(a,\lambda))=\lambda(g_{21}(a)(1))=\lambda(a)$
for $a\in A^\times$ and $\lambda\in A'$.
However, $A$ being a non-normable
locally convex space, the bilinear, separately continuous
evaluation map
$\ve\!: A\times A'_b \to \K$, $(a,\lambda)\mto \lambda(a)$
is discontinuous, and hence so is its
restriction $\ve|_{A^\times \times A'_b}=\ev_1\circ H_{12}$
to the non-empty
open subset $A^\times \times A'_b$,
as is readily verified.
Now $\ev_1\circ H_{12}$ being discontinuous,
also $H_{12}$ is discontinuous
(and therefore not $C^t_\K$).
As a consequence, also
$\wt{\psi}_1\circ \wt{\psi}_2^{-1}=(\pr_1,H_{12})$
is discontinuous. Summing up:\\[2.5mm]
{\em There
is no canonical vector bundle
structure of class $C^t_\K$ on~$E'$
because the two vector bundle structures
on $E'$ making $\wt{\psi}_1$ $($resp., $\wt{\psi}_2)$
a global trivialization do not coincide.}
\end{example}
\begin{rem}
In the preceding situation,
set $M:=A^\times$, $F:=A'_b$,
$I:=\{1,2\}$, $U_i:=M$ for $i\in I$,
and $\pi:=\pr_1\colon M\times F\to M$.
If we let 
$M\times A'_b$ play the role of $E$ in Proposition~\ref{standfact}
and $\wt{\psi}_i\colon \pi^{-1}(U_i)\to U_i\times F$
the role of $\psi_i$ in Proposition~\ref{standfact}\,(e),
then all of Conditions
(a)--(f) of Proposition~\ref{standfact}
and Condition $\mbox{(g)}'$ of Corollary~\ref{cornew}
are satisfied for $r\in \N_0\cup \{\infty,\omega\}$
(with $\bL:=\K$).
However, there is no $C^r_\K$-vector bundle structure
on~$M\times F$
making each $\wt{\psi}_i$ a trivialization,
as just observed,
i.e., the conclusion of Corollary~\ref{cornew}
becomes false.
\end{rem}
\begin{rem}
Let $\K\in\{\R,\C\}$,
$r\in\N\cup\{\infty,\omega\}$,
$t\in\N_0\cup\{\infty,\omega\}$
with $t\leq r$
and~$M$ be a $C^r_\K$-manifold
modeled on a locally convex space~$Z$.
Then the tangent bundle~$TM$ is a $\K$-vector bundle
of class $C^{r-1}_\K$ over~$M$,
with typical fibre~$Z$.
Pick a locally convex vector topology $\cT$
on~$Z'$.
Let $\cA$ be the set of all
maps $\wt{\psi}$
as in~(\ref{wtpsi}),
with $(Z',\cT)$ in place of $F'_\cS$,
for~$\psi$ ranging through the set
of all local trivializations of~$TM$
(alternatively,
only those of the form $(\pi_{TU},d\phi)$
for charts $\phi\colon U\to V\sub Z$
of~$M$, using the bundle projection
$\pi_{TU}\colon TU\to U$).
Let us say that~$M$ has a \emph{canonical
cotangent bundle} of class $C^t_\K$
with respect to~$\cT$
if $T'M:=\bigcup_{x\in M}(T_xM)'$
admits a $\K$-vector bundle structure
of class $C^t_\K$ over~$M$
with typical fibre $(Z',\cT)$,
which makes each $\wt{\psi}\colon p^{-1}(U)\to U\times (Z',\cT)$
a local trivialization
(with $p\colon T'M\to M$, $(T_xM)'\ni\lambda\mto x$).
Then the evaluation map
\[
\ve\colon (Z',\cT)\times Z\to \K,\quad (\lambda,x)\mto \lambda(x)
\]
must be continuous and hence~$Z$ normable.
For $\K=\R$,
this is explained in \cite[Remark~1.3.9]{NSU}
(written after Example~\ref{nodual} was found) if $r=\infty$.
This implies the case $r\in\N$.
As the diffeomorphism~$f$ employed as a change of charts
is real analytic,
the case $(\omega,\R)$ follows and also the
complex case, using a $\C$-analytic extension of~$f$.
When $\cT$ is the compact-open topology, existence of a canonical
cotangent bundle for~$M$ even implies that~$Z$
is finite dimensional.\footnote{If $\ve$ is continuous on
$Z'_c\times Z$, then there exists a compact subset $K\sub Z$
and a $0$-neighbourhood $W\sub Z$ such that $\ve((K^\circ)\times W)\sub \bD$.
Hence $K^\circ\sub W^\circ$. Since $K^\circ$ is a $0$-neighbourhood in~$Z'_c$
and $W^\circ$ compact (by Ascoli's Theorem),
$Z'_c$ is locally compact and hence finite dimensional.
As $Z'_c$ separates points on~$Z$, also~$Z$
must be finite dimensional.}\vspace{-1mm}
\end{rem}
Cotangent bundles
are not needed to define
$1$-forms on an infinite-dimensional manifold~$M$.
Following~\cite{Beg},
these can be considered
as smooth maps on $TM$ which are linear
on the fibres (and a similar remark applies
to differential forms of higher order).
\begin{center}
{\bf Differentiability properties of the {\boldmath$G$}-action
on the dual bundle}
\end{center}
Let $\bL\in\{\R,\C\}$, $\K\in\{\R,\bL\}$,
$s\in \{\infty,\omega\}$,
$r\in\N_0\cup\{\infty,\omega\}$
with $r\leq s$,
and $G$ be a $C^s_\K$-Lie group
modeled on a locally convex $\K$-vector space~$Y$.
Let $M$ be a $C^r_\K$-manifold
modeled on a locally convex $\K$-vector space~$Z$
and $\alpha\colon G\times M\to M$ be a $G$-action
of class~$C^r_\K$.
\begin{prop}\label{actondual}
Let $\pi\colon  E\to M$ be an
equivariant $\bL$-vector bundle of class $C^r_\K$,
with typical fibre~$F$ and $G$-action $\beta\colon G\times E\to E$
of class $C^r_\K$. Let
$\cS\in \{b,c\}$. If $\cS=c$, set $r_-:=r$;
if $\cS=b$, assume $r\geq 1$ and set $r_- :=r-1$.
Consider the following conditions:
\begin{itemize}
\item[\rm(a)]
$\eta_{F,\cS}$ is continuous,
and, moreover,
$(Y\times Z\times F'_\cS)\times (Y\times Z\times F'_\cS)$
is a $k_\R$-space,
or $r_- =0$ and $Y\times Z\times F'_\cS$ is a $k_\R$-space, or
$(r,\K)=(\infty,\C)$ and $Y\times Z\times F'_\cS$ is a $k_\R$-space;
\item[\rm(b)]
$M$ and $G$ are finite dimensional,
$\eta_{F,\cS}$ is continuous, and $F'_\cS$
is barrelled; or
\item[\rm(c)]
$F$ is normable.
\end{itemize}
If $\cS=c$ and {\rm (a)} or {\rm (b)} holds,
then $E$ has a canonical dual bundle~$E'$
of class $C^{r_-}_\K$
with respect to $\cS$,
and the map $\beta^*\colon G\times E'\to E'$,
defined using adjoint linear maps
via
\[
\beta^*(g,\lambda):=
(\beta(g^{-1},\cdot)|_{E_{\alpha(g,x)}}^{E_x})'(\lambda)
\]
for $g\in G$, $\lambda\in (E_x)'$,
turns $E'$ into
an equivariant $\bL$-vector bundle of class
$C^{r_-}_\K$ over the $G$-manifold~$M$.
If $\cS=b$
and {\rm (a), (b)}, or {\rm (c)} is satisfied,
then the same conclusion holds.
\end{prop}
\begin{proof}
In view of Proposition~\ref{finalprop},
the hypotheses imply that~$E$ has a canonical dual bundle $p\colon
E'\to M$ of class $C^{r_-}_\K$.
It is apparent that $\beta^*\colon G\times E'\to E'$ is an action,
and $E_x'$
is taken
$\bL$-linearly to $E_{\alpha(g,x)}'$
by $\beta^*(g,\cdot)$,
for each $g \in G$ and $x\in M$.
It therefore only remains to show that~$\beta^*$
is~$C^{r_-}_\K$.
To this end, let $g_0\in G$ and $x_0\in M$;
we show that~$\beta^*$
is~$C^{r_-}_\K$ on $U\times p^{-1}(V)$,
for some open neighbourhood~$U$
of~$g_0$ in $G$ and an open neighbourhood~$V$ of~$x_0$ in~$M$.
Indeed, there exists a local trivialization
$\psi\colon \pi^{-1}(W)\to W\times F$ of~$E$
over an open neighbourhood~$W$ of $\alpha(g_0,x_0)$
in~$M$. The action $\alpha$ being continuous,
we find an open neighbourhood~$U$ of~$g_0$
in~$G$ and an open neighbourhood~$V$ of~$x_0$ in~$M$
over which~$E$ is trivial, such that $\alpha(U\times V)\sub W$.
Let $\phi\colon \pi^{-1}(V)\to V\times F$
be a local trivialization of~$E$ over~$V$.
Then
\[
\phi(\beta(g^{-1},\psi^{-1}(\alpha(g,x),v)))=(x,A(g,x,v))\quad
\mbox{for all $g\in U$, $x\in V$, and $v\in F$,}
\]
for a $C^r_\K$-map
$A\colon U\times V\times F\to F$
which is $\bL$-linear in the third argument.
By Corollary~\ref{reallyneed1},
the map $a\colon U\times V\to L(F)_\cS$,
$(g,x)\mto A(g,x,\cdot)$ is $C^{r_-}_\K$.
In view of the hypotheses,
Lemmas~\ref{lemmem1}
and~\ref{lemmem2} entail that also $a^*\colon U\times V\to L(F'_\cS)_\cS$,
$(g,x)\mto (a(g,x))'$ is $C^{r_-}_\K$-map.
Now, again using the specific hypotheses,
Proposition~\ref{reallyneed2} shows
that also the mapping $A^*\colon U\times V\times F'_\cS\to F'_\cS$,
$(g,x,\lambda)\mto a^*(g,x)(\lambda)$ is $C^{r_-}_\K$.
However, for $g\in U$, $x\in V$, and
$\lambda\in F'$, we calculate
\begin{eqnarray*}
\wt{\psi}
(\beta^*(g,\wt{\phi}^{-1}(x,\lambda)))
& = & \left(\alpha(g,x),\,
\left(\pr_F\circ \phi|_{E_x}\circ \beta(g^{-1},\cdot)|_{E_{\alpha(g,x)}}^{E_x}
\circ (\pr_F\circ \psi|_{E_{\alpha(g,x)}})^{-1}\right)'(\lambda)\right)\\
& = & (\alpha(g,x),A^*(g,x,\lambda))\,,
\end{eqnarray*}
using notation as in (\ref{wtpsi}).
We conclude that $\beta^*|_{U\times p^{-1}(V)}$
is $C^{r_-}_\K$.
\end{proof}
\begin{example}\label{mainexdiff}
For elementary examples, recall that
the group $\Diff(M)$
of all smooth diffeomorphisms
of a $\sigma$-compact, finite-dimensional smooth manifold~$M$
can be made a smooth Lie group,
modeled on the (LF)-space $\Gamma_c(TM)$
of compactly supported smooth vector fields on~$M$
(see~\cite{Mic, DIF, GaN}).
The natural action
$\Diff(M)\times M\to M$
is smooth~\cite{DIF}.
In view of Example~\ref{acttan},
Proposition~\ref{actondual}\,(b),
Proposition~\ref{proptensor}
and
Proposition~\ref{propsubbun},
we readily deduce
that also the natural action
of $\Diff(M)$ on $TM$ is smooth, as well as the natural
actions on $T^*M:=(TM)'$, $TM^{\tensor n}\tensor
(T^*M)^{\tensor m}$ for all $n,m\in \N_0$,
and the natural action on
the subbundles $S^n(T^*M)$
and $\bigwedge^n T^*M$
of $(T^*M)^{\tensor n}$
given by symmetric and exterior powers,
respectively.
\end{example}
As a consequence, also the natural actions
of $\Diff(M)$ on the associated spaces of smooth
(or smooth compactly supported) sections
are smooth (see \cite{DIF}).
For general results ensuring smoothness of
the action on the space of smooth sections
in a $G$-equivariant vector bundle, see
\cite[Proposition~7.4]{DIF}.
\section{Locally convex Poisson vector spaces}\label{secpoisson}
We discuss a slight generalization
of the concept of
a locally convex Poisson vector space
introduced in~\cite{MEM}. Fix $\K\in \{\R,\C\}$.
\begin{numba}
A bounded set-functor $\cS$
associates with each locally convex $\K$-vector space~$E$
a set $\cS(E)$
of bounded subsets of~$E$,
such that $\{\lambda(M)\colon M\in\cS(E)\}\sub \cS(F)$
for each continuous $\K$-linear
map $\lambda\colon E\to F$ between
locally convex $\K$-vector spaces
(cf.\ \cite[Definition 16.15]{MEM}).
Given locally convex $\K$-vector spaces~$E$ and~$F$,
we shall write $L(E,F)_\cS$
as a shorthand for $L_\K(E,F)_{\cS(E)}$.
We write $E'_\cS:=L_\K(E,\K)_\cS$.
\end{numba}
\begin{numba}\label{good-bounded-sets}
Throughout this section, we let
$\cS$ be a bounded set-functor
such that, for each locally convex space~$E$, we have
$\{K\sub E\colon \mbox{$K$ is compact}\}\sub \cS(E)$.
\end{numba}
Then $\{x\}\in\cS(E)$ for each $x\in E$,
entailing we get a continuous linear point evaluation
\[
\eta_{E,\cS}(x)\colon E'_\cS\to\K,\quad\lambda\mto \lambda(x).
\]
\begin{defn}\label{nonrcase}
A \emph{locally convex Poisson vector space}
with respect to $\cS$ is a locally convex $\K$-vector
space~$E$
such that $E\times E$ is a $k_\R$-space
and
\[
\eta_{E,\cS}\colon E\to (E'_\cS)'_\cS,\quad
x\mto \eta_{E,\cS}(x)
\]
a topological embedding,
together with a
bilinear map
\mbox{$[.,.] \colon E'_\cS \times E'_\cS \to E'_\cS$,}
$(\lambda,\eta)\mto [\lambda,\eta]$
which makes $E'_\cS$ a Lie algebra,
is $\cS(E'_\cS)$-hypocontinuous
in its 2nd argument,
and satisfies
\begin{equation}\label{compcond}
\eta_{E,\cS}(x)\circ \ad_\lambda\; \in \; \eta_{E,\cS}(E)\quad
\mbox{for all $\,x\in E$ and $\lambda\in E'$,}
\end{equation}
writing $\ad_\lambda:=\ad(\lambda):=[\lambda, .]\colon E'\to E'$.
\end{defn}
\begin{rem}
\begin{itemize}
\item[(a)]
Definition 16.35 in \cite{MEM}
was more restrictive; $E$ was assumed
to be a $k^\infty$-space there.
\item[(b)]
In \cite[16.31\,(b)]{MEM},
the following additional condition was
imposed:
For each $M\in \cS(E'_\cS)$
and $N\in \cS(E)$,
the set $\ve(M\times N)$
is bounded in~$\K$, where
$\ve\colon E'\times E\to\K$
is the evaluation map.
As we assume \ref{good-bounded-sets},
the latter condition is automatically
satisfied, by \cite[Proposition~16.11\,(a) and Proposition~16.14]{MEM}.
\item[(c)]
Let us say that a locally convex space~$E$
is \emph{$\cS$-reflexive}
if $\eta_{E,\cS}\colon E\to (E'_\cS)'_\cS$
is an isomorphism of topological vector spaces.
\item[(d)]
Of course, we are mostly interested
in the case where~$[.,.]$ is continuous,
but only hypocontinuity
is required for the basic theory.
\end{itemize}
\end{rem}
\begin{defn}\label{dfnpoivec}
Let $(E,[.,.])$ be a locally
convex Poisson vector space
with respect to~$\cS$,
and $U\sub E$ be open.
Given
$f,g\in C^\infty_\K(U,\K)$,
we define a function
$\{f,g\}\colon U\to\K$ via
\begin{equation}\label{poissonbr}
\{f,g\}(x)\; :=\; \langle [f'(x),g'(x)], x\rangle\quad
\mbox{for $x\in U$,}
\end{equation}
where $\langle .,.\rangle\colon E'\times E\to \K$,
$\langle \lambda,x\rangle :=\lambda(x)$
is the evaluation map and $f'(x)=df(x,.)$.\\[2.5mm]
Condition (\ref{compcond}) in Definition~\ref{nonrcase}
enables us to define
a map $X_f\colon U \to E$ via
\begin{equation}\label{defhamil}
X_f(x)\, :=\,  \eta_{E,\cS}^{-1}\bigl(\eta_{E,\cS}(x)
\circ \ad(f'(x))\bigr)\quad
\mbox{for $\,x\in U$.}
\end{equation}
\end{defn}
\begin{numba}
Using Lemma~\ref{diff-compo-with-hypo}
instead of \cite[Theorem~16.26]{MEM},
we see as in the proof of \cite[Theorem~16.40\,(a)]{MEM}
that the function $\{f,g\}\colon U\to \K$
is~$C^\infty_\K$.
The $C^\infty_\K$-function $\{f,g\}$
is called the \emph{Poisson bracket} of $f$ and $g$.
Using Lemma~\ref{diff-compo-with-hypo}
instead of \cite[Theorem~16.26]{MEM},
we see as in the proof of \cite[Theorem~16.40\,(b)]{MEM}
that $X_f\colon U \to E$ is a $C^\infty_\K$-map;
it is called the \emph{Hamiltonian vector field}
associated with~$f$.
As in \cite[Remark~16.43]{MEM},
we see that the Poisson bracket
just defined makes $C^\infty_\K(U,\K)$
a Poisson algebra.
\end{numba}
\begin{numba}
We shall write ``$b$'' and ``$c$''
in place of~$\cS$ if $\cS$
is the bounded set functor taking a locally convex space~$E$
to the set $\cS(E)$ of all bounded subsets
and compact subsets of~$E$, respectively.
Both of these satisfy the hypothesis of~\ref{good-bounded-sets}.
\end{numba}
In the following, we describe
new results for locally convex Poisson vector spaces
over $\cS=c$. We mention that
the embedding property
of $\eta_{E,c}$ is automatic
in this case, as $E\times E$
is a $k_\R$-space
in Definition~\ref{dfnpoivec};
thus~$E$ is a $k_\R$-space
and
Remark~\ref{lemmem2} applies.
\begin{example}
Let
$(\cg_j)_{j\in J}$
be a family of finite-dimensional real Lie algebras~$\cg_j$.
Endow $\cg:=\bigoplus_{j\in J}\cg_j$
with the locally convex direct sum
topology, which coincides with the
finest locally convex vector topology.
Then $\cg$ is $c$-reflexive,
like every vector space with its finest
locally convex vector topology
(see \cite[Theorem~7.30\,(a)]{HaM}).
As a consequence, also $\cg'_c$
is $c$-reflexive
(cf.\ \cite[Proposition 7.9\,(iii)]{HaM}).
Using
\cite[Proposition 7.1]{MEA},
we see that the component-wise
Lie bracket $\cg\times \cg\to\cg$
is continuous on the locally convex space
$\cg\times\cg$,
which is naturally isomorphic to the locally convex direct sum
$\bigoplus_{j\in J}(\cg_j\times\cg_j)$.
We set $E:=\cg_c'$ and give $E'_c$ the continuous Lie bracket
$[.,.]$
making $\eta_{\cg,c}\colon \cg\to (\cg'_c)'_c=E'_c$
an isomorphism of topological Lie algebras.
Then
\[
E=\cg'_c\cong \prod_{j\in J}(\cg_j)'_c
\]
is a $k_\R$-space, being a cartesian product
of locally compact spaces (see \cite{Nob} or \cite{GaM}).
Thus $(E,[.,.])$
is a locally convex Poisson vector space
over $\cS=c$,
in the sense of Definition~\ref{nonrcase}.
If $J$ has cardinality $\geq 2^{\aleph_0}$
and $\cg_j\not=\{0\}$ for all $j\in J$
(e.g., if we take an abelian $1$-dimensional
Lie algebra $\cg_j$
for each $j\in J$),
then $E\cong \R^J$ is not a $k$-space.
Hence $E$ is not a $k^\infty$-space,
and hence it is not a Poisson vector space
in the more restrictive sense of~\cite{MEM}.
\end{example}
\section{Continuity properties of the Poisson bracket}\label{nwsec1}
If $E$ and $F$ are locally convex $\K$-vector
spaces and $U\sub E$ an open subset, we endow
$C^\infty(U,F)$
with the compact-open $C^\infty$-topology.
Our goal is the following result:
\begin{thm}\label{poiss2}
Let $(E,[.,.])$ be a locally
convex Poisson vector space with respect to
$\cS=c$.
Let $U\sub E$ be open.
Then the Poisson bracket
\[
\{.,.\}\colon C^\infty_\K(U,\K)\times
C^\infty_\K(U,\K)\to C^\infty_\K(U,\K)
\]
is $c$-hypocontinuous in its second variable.
If $[.,.]\colon E'_c\times E'_c\to E'_c$
is continuous, then also
the Poisson bracket
is continuous.
\end{thm}
Various auxiliary results are
needed to prove
Theorem~\ref{poiss2}.
With little risk of confusion
with subsets of spaces
of operators,
given a $0$-neighbourhood $W\sub F$
and a compact set $K\sub U$ we shall write
$\lfloor K, W\rfloor:=\{f \in C(U,F)\colon
f(K)\sub W\}$.
\begin{la}\label{hlp0}
Let $E, F$ be locally convex spaces
and $U\sub E$ be open.
Then the linear map
\[
D\colon C^\infty_\K(U,F)\to C^\infty_\K(U, L(E,F)_c)\,,\quad
f\mto f'
\]
is continuous.
\end{la}
\begin{proof}
By Corollary~\ref{dfisC},
$f'\in C^\infty_\K(U,L(E,F)_c)$
for each $f\in C^\infty_\K(U,F)$.
As $D$ is linear
and also $C^\infty(U,L(E,F)_c)\to
C(U\times E^k,L(E,F)_c)$,
$f \mto d^k f$
is linear for each $k\in \N_0$,
\begin{equation}\label{djD}
d^k\circ D\colon C^\infty(U,F)\to C(U\times E^k,L(E,F)_c)_{c.o.}
\end{equation}
is linear,
whence it will be continuous if it is continuous
at~$0$.
We pick a typical $0$-neighbourhood
in $C(U\times E^k,L(E,F)_c)_{c.o.}$,
say $\lfloor K,V\rfloor$
with a compact subset $K\sub U\times E^k$
and a $0$-neighbourhood $V\sub L(E,F)_c$.
After shrinking~$V$, we may assume
that $V=\lfloor A ,W\rfloor$
for some compact set $A\sub E$ and $0$-neighbourhood
$W\sub F$.\\[2.5mm]
We now recall that for $f\in C^\infty_\K(U,F)$,
we have
\begin{equation}\label{usfrec}
d^k(f')(x,y_1,\ldots,y_k)
\; =\; d^{k+1}f(x,y_1,\ldots, y_k,\cdot)\colon E\to F
\end{equation}
for all $k\in \N_0$, $x\in U$ and
$y_1,\ldots, y_k\in E$
(cf.\ Corollary~\ref{dfisC}).
Since $\lfloor K\times A, W\rfloor$
is an open $0$-neighbourhood in $C(U\times E^{k+1},F)$
and the map $C^\infty(U,F)
\to C(U\times E^{k+1},F)_{c.o.}$, $f \mto d^{k+1}f$
is continuous, we see that the set $\Omega$
of all $f \in C^\infty(U,F)$
such that $d^{k+1}f \in\lfloor K\times A, W\rfloor$
is a $0$-neighbourhood in
$C^\infty(U,F)$.
In view of (\ref{usfrec}),
we have $d^k(f')\in \lfloor K,\lfloor A,W\rfloor\rfloor$
for each $f \in\Omega$.
Hence $d^k\circ D$ from~(\ref{djD})
is continuous at~$0$, as required.
\end{proof}
\begin{la}\label{wannerf}
Let $X$ be a Hausdorff topological space,
$F$ be a locally convex space,
$K\sub X$ be compact
and $M\sub C(X,F)_{c.o.}$ be compact.
Let $\ev\colon C(X,F)\times X\to F$,
$(f,x)\mto f(x)$ be the evaluation map. 
Then $\ev(M\times K)$
is compact.
\end{la}
\begin{proof}
The
map $\rho\colon C(X,F)_{c.o.}
\to C(K,F)_{c.o.}$, $f\mto f|_K$
is continuous by \cite[\S3.2\,(2)]{Eng}.
Thus $\rho(M)$ is compact
in $C(K,F)_{c.o.}$.
The map
$\ve\colon C(K,F)\times K\to F$, $(f,x)\mto f(x)$
is continuous by \cite[Theorem~3.4.2]{Eng}.
Hence
$\ev(M\times K)
=\ve(\rho(M)\times K)$ is compact.
\end{proof}
\begin{la}\label{hlp1}
Let $E$, $F_1$, $F_2$, and~$G$
be locally convex $\K$-vector spaces
and $\beta\colon F_1\times F_2\to G$
be a bilinear map which is $c$-hypocontinuous
in its second argument.
Let $U\sub E$ be an open subset and $r\in \N_0\cup\{\infty\}$.
Assume that $E\times E$ is a $k_\R$-space,
or $r=0$ and~$E$ is a $k_\R$-space,
or $(r,\K)=(\infty,\C)$ and~$E$ is a $k_\R$-space.
Then the following holds:
\begin{itemize}
\item[\rm(a)]
We have $\beta\circ (f,g)\in C^r_\K(U,G)$
for all $(f,g)\in C^r_\K(U,F_1)\times C^r_\K(U,F_2)$.
The map
\[
C^r_\K(U,\beta )\colon C^r_\K(U,F_1)\times C^r_\K(U,F_2)\to
C^r_\K(U,G),
\,\;(f,g)\mto \beta\circ (f, g)
\]
is bilinear. For each compact subset $M\sub C^r_\K(U,F_2)$
and $0$-neighbourhood $W\sub C^r_\K(U,G)$,
there is a $0$-neighbourhood
$V\sub C^r_\K(U,F_1)$ such that $C^r_\K(U,\beta)(V\times M)\sub W$.
\item[\rm(b)]
For each $g\in C^r_\K(U,F_2)$,
the map $C^r_\K(U,F_1)\to C^r_\K(U,G)$,
$f\mto \beta\circ (f,g)$ is continuous and linear.
\item[\rm(c)]
If $\beta$ is also $c$-hypocontinuous
in its first argument,
then $C^r_\K(U,\beta)$
is $c$-hypocontinuous
in its second argument
and $c$-hypocontinuous in its first
argument.
\item[\rm(d)]
If $\beta$ is continuous,
then $C^r_\K(U,\beta)$ is continuous.
\end{itemize}
\end{la}
\begin{proof}
(a) By Lemma~\ref{diff-compo-with-hypo},
$\beta\circ (f,g)\in C^r_\K(U,G)$.
The bilinearity of $C^r(U,\beta)$ is clear.
It suffices to prove the remaining assertion
for each $r\in \N_0$. To see this,
let $M\sub C^\infty_\K(U,F_2)$ be a compact subset
and $W\sub C^\infty_\K(U,G)$ be a $0$-neighbourhood.
Since the topology on
$C^\infty_\K(U,G)$ is initial
with respect to the maps inclusion map $C^\infty_\K(U,G)\to C^r_\K(U,G)$
for $r\in\N_0$,
there exists $r\in\N_0$
and a $0$-neighbourhood~$Q$ in $C^r_\K(U,G)$
such that $C^\infty_\K(U,G)\cap Q\sub W$.
If the assertion holds for~$r$,
we find a $0$-neighbourhood $P\sub C^r_\K(U,F_1)$
such that $C^r_\K(U,\beta)(P\times M)\sub Q$.
Then $V:=C^\infty_\K(U,F_1)\cap P$
is a $0$-neighbourhood in $C^\infty_\K(U,F_1)$
and $C^\infty_\K(U,\beta)(V\times M)\sub C^\infty_\K(U,G)\cap
C^r_\K(U,\beta)(P\times M)\sub C^\infty_\K(U,G)\cap Q\sub W$.\\[2.5mm]
The case $r=0$.
Let $M\sub C(U,F_2)$ be compact
and $W\sub C(U,G)$ be a
$0$-neighbourhood.
Then $\lfloor K, Q\rfloor\sub W$
for some compact subset $K\sub U$ and
some $0$-neighbourhood $Q\sub G$.
By Lemma~\ref{wannerf},
the set $N:=\ev(M\times K)\sub F_2$
is compact, where $\ev\colon C(U,F_2)\times U\to F_2$
is the evaluation map.
Since~$\beta$ is $c$-hypocontinuous in its second argument,
there exists a $0$-neighbourhood
$P\sub F_1$ with
$\beta(P\times N )\sub Q$.
Then
$\beta \circ (\lfloor K, P \rfloor\times M) \sub
\lfloor K, Q\rfloor\sub W$.\\[2.5mm]
Induction step.
Let $M\sub C^r_\K(U,F_2)$
be a compact subset and $W\sub C^r_\K(U,G)$
be a $0$-neighbourhood.
The topology on $C^r(U,G)$
is initial with respect to the linear
maps $\lambda_1\colon C^r_\K(U,G)\to C(U,G)_{c.o.}$,
$f \mto f$
and $\lambda_2\colon C^r_\K(U,G)\to C^{r-1}_\K(U\times E,G)$,
$f\mto df$
(by \cite[Lemma~A.1\,(d)]{ZOO}).\footnote{Note that
the ordinary $C^r$-topology is used there, by
\cite[Proposition 4.19\,(d) and Lemma~A2]{ZOO}.}
After shrinking~$W$, we may therefore
assume that
\[
W=(\lambda_1)^{-1}(W_1)\cap (\lambda_2)^{-1}(W_2)
\]
with absolutely convex $0$-neighbourhoods $W_1\sub C(U,G)$
and $W_2\sub C^{r-1}_\K(U\times E,G)$.
Applying the case $r=0$ to $C(U,\beta)$,
we find a $0$-neighbourhood $V_1\sub C(U,F_1)$
such that $C(U,\beta)(V_1\times M)\sub W_1$.
The map $\delta_j \colon C^r_\K(U,F_j)\to C^{r-1}_\K(U\times E,F_j)$,
$f\mto df$
is continuous linear
and $\pi\colon U\times E\to U$, $(x,y)\mto x$
is smooth, whence
$\rho_j \colon C^r_\K(U,F_j)\to C^{r-1}_\K(U\times E,F_j)$,
$f\mto f\circ \pi$ is continuous
linear (cf.\ \cite[Lemma~4.4]{ZOO}).
By~(\ref{formula-der}),
we have
\begin{equation}\label{git}
\lambda_2\circ C^r_\K(U,\beta)\;=\;
C^{r-1}_\K(U\times E,\beta)\circ (\delta_1\times \rho_2)
\,+\, C^{r-1}_\K(U\times E,\beta)\circ (\rho_1\times \delta_2)\,.
\end{equation}
The subsets $\rho_2(M)\sub C^{r-1}_\K(U\times E,F_2)$
and $\delta_2(M)\sub C^{r-1}_\K(U\times E,F_2)$
are compact. Using the case $r-1$ (with $U\times E$ in place of~$U$),
which holds as the inductive hypothesis,
we find $0$-neighbourhoods $V_2,V_3\sub C^{r-1}_\K(U\times E,F_1)$
such that $C^{r-1}_\K(U,\beta)(V_2\times\rho_2(M))\sub (1/2)W_2$
and $C^{r-1}_\K(U,\beta)(V_3\times \delta_2(M))\sub (1/2)W_2$.
Then $Q:=(\delta_1)^{-1}(V_2)\cap (\rho_1)^{-1}(V_3)$
is an open $0$-neighbourhood in $C^r_\K(U,F_1)$.
Since $(1/2)W_2+(1/2)W_2=W_2$,
we deduce from~(\ref{git}) that
\[
\lambda_2(C^r_\K(U,\beta)(Q\times M))
\sub C^{r-1}_\K(U\times E,\beta)(V_2\times \rho_2(M))
+C^{r-1}_\K(U\times E,\beta)(V_3\times \delta_2(M))\sub W_2.
\]
Thus $C^r_\K(U,\beta)(Q\times M)\sub (\lambda_2)^{-1}(W_2)$.
Now $V:=V_1\cap Q$ is a $0$-neighbourhood in $C^r_\K(U,F_1)$
such that $C^r_\K(U,\beta)(V\times M)\sub (\lambda_1)^{-1}(W_1)\cap(\lambda_2)^{-1}(W_2)
= W$.

(b) Since $C^r_\K(U,\beta)$ is bilinear,
the map $f\mto \beta\circ (f,g)$ is linear.
Its continuity follows from~(a),
applied with the singleton $M:=\{g\}$.

(c) By (a) just established, the condition in Lemma~\ref{prehypo}\,(a)
is satisfied.
By~(b), the map $C^r_\K(U,\beta)$
is continuous in its first argument.
Interchanging the roles of $F_1$ and $F_2$,
we see that $C^r_\K(M,\beta)$
is also continuous in its second argument
and hence $c$-hypocontinuous in its second argument.
Likewise, $C^r_\K(U,\beta)$ is $c$-hypoocontinuous
in its first argument.

(d) If $\beta$ is continuous
and hence smooth, then $C^r(U,\beta)$
is smooth and hence continuous,
as a very special case of
\cite[Proposition~4.16]{ZOO}.
\end{proof}
{\bf Proof of Theorem~\ref{poiss2}.}
By Lemma~\ref{hlp0},
the mapping $D\colon C^\infty(U,\K)\to C^\infty(U, E'_c)$,\linebreak
$f\mto f'$ is continuous and linear.
By Lemma~\ref{hlp1}\,(c),
the bilinear map
\[
C^\infty(U,[.,.])\colon C^\infty(U,E')\times
C^\infty(U,E')\to C^\infty(U,E')\,,\quad
(f,g)\mto (x\mto [f(x),g(x)])
\]
is $c$-hypocontinuous in its second argument;
if $[.,.]$ is continuous,
then also $C^\infty(U,[.,.])$,
by Lemma~\ref{hlp1}\,(d).
The evaluation map
$\beta \colon E\times E'_c\to \K$, $(x,\lambda)\mto \lambda(x)$
is $c$-hypocontinuous in its first argument,
by Proposition~\ref{resulteval}.
As a consequence, $\beta_*\colon C^\infty(U,E'_c)\to C^\infty(U,\K)$,
$f \mto \beta\circ (\id_U,f)$
is continuous linear by Lemma~\ref{hlp1}\,(b).
Since
\[
\{.,.\}\, =\, \beta_*\circ C^\infty(U,[.,.])
\circ (D\times D)
\]
by definition,
we see that $\{.,.\}$ is a composition of
continuous maps if~$[.,.]$ is continuous,
and hence continuous.
In the general case,
$\{.,.\}$ is a composition of a
bilinear map which is $c$-hypocontinuous in
its second argument and continuous linear maps,
whence $\{.,.\}$ is $c$-hypocontinuous in its second arguemnt.\,\Punkt
\section{Continuity of the map taking {\boldmath $f$} to the Hamiltonian vector field {\boldmath $X_f$}}\label{nwsec2}
In this section, we show continuity of the
mapping which takes
a smooth function to the corresponding Hamiltonian vector
field, in the case $\cS=c$.
\begin{thm}\label{poiss3}
Let $(E,[.,.])$ be a locally
convex Poisson vector space with respect to
$\cS=c$.
Let $U\sub E$ be an open subset.
Then the map
\begin{equation}\label{thmap}
\Psi \colon C^\infty_\K(U,\K)\to C^\infty_\K(U,E)\,,
\quad f\mto X_f
\end{equation}
is continuous and linear.
\end{thm}
\begin{proof}
Let $\eta_E\colon E\to (E'_c)'_c$
be the evaluation homomorphism
and
$V:= \{A\in L(E'_c,E'_c)\colon$
$(\forall x\in E)\; \eta_E(x)\circ A\in
\eta_E(E)\}$.
Then~$V$ is a vector subspace of $L(E'_c,E'_c)$
and $\ad(E')\sub V$.
The composition map
$\Gamma\colon (E'_c)'_c\times L(E'_c,E'_c)_c\to
(E'_c)'_c$, $(\alpha ,A)\mto \alpha\circ A$
is hypocontinuous with respect to equicontinuous
subsets of $(E'_c)'_c$,
by Proposition~9 in \cite[Chapter~III, \S5, no.\,5]{Bou}.
If $K\sub E$ is compact, then the polar
$K^\circ$ is a $0$-neighbourhood in~$E'_c$,
entailing that $(K^\circ)^\circ\sub (E'_c)'$
is equicontinuous.
Hence $\eta_E$ takes
compact subsets of~$E$ to equicontinuous subsets
of $(E'_c)'$, and hence
\[
\beta\colon E \times V \to E\,,\quad
(x,A) \mto \eta_E^{-1}(\Gamma(\eta_E(x), A))
\]
is $c$-hypocontinuous in its first argument.
By Lemma~\ref{hlp1}\,(c),
$\beta_*\colon
C^\infty(U,V)\to C^\infty(U,E)$, $f\mto \beta\circ (\id_U,f)$
is continuous linear.
Also the map
$D\colon C^\infty(U,\K)\to C^\infty(U,E'_c)$,
$f\mto f'$ is continuous linear by Lemma~\ref{hlp0}.
Furthermore,
$\ad=[.,.]^\vee \colon E'_c\to L(E'_c,E'_c)_c$
is continuous linear since
$[.,.]$ is $c$-hypocontinuous in its second argument
(see Lemma~\ref{prehypo}\,(b)),
whence
\[
C^\infty(U,\ad)\colon C^\infty(U,E'_c)\to
C^\infty(U,L(E'_c,E'_c)_c)\,,\quad
f \mto \ad\circ \, f
\]
is continuous linear
(see, e.g., \cite[Lemma~4.13]{ZOO}).
Hence
$\Psi= \beta_*\circ C^\infty(U,\ad)\circ D$
is continuous and linear.
\end{proof}
\appendix
\section{Proofs for basic facts in Section~\ref{secprelim}}\label{appA}
{\bf Proof of Lemma~\ref{wippe}.}
Let $E:=E_1\times\cdots \times E_k$.
Since $df\colon U\times E\times X\times E\to F$
is continuous and $df(x,0,0,0)=0$,
given~$q$ there
exist a continuous seminorm~$p$ on~$X$
such that $B^p_1(x)\sub U$, and
continuous seminorms~$p_j$ on~$E_j$
for $j\in\{1,\ldots, k\}$ such that
\begin{equation}\label{wuse3}
\| df(y,v_1,\ldots, v_k,z,w_1,\ldots, w_k)\|_q\leq 1
\end{equation}
for all $v_j,w_j\in B^{p_j}_1(0)$,
$y\in B^p_1(x)$, and $z\in B^p_1(0)$.
For $y\in B^p_1(x)$ and
$(v_1,\ldots, v_k)\in B^{p_1}_1(0)\times\cdots \times B^{p_k}_1(0)$,
the Mean Value Theorem (see \cite[Proposition 1.2.6]{GaN})
shows that
\[
f(y,v_1,\ldots, v_k)=\int_0^1 df(y,tv_1,\ldots,tv_k,0,v_1,\ldots, v_k)\, dt\,.
\]
Since $\|df(y,tv_1,\ldots,tv_k,0,v_1,\ldots, v_k)\|_q\leq 1$
for each~$t$, it follows that
$\|f(y,v_1,\ldots, v_k)\|_q\leq 1$ in the preceding situation.
Because $f(y,\cdot)$ is $k$-linear,
we deduce that (\ref{wuse1}) holds.
To prove~(\ref{wuse2}),
we first note that (\ref{wuse3})
implies that
\begin{equation}\label{wuse4}
\|df(y,v_1,\ldots, v_k,z,0,\ldots,0)\|_q\leq \|z\|_p
\end{equation}
for all $y\in B^p_1(x)$,
$(v_1,\ldots, v_k)\in B^{p_1}_1(0)\times\cdots \times B^{p_k}_1(0)$
and $z\in X$, exploiting the linearity of
$df(y,v_1,\ldots, v_k,z,0,\ldots,0)$ in~$z$.
We now use
the Mean Value Theorem to write
\[
f(y,v_1,\ldots, v_k)-f(x,v_1,\ldots,v_k)
=\int_0^1 df(x+t(y-x),v_1,\ldots,v_k,y-x,0,\ldots, 0)\, dt
\]
for $y\in B^p_1(x)$ and
$(v_1,\ldots, v_k)\in B^{p_1}_1(0)\times\cdots \times B^{p_k}_1(0)$.
By (\ref{wuse4}), we have
\[
\|df(x+t(y-x),v_1,\ldots,v_k,y-x,0,\ldots, 0)\|_q\leq \|y-x\|_p
\]
and hence
$\|f(y,v_1,\ldots, v_k)-f(x,v_1,\ldots,v_k)\|_q\leq \|y-x\|_p$.
Now (\ref{wuse2}) follows,
using the $k$-linearity
of the map $f(y,\cdot)-f(x,\cdot)\colon E_1\times \cdots\times E_k \to F$.\,\Punkt\\[2.5mm]
{\bf Proof of Lemma~\ref{cont-via-gateaux}.}
By the Polarization Formula for symmetric $k$-linear
maps (see, e.g., \cite[Proposition 1.6.19]{GaN}),
we have
\[
f(x,y_1,\ldots, y_k)=\frac{1}{k!2^k}\sum_{\ve_1,\ldots,\ve_k\in \{{-1},1\}}
\ve_1\cdots \ve_k \, h(x,\ve_1 y_1+\cdots+\ve_ky_k)
\]
for all $x\in U$ and $y_1,\ldots, y_k\in E$.
Thus $f$ is $C^r_\K$ if $h$ is so.\,\Punkt\\[2.3mm]
{\bf Proof of Lemma~\ref{basic-k-R}.}
(a) Let $\pr_2\colon X_1\times X_2\to X_2$, $(x,y)\mto y$
be the projection onto the second component
and pick $x_0\in X_1$.
Since $\pr_2$ is continuous,
every $k$-continuous function $f\colon X_2\to\R$
yields a $k$-continuous function $f\circ \pr_2$
on~$X$. Then $f\circ \pr_2$ is continuous
and hence also $f=(f\circ \pr_2)(x_0,\cdot)$.

(b) Let $f\colon U\to \R$ be $k$-continuous and $x\in U$.
As~$X$ is completely regular,
we find a continuous function $g\colon X\to\R$
with $g(x)\not=0$ and support $\Supp(g)\sub U$.
Define $h\colon X\to\R$ via $h(y):=f(y)g(y)$
if $y\in U$, $h(y):=0$ if $y\in X\setminus\Supp(g)$.
If $K\sub X$ is a compact subset,
then each $x\in K$ has a compact neighbourhood
$K_x$ in~$K$ which is contained in $U$
or in $X\setminus \Supp(g)$.
In the first case, $h|_{K_x}=f|_{K_x}g|_{K_x}$
is continuous by $k$-continuity of~$f$.
In the second case, $h|_{K_x}=0$
is continuous as well. Thus $h|_K$
is continuous. Since $X$ is a $k_\R$-space,
continuity of $h$ follows.
Thus $f$ is continuous on the open $x$-neighbourhood
$g^{-1}(\R\setminus\{0\})$.\,\Punkt\\[2.5mm]
A simple fact will be useful (see, e.g.,
\cite[Lemma~1.13]{MEM}).
\begin{la}\label{pla}
Let~$X$ be a topological space, $F$ be a locally convex space,
and $BC(X,F)$ be the space of bounded $F$-valued
continuous functions on~$X$,
endowed with the topology of uniform convergence.
Then
$\mu\colon BC(X,F)\times X\to F$,
$(f,x)\mto f(x)$
is continuous.\,\Punkt
\end{la}
{\bf Proof of Lemma~\ref{prehypo}.}\footnote{If $k=2$,
see Proposition~3 and~4
in \cite[Chapter~III, \S5, no.\,3]{Bou}
for the equivalence
(a)$\aeq$(b)
and the implication (b)$\impl$(c);
(c)$\impl$(a) can be found in \cite[Proposition~1.8]{MEM}.}
(a)$\aeq$(b): $\beta(V\times M)\sub W$ is equivalent to
$\beta^\vee(V)\in \lfloor M,W\rfloor$.
Hence (a) is equivalent to continuity
of $\beta^\vee$ in~$0$ and hence to its
continuity (see Proposition~5 in
\cite[Chapter I, \S1, no.\,6]{Bou}).\vspace{1mm}

(b)$\impl$(c):
If $M\in \cS$,
then $\ve\colon L^{k-j+1}(E_j,\ldots, E_k,F)_\cS\times M\to F$,
$\ve(\alpha,x):=\alpha(x)$ is continuous
as a consequence of Lemma~\ref{pla}.
Hence $\beta|_{E_1\times\cdots\times E_{j-1}\times M}=\ve\circ
(\beta^\vee\times \id_M)$ is continuous.\vspace{1mm}

(c)$\impl$(a) if (\ref{simplf}) holds:
Given $M\in \cS$ and a $0$-neighbourhood
$W\sub F$, by hypothesis
we can find $N\in \cS$ such that
$\bD M\sub N$.
By continuity
of $\beta|_{E_1\times \cdots\times E_{j-1}\times N}$,
there exist
$0$-neighbourhoods $V_i\sub E_i$
for $i\in \{1,\ldots, k\}$
such that $\beta(V\times (N\cap U))\sub W$,
where $V:=V_1\times \cdots\times V_{j-1}$
and $U:=V_j\times\cdots\times V_k$.
Set $a:=\frac{j-1}{k-j+1}$.
Since~$M$ is bounded,
$M \sub n^a U$ for some $n\in \N$.
Then $\frac{1}{n^a}M\sub N\cap U$.
Using that~$\beta$ is $k$-linear, we obtain
$\beta((\frac{1}{n}V)\times M)=\beta(V \times(\frac{1}{n^a}M))
\sub \beta(V\times (N\cap U))\sub W$.\,\vspace{2.3mm}\Punkt

\noindent
{\bf Proof of Lemma~\ref{resulteval}.}
Given $\alpha \in L^k(E_1,\ldots,E_k,F)$, we have
$\ve^\vee(\alpha)=\ve(\alpha,\cdot)=\alpha$,
which is a continuous $k$-linear map.
The map~$\ve$ is also continuous in its first argument, as
the topology on $L^k(E_1,\ldots, E_k,F)_\cS$
is finer than the topology
of pointwise convergence,
by the hypothesis on~$\cS$.
The linear map $\ve^\vee\colon L^k(E_1,\ldots,E_k)_\cS\to L^k(E_1,\ldots, E_k)_\cS$,
$\alpha\mto\alpha$ being continuous,
condition~(b)
of Lemma~\ref{prehypo}
is satisfied by~$\ve$ in place of~$\beta$
and hence also the equivalent
condition~(a), whence
$\ve$ is $\cS$-hypocontinuous
in its arguments $(2,\ldots,k+1)$.\\[2.5mm]
Now assume that $k=1$.
Since~$\cO$ is finer than the topology of
pointwise convergence,
the map~$\ve$ remains separately continuous
in the situation described at the end
of the lemma.
Hence, if~$E$ is barrelled,
Lemma~\ref{auothypo}
ensures hypocontinuity
with respect to~$\cT$.\,\Punkt\\[2.5mm]
{\bf Proof of Lemma~\ref{compo-is-cts}.}
(a) The composition $\beta\circ f$ is sequentially continuous and hence
continuous, its domain $X$ being first countable.\vspace{1mm}

(b) Write $f=(f_1,\ldots, f_k)$
with components $f_j\colon X\to E_\nu$ for $\nu\in \{1,\ldots, k\}$.
If $K$ is a compact subset of~$X$, then
$M:=(f_j,\ldots, f_k)(K)$ is a compact subset of $E_j\times\cdots\times E_k$.
Since $\beta|_{E_1\times\cdots\times E_{j-1}\times M}$
is continuous by Lemma~\ref{prehypo}\,(c),
the composition
\[
\beta\circ f|_K=\beta|_{E_1\times\cdots\times E_{j-1}\times M}\circ f|_K
\]
is continuous. Thus $\beta\circ f$ is $k$-continuous
and hence continuous, as $X$ is a $k_\R$-space
and $F$ is completely regular.\,\Punkt\\[2.5mm]
{\bf Proof of Lemma~\ref{lem-lipdiff}.}
(a) The case $r=0$: Let $q$ be a continuous
seminorm on $F:=\prod_{j\in J}F_j$, and $x\in U$.
After increasing~$q$,
we may assume that
\begin{equation}\label{typisem}
q(y)=\max\{q_j(y_j)\colon j\in \Phi\}\quad\mbox{for all $\, y=(y_j)_{j\in J}\in F$,}
\end{equation}
for some non-empty, finite subset $\Phi\sub J$
and continuous seminorms~$q_j$ on~$F_j$ for $j\in \Phi$.
If each $f_j$ is $LC^0_\K$,
then we find a continuous seminorm~$p_j$ on~$E$ for each $j\in \Phi$
such that $B^{p_j}_1(x)\sub U$ and $q_j(f_j(z)-f_j(y))\leq p_j(z-y)$
for all $z,y\in B^{p_j}_1(x)$.
Then
\[
p\colon E\to [0,\infty[, \quad y\mto\max\{p_j(y)\colon j\in\Phi\}
\]
is a continuous seminorm on~$E$ such that $B^p_1(x)\sub U$ and
$q(f(z)-f(y))\leq p(z-y)$ for all $z,y\in B^p_1(x)$.
If~$f$ is $LC^0_\K$,
let us show that $f_j$ is $LC^0_\K$
for each $j\in J$.
Let $q$ be a continuous seminorm on
$F_j$ and $x\in U$. Let $\pr_j\colon F\to F_j$,
$(y_i)_{i\in J}\mto y_j$ be the continuous linear projection onto the $j$th
component.
Then $q\circ \pr_j$
is a continuous seminorm on~$F$,
whence we find a continuous seminorm~$p$
on~$E$ such that $B^p_1(x)\sub U$ and
$(q\circ \pr_j)(f(z)-f(y))\leq p(z-y)$ for all $z,y\in B^p_1(x)$.
Since $(q\circ \pr_j)(f(z)-f(y))=q(f_j(z)-f_j(y))$,
we see that $f_j$ is $LC^0_\K$.\vspace{1mm}

If $r\in \N\cup\{\infty\}$, then $f$ is $C^r_\K$
if and only if each $f_j$ is $C^r_\K$, and
$d^kf=(d^k f_j)_{j\in J}$ in this case
for all $k\in \N_0$ such that $k\leq r$
(see \cite[Lemma 1.3.3]{GaN}).
By the case $r=0$, the map $d^kf$ is $LC^0_\K$
if and only if $d^k(f_j)$ is $LC^0_\K$
for all $j\in J$. The assertion follows.\vspace{1mm}

(b) Let $E$, $F$, and $Y$ be locally convex $\K$-vector spaces
and $f\colon U\to F$ as well as $g\colon V\to Y$
be $LC^r_\K$-maps on open subsets $U\sub E$ and $V\sub F$,
such that $f(U)\sub V$.\vspace{1mm}

If $r=0$, let $x\in U$ and $q$ be a continuous
seminorm on~$Y$. There exists a continuous seminorm~$p$
on~$F$ such that $B^p_1(f(x))\sub V$ and
$q(g(b)-g(a))\leq p(b-a)$ for all $a,b\in B^p_1(f(x))$.
There exists a continuous seminorm~$P$ on~$E$
with $B^P_1(x)\sub U$ and $p(f(z)-f(y))$
$\leq P(z-y)$
for all $z,y\in B^P_1(x)$.
Then $f(B^P_1(x))\sub B^p_1(f(x))$ and hence
\[
q(g(f(z))-g(f(y)))\leq p(f(z)-f(y))\leq P(z-y)
\]
for all $y,z\in B^P_1(x)$. Thus $g\circ f\colon U\to Y$ is~$LC^0_\K$.

If $r\in\N\cup\{\infty\}$ and $k\in \N$ such that $k\leq r$,
we can use Fa\`{a} di Bruno's Formula
\begin{equation}\label{use-faa}
d^k(g\circ f)(x,y)
=\sum_{j=1}^k\sum_{P\in P_{k,j}}d^jg(f(x), d^{|I_1|}(x,y_{I_1}),\ldots,
d^{|I_j|}(x,y_{I_j}))
\end{equation}
for $x\in U$ and $y=(y_1,\ldots, y_k)\in E^k$,
as in \cite[Theorem 1.3.18]{GaN}.
Here $P_{k,j}$ is the set of all partitions
$P=\{I_1,\ldots, I_j\}$ of $\{1,\ldots, k\}$
into $j$ disjoint, non-empty subsets $I_1,\ldots, I_j\sub\{1,\ldots, k\}$.
For a non-empty subset $J\sub \{1,\ldots,k\}$
with elements $j_1<\cdots< j_m$,
let $y_J:=$ $(y_{j_1},\ldots, y_{j_m})$.
Using (a) and the case $r=0$,
we deduce from~(\ref{use-faa}) that $d^k(g\circ f)$
is~$LC^0_\K$.\vspace{1mm}

(c) For each continuous seminorm~$q$ on~$F$,
the restriction $q|_{F_0}$
is a continuous seminorm on~$F_0$,
and each continuous seminorm~$Q$ on~$F_0$
arises in this way.
In fact, we find an open, absolutely convex
$0$-neighbourhood $V\sub F$
such that $V\cap F_0\sub B^Q_1(0)$.
Then the absolutely convex hull~$W$
of $V\cup B^Q_1(0)$ is a $0$-neighbourhood
in~$F$ with $W\cap F_0=B^Q_1(0)$,
whence $q|_{F_0}=Q$ holds for the Minkowski funtional~$q$
of~$W$.
The case $r=0$ follows.\vspace{1mm}

If $r\in\N\cup\{\infty\}$,
let $\iota\colon F_0\to F$ be the inclusion map
and $f\colon U\to F_0$
be a map on an open subset $U\sub E$.
Then $f$ is $C^r_\K$ if and only
if $\iota\circ f$ is $C^r_\K$,
and $d^k(\iota\circ f)=\iota\circ (d^kf)$
for all $k\in \N_0$ such that $k\leq r$
(see \cite[Lemma~1.3.19]{GaN}). By the case $r=0$,
each of the maps $d^kf$ is $LC^0_\K$
if and only if $\iota\circ (d^kf)$ is so,
from which the assertion follows.\vspace{1mm}

(d) is immediate from (a) and (c).\, \Punkt\\[2.5mm]
{\bf Proof of Lemma~\ref{new-Ck-top-on-multi}.}
For each $x\in E_1\times\cdots\times E_k=:E$,
the point evaluation $\ev_x\colon C^r_\K(E,F)\to F$, $f\mto f(x)$
is linear and continuous,
the compact-open $C^r$-topology on $C^r_\K(E,F)$ being finer than the compact-open topology.
Now $L^k_\bL(E_1,\ldots, E_k,F)$
is closed in $C^r_\K(E,F)$, being the intersection of the closed sets
\[
\{f\in C^r_\K(E,F)\colon
\ev_z(f)- a\ev_x(f)- b \ev_{x'}(f)=0\}
\]
for $j\in \{1,\ldots, k\}$, $x=(x_1,\ldots, x_k)\in E$,
$x':=(x_1,\ldots, x_{j-1},x_j',x_{j+1},\ldots, x_k)$
with $x_j'\in E_j$
and $z:=(x_1,\ldots, x_{j-1},a x_j+ b  x_j',x_{j+1},\ldots,x_k)$
with $a,b\in \bL$.

Let $\iota\colon L^k_\bL(E_1,\ldots, E_k,F)\to C^k_\K(E,F)$
be the inclusion map. Then $d^0\circ \iota$
is the inclusion map $L^k_\bL(E_1,\ldots, E_k,F)\to C(E,F)_{c.o.}$,
which is a topological embedding (and hence continuous).
We claim: For all
$j\in\N$ such that $j\leq r$,
there are an $m_j\in \N$
and $C^\infty_\K$-maps
$g_{j,\mu}\colon E^{j+1}\to E$
for $\mu\in\{1,\ldots,m_j\}$
such that
\begin{equation}\label{highderbeta}
d^j\beta=\sum_{\mu=1}^{m_j}\beta\circ g_{j,\mu}\;\;\,
\mbox{for all $\, \beta\in L^k_\bL(E_1,\ldots, E_k,F)$.}
\end{equation}
If this is true, then $d^j\circ \iota\colon 
L^k_\bL(E_1,\ldots, E_k,F)\to C(E^{k+1},F)_{c.o.}$
is a restriction of the mapping
$\sum_{\mu=1}^{m_j}
(C(g_{j,\mu},F)\circ \iota)$
and hence continuous,
by \cite[Lemma~A.6.9]{GaN}; thus $\iota$ is continuous.
As $d^0\circ\iota$ is a topological embedding,
we deduce that also $\iota$ is a topological embedding.

To prove the claim, we proceed by induction.
For all $\beta\in L^k_\bL(E_1,\ldots, E_k,F)$,
$x=(x_1,\ldots, x_k)\in E$ and $y=(y_1,\ldots, y_k)\in E$,
we have
\begin{equation}\label{diffmultilin}
d\beta(x,y)=\sum_{\nu=1}^k \beta(x_1,\ldots, x_{\nu-1},y_\nu,x_{\nu+1},\ldots,x_k),
\end{equation}
see \cite[Example 1.2.3]{GaN}.
This establishes the claim for $j=1$.
Let $\pi\colon E^{j+2}\to E^{j+1}$
be the continuous linear projection
$(x,y_1,\ldots, y_{j+1})\mto (x,y_1,\ldots, y_j)$
and $\iota\colon E^{j+2}\to E^{2j+2}$
be the continuous linear map
taking $(x,y_1,\ldots,y_{j+1})$ to
$(x,y_1,\ldots,y_{j+1},0,\ldots,0)$.
If $1\leq j<r$ and the claim holds for $j$,
write $g_{j,\mu}=(g_{j,\mu,1},\ldots,g_{j,\mu,k})$
in components
for $\mu\in\{1,\ldots, m_j\}$.
Then
$d^{j+1}\beta=\sum_{\mu=1}^{m_j}\sum_{\nu=1}^k \beta\circ h_{\mu,\nu}$ with
$h_{\mu,\nu}\colon E^{j+2}\to E$,
\[
h_{\mu,\nu}:=(g_{j,\mu,1}\circ \pi,\, \ldots,\, g_{j,\mu,\nu-1}\circ\pi,
\, dg_{j,\mu,\nu}\circ\iota ,\, g_{j,\mu,\nu+1}\circ\pi,\,\ldots,\,
g_{j,\mu, k}\circ\pi).
\]
Thus also $d^{j+1}\beta$ is of the asserted
form and the claim holds for $j+1$.\,\Punkt
\section{Smooth maps need not extend to the completion}\label{appB}
Let $E:=\{(x_n)_{n\in\N}\in \ell^1\colon
(\exists N\in\N)(\forall n\geq N)\, x_n=0\}$ be the space of finite sequences,
endowed with the topology induced by the real Banach space
$\ell^1$ of absolutely summable real sequences.
Then $E$ is a dense proper vector subspace
of~$\ell^1$,
and $\ell^1$ is a completion of~$E$.
In this appendix, we provide a smooth map with
the following pathological properties.
\begin{prop}\label{pathoprop}
There exists a smooth map $f\colon E\to F$
to a complete
locally convex space~$F$
which does not admit a continuous
extension to $E\cup\{z\}$
for any $z\in \ell^1\setminus E$.
\end{prop}
\begin{proof}
Given $z=(z_n)_{n\in\N}\in \ell^1\setminus E$,
the set $S:=\{n\in \N\colon z_n\not=0\}$ is infinite.
For each $n\in \N$, we pick a smooth map
$h_n\colon \R\to\R$ such that
$h_n(z_n)=1$;
if $n\in S$, we also require that
$h_n$ vanishes on some~$0$-neighbourhood.
Endow $\R^\N$ with the product topology.
Then
\[
g\colon \ell^1\to \R^\N\,,\quad x=(x_n)_{n\in \N}\mto
(h_1(x_1)\cdots h_n(x_n))_{n\in \N}
\]
is a smooth map, as its components $g_n\colon
\ell^1\to \R$,
$x\mto h_1(x_1)\cdots h_n(x_n)$ are smooth.
If $x=(x_n)_{n\in\N} \in E$, then there is $N\in S$
such that $x_n=0$ for all $n\geq N$.
Thus $g_n(x)=0$ for all $n\geq N$
and hence $g(x)\in E$.
Notably, $g(x)\in\ell^1$.
It therefore makes sense to define
\[
f_z\colon E\to \ell^1\,,\quad x\mto g(x)\,.
\]
We now show:
\emph{$f_z\colon E\to \ell^1$ is a smooth
map to~$\ell^1$
which does not admit a continuous
extension to $E\cup\{z\}$.}\\[2.5mm]
In fact, for $x$ and~$N$ as above,
there exists $\ve>0$ such that $h_N(t)=0$
for each $t\in \;]{-\ve},\ve[$.
Identify
$\R^N$ with the closed vector subspace $\R^N\times \{0\}$
of~$E$ and $\R^\N$.
Then
\[
U\; :=\; \{y=(y_n)_{n\in \N}\in E\colon |y_N|<\ve\}
\]
is an open neighbourhood of~$x$ in~$E$
such that $f_z(U)\sub\R^N$.
Thus $f_z|_U$ is smooth as a map to~$\R^N$
and hence also as a map to~$\ell^1$.
As a consequence, $f_z\colon E\to \ell^1$ is smooth.\\[2.5mm]
Now suppose that
$p =(p_n)_{n\in \N}\colon E\cup \{z\}\to\ell^1$
was a continuous extension of~$f_z$;
we shall derive a contradiction.
To this end, set $y_k:=(z_1,\ldots,z_k,0,0,\ldots)\in E$
for $k\in \N$.
Then $y_k\to z$ in~$E$ as $k\to \infty$.
The inclusion map $\ell^1\to \R^\N$ being continuous,
we deduce that
\[
p_n(y_k)\to p_n(z)\quad \mbox{as $\, k\to\infty$,}
\]
for each $n\in \N$.
Since $p_n(y_k)=g_n(y_k)=h_1(z_1)\cdots h_n(z_n)=1$
for all $k\geq n$, it follows that
$p_n(z)=1$ for all $n\in \N$
and thus $(1,1,\ldots)=p(z)\in \ell^1$,
which is absurd. Therefore $f_z$ has
all of the asserted properties.\\[2.5mm]
We now define $\Omega:=\ell^1\setminus E$
and endow $F:=(\ell^1)^\Omega$ with the product
topology.
We let $f:=(f_z)_{z\in \Omega}\colon E\to F$
be the map with components $f_z$ as defined
before. By construction, $f$
has the properties described in Proposition~\ref{pathoprop}.
\end{proof}
\noindent
{\footnotesize
{\bf Helge Gl\"{o}ckner}, Universit\"{a}t Paderborn, Institut f\"{u}r Mathematik,
Warburger Str.\ 100, 33098 Paderborn, Germany.
\,Email: {\tt glockner@math.upb.de}

{\small\begin{thebibliography}{10}\itemsep+.2pt
%
%
\bibitem{AaS}
Alzaareer, H. and A. Schmeding,
\emph{Differentiable mappings on products with different
degrees of differentiability in the two factors},
Expo.\ Math.\ {\bf 33} (2015), 184--222. 
%
%
\bibitem{Cha}
Ardanza-Trevijano, S. and M.\,J. Chasco,
\emph{The Pontryagin duality of sequential limits of topological Abelian groups},
J. Pure Appl.\ Algebra {\bf 202} (2005), 11--21. 
%
%
\bibitem{Ban}
Banaszczyk, W.,
``Additive Subgroups of Topological Vector Spaces,''
Springer, Berlin, 1991. 
%
%
\bibitem{Bas}
Bastiani, A.,
\emph{Applications diff\'{e}rentiables et vari\'{e}t\'{e}s diff\'{e}rentiables
de dimension infinie},
J. Anal.\ Math.\ {\bf 13} (1964), 1--114.
%
%
\bibitem{Beg} Beggs, E., \emph{De Rham's theorem for
infinite-dimensional manifolds}, Quart.\ J.\ Math.\ {\bf 38} (1987),
131--154.
%
%
\bibitem{Bel}
Beltiţ\u{a} D.,
T. Goli\'{n}ski, A.-B. Tumpach,
\emph{Queer Poisson brackets},
J. Geom.\ Phys.\ {\bf 132} (2018), 358--362.
%
%
\bibitem{BGN}
Bertram, W., H. Gl\"{o}ckner,
and K.-H. Neeb,
\emph{Differential calculus over general base fields and rings},
Expo.\ Math.\ {\bf 22} (2004), 213--282. 
%
%
\bibitem{BaS} Bochnak, J. and J. Siciak, \emph{Analytic functions
in topological vector spaces}, Studia Math.\ {\bf 39\/} (1971),
77--112.
%
%
\bibitem{Bou} Bourbaki, N., ``Topological Vector Spaces,
Chapters 1-5,'' Springer, Berlin, 1987.
%
%
\bibitem{Eng}
Engelking, R., ``General Topology,''
Heldermann, Berlin, 1989.
%
%
\bibitem{FHS}
Ferrer, M.\,V.,
S. Hern\'{a}ndez, and D. Shakhmatov,
\emph{A countable free closed non-reflexive subgroup of $\Z^c$},
Proc.\ Amer.\ Math.\ Soc.\ {\bf 145} (2017), 3599--3605.
%
%
\bibitem{Fra}
Franklin, S.\,P.
and B.\,V. Smith Thomas,
\emph{A survey of $k_\omega$-spaces},
Topol.\ Proc.\ {\bf 2} (1978), 111--124.
%
% 
\bibitem{RES} Gl\"{o}ckner, H., \emph{Infinite-dimensional
Lie groups without completeness restrictions},
pp.\ 43--59 in: Strasburger, A. et al.\ (eds.),
``Geometry and Analysis on Finite- and Infinite-Dimensional
Lie Groups''
Banach Center Publications, Vol.\,{\bf 55}, Warsaw, 2002.
%
%
\bibitem{GCX} Gl\"{o}ckner, H., \emph{Lie group structures
on quotient groups and universal complexifications
for infinite-dimensional Lie groups},
J. Funct.\ Anal.\ {\bf 194} (2002), 347--409.
%
%
\bibitem{ALG} Gl\"{o}ckner, H.,
\emph{Algebras whose groups of units
are Lie groups}, Studia Math.\ {\bf 153} (2002), 147--177.
%
%
\bibitem{MEA} Gl\"{o}ckner, H., \emph{Lie groups of measurable mappings},
Canadian J. Math.\ {\bf 55} (2003), 969--999.
%
%
\bibitem{FUN}
Gl\"{o}ckner, H.,
\emph{Fundamentals of direct limit Lie theory},
Compos.\ Math.\ {\bf 141} (2005), 1551--1577. 
%
%
\bibitem{IMP} Gl\"{o}ckner, H., \emph{Implicit functions from topological
vector spaces to Banach spaces},
Israel J. Math.\ {\bf 155} (2006), 205--252.
%
%
\bibitem{COM}
Gl\"{o}ckner, H.,
\emph{Direct limits of infinite-dimensional
Lie groups compared to direct limits in related categories},
J. Funct.\ Anal.\ {\bf 245} (2007), 19--61.
%
%
\bibitem{MEM} Gl\"{o}ckner, H., \emph{Applications of hypocontinuous bilinear
maps in infinite-dimensional differential calculus},
pp.\ 171--186 in: S. Silvestrov, E. Paal, V. Abramov and A. Stolin (eds.),
``Generalized Lie Theory in Mathematics, Physics and Beyond,'' Springer, Berlin,
2008.
%
%
\bibitem{DSU}
Gl\"{o}ckner, H.,
\emph{Direct limits of infinite-dimensional Lie groups},
pp.\,243--280 in:
K.-H. Neeb
and A. Pianzola (eds.),
``Developments and Trends in Infinite-Dimensional Lie Theory,''
Birkh\"{a}user, Basel, 2011.
%
%
\bibitem{ZOO}
Gl\"{o}ckner, H., \emph{Lie groups over non-discrete
topological fields}, preprint,
arXiv:math/0408008.
%
%
\bibitem{DMS} Gl\"{o}ckner, H., \emph{Differentiable mappings between
spaces of sections},
preprint, arXiv:1308.1172.
%
%
\bibitem{MRG}
Gl\"{o}ckner, H.,
\emph{Measurable regularity properties of infinite-dimensional Lie groups},
preprint, arXiv:1601.02568.
%
%
\bibitem{SMO}
Gl\"{o}ckner, H.,
\emph{Smoothing operators for vector-valued functions and extension
operators}, preprint, arXiv:2006.00254.
%
%
\bibitem{DIF} Gl\"{o}ckner, H., \emph{Patched locally convex spaces, almost
local mappings, and diffeomorphism groups of non-compact manifolds},
manuscript, 2002.
%
%
\bibitem{GGH}
Gl\"{o}ckner, H.,
R. Gramlich, and T. Hartnick,
\emph{Final group topologies, Kac-Moody groups and
Pontryagin duality},
Isr.\ J. Math.\ {\bf 177} (2010), 49--101.
%
%
\bibitem{GaH}
Gl\"{o}ckner, H. and J. Hilgert,
\emph{Aspects of control theory on infinite-dimensional Lie groups and
$G$-manifolds},
preprint, arXiv:2007.11277.
%
% 
\bibitem{GaM}
Gl\"{o}ckner, H. and N. Masbough,
\emph{Products of regular locally compact spaces are
$k_\R$-spaces},
Topology Proc.\ {\bf 55} (2020), 35--38. 
%
%
\bibitem{GaN} Gl\"{o}ckner, H. and K.-H. Neeb,
``Infinite Dimensional Lie Groups,''
book in preparation.
%
%
\bibitem{Gra}
Grauert, H.,
\emph{Analytische Faserungen über holomorph-vollst\"{a}ndigen R\"{a}umen},
Math.\ Ann.\ {\bf 135} (1958), 263--273. 
%
%
\bibitem{Ham} Hamilton, R., \emph{The inverse function theorem
of Nash and Moser}, Bull.\ Amer.\ Math.\ Soc.\ {\bf 7}\,(1982),
65--222.
%
%
\bibitem{HaR}
Hewitt, E. and K.\,A. Ross, ``Abstract Harmonic Analysis I,''
Springer, New York, ${}^2$1979.
%
%
\bibitem{Hir}
Hirai, T., H. Shimomura, N. Tatsuuma, and E. Hirai,
\emph{Inductive limits of topologies, their direct products,
and problems related to algebraic structures},
J. Math.\ Kyoto Univ.\ {\bf 41} (2001), 475--505. 
%
%
\bibitem{HaM}
Hofmann, K.\,H. and S.\,A. Morris,
``The Structure of Compact Groups,''
de Gruyter, Berlin, 1998.
%
%
\bibitem{Jar}
Jarchow, H.,
``Locally Convex Spaces,''
B.\,G. Teubner, Stuttgart, 1981.
%
%
\bibitem{Kel} Keller, H.\,H., ``Differential Calculus
in Locally Convex Spaces,'' Springer, Berlin, 1974.
%
%
\bibitem{Kly}
Kelley, J.\,L., ``General Topology,''
Springer, New York, 1975.
%
%
\bibitem{KaM} Kriegl, A. and P.\,W. Michor, ``The Convenient
Setting of Global Analysis,''
AMS, Providence, 1997.
%
%
\bibitem{MMM}
Micheli, M., P.\,W. Michor, and D. Mumford,
\emph{Sobolev metrics on diffeomorphism groups and the derived geometry of spaces of submanifolds},
Izv.\ Math.\ {\bf 77} (2013),
541--570.
%
%
\bibitem{Mic} Michor, P.\,W., ``Manifolds of
Differentiable Mappings,'' Shiva, Orpington,
1980.
%
%
\bibitem{Mil} Milnor, J., \emph{Remarks on infinite-dimensional
Lie groups}, pp.\,1008--1057 in:
B. DeWitt and R. Stora (eds.), ``Relativity, Groups
and Topology~II,'' North Holland, 1984.
%
%
\bibitem{NSU} Neeb, K.-H., \emph{Towards a Lie theory
of locally convex groups}, Jpn.\ J.\ Math.\ {\bf 1} (2006), 291--468.
%
%
\bibitem{NST}
Neeb, K.-H., H. Sahlmann, and T. Thiemann,
\emph{Weak Poisson structures on infinite dimensional manifolds and Hamiltonian actions},
pp.\ 105--135 in:
V. Dobrev (ed.), ``Lie Theory and its Applications in Physics,''
Springer, Tokyo, 2014.
%
%
\bibitem{Nob}
Noble, N.,
\emph{The continuity of functions on cartesian products},
Trans.\ Amer.\ Math.\ Soc.\ {\bf 149} (1970), 187--198. 
%
%
\bibitem{Nb2}
Noble, N.,
\emph{$k$-groups and duality},
Trans.\ Amer.\ Math.\ Soc.\ {\bf 151} (1970), 551--561.
%
% 
\bibitem{OR1}
Odzijewicz, A. and T.\,S. Ratiu,
\emph{Banach Lie-Poisson spaces and reduction},
Comm.\ Math.\ Phys.\ {\bf 243} (2003), 1--54.
%
%
\bibitem{OR2}
Odzijewicz, A. and T.\,S. Ratiu,
\emph{Extensions of Banach Lie-Poisson spaces},
J. Funct.\ Anal.\ {\bf 217} (2004), 103--125.
%
%
\bibitem{Pat}
Patyi, I.,
\emph{On holomorphic Banach vector bundles over Banach spaces},
Math.\ Ann.\ {\bf 341} (2008), 455--482.
%
%
\bibitem{Sei}
Seip, U.,
``Kompakt erzeugte Vektorr\"{a}ume und Analysis,''
Springer, Berlin, 1972.
%
%
\bibitem{Tho}
Thomas, E.\,G.\,F.,
``Calculus on Locally Convex Spaces,''
preprint, University of Groningen, 1996.
%
%
\bibitem{Tre} Treves, F., ``Topological Vector Spaces,
Distributions and Kernels,'' Academic Press, New York, 1967.
%
%
\bibitem{Wur} Wurzbacher, T., {\em Fermionic second quantization
and the geometry of the restric\-{}ted Grassmannian},
pp.\,287--375 in: Huckleberry, A.\,T. and T. Wurzbacher (eds.),
``Infinite-Dimensional K\"{a}hler Manifolds,''
Birkh\"{a}user, Basel, 2001.
%
\end{thebibliography}}
\end{document}